\documentclass[12pt]{amsart}
\usepackage{amscd}     
\usepackage{amssymb}
\usepackage{amsmath, amsthm, graphics}
\allowdisplaybreaks[3]
\usepackage{amsfonts}
\usepackage[normalem]{ulem}
\usepackage{hyperref}
\usepackage{tikz}
\usepackage{mathtools}
\usepackage{verbatim}

\usepackage{latexsym}
\usepackage{epsfig}
\usepackage{amsfonts}
\usepackage{enumerate}
\usepackage{times}

\newtheorem{prop}{Proposition}
\newtheorem{lem}[prop]{Lemma}

\newtheorem{thm}[prop]{Theorem}

\newtheorem{theorem}{Theorem}[section]

\newtheorem{definition}[theorem]{Definition}

\theoremstyle{remark}

\theoremstyle{remark}

\numberwithin{equation}{section}

\newcommand{\Mbar}{\overline{\M}}
\newcommand{\proj}{\mathbb{P}}

\newcommand{\com}{\mathbb{C}}

\newcommand{\M}{\mathcal{M}}

\newcommand{\V}{\mathsf{V}}

\newcommand{\T}{\mathsf{T}}

\newcommand{\Hilb}{\mathsf{Hilb}^n(\mathbb{C}^2)}
\newcommand{\Sym}{\mathsf{Sym}^n}
\newcommand{\Id}{{\bf 1}}
\newcommand{\WWW}{{\mathcal{S}}}
\newcommand{\WW}{{\mathcal{R}}}

\newcommand{\hilbnc}{\mathsf{Hilb}^n(\mathbb{C}^2)}

\newcommand{\MM}{\mathsf{M}}

\newcommand{\cF}{\mathcal{F}}
\newcommand{\vac}{v_\emptyset}
\newcommand{\lv}{\left |}

\newcommand{\zz}{{\mathfrak{z}}}
\newcommand{\lang}{\left\langle}
\newcommand{\rang}{\right\rangle}
\newcommand{\blang}{\big\langle}
\newcommand{\brang}{\big\rangle}

\def\<{\left\langle}
\def\>{\right\rangle}

\newcommand{\RRR}{\mathsf{R}}
\newcommand{\oM}{\overline{\mathcal{M}}}
\newcommand{\RR}{\mathsf{A}}
\newcommand{\CC}{\mathbb{C}}
\newcommand{\Q}{\mathbb{Q}}
\newcommand{\ZZ}{\mathbb{Z}}
\newcommand{\cO}{\mathcal{O}}

\newcommand{\cI}{\mathcal{I}}
\DeclareMathOperator{\Aut}{Aut}
\DeclareMathOperator{\sgn}{sgn}

\def\ch{{\rm ch}}
\def\End{{\rm End}}
\def\b1{{\mathbf 1}}

\begin{document}

\title[Higher genus GW theory of $\Hilb$ and $\mathsf{CohFTs}$ associated to local curves]{Higher genus Gromov-Witten theory of $\Hilb$\\ and $\mathsf{CohFTs}$ associated to local curves}
\author[Pandharipande]{Rahul Pandharipande}
\address{Department of  Mathematics\\ ETH Z\"urich \\ R\"amistrasse 101 \\ 8092 Z\"urich\\ Switzerland}
\email{rahul@math.ethz.ch}

\author[Tseng]{Hsian-Hua Tseng}
\address{Department of Mathematics\\ Ohio State University\\ 100 Math Tower, 231 West 18th Ave. \\ Columbus,  OH 43210\\ USA}
\email{hhtseng@math.ohio-state.edu}

\date{May 2019}

\maketitle

{\centering \em Dedicated to A. Givental on the occasion of his 60th birthday\par}

\begin{abstract}
We study the higher genus equivariant Gromov-Witten theory of the
Hilbert scheme of $n$ points of $\mathbb{C}^2$. 
Since the equivariant quantum cohomology, computed in \cite{op},
is semisimple, the higher genus theory is determined by an
$\mathsf{R}$-matrix via the Givental-Teleman classification of 
Cohomological Field Theories
(CohFTs). We uniquely specify the required $\mathsf{R}$-matrix by explicit
data in degree $0$. As a consequence, we lift the basic triangle of
equivalences relating the equivariant
quantum cohomology of the Hilbert scheme $\Hilb$
and 
the Gromov-Witten/Donaldson-Thomas correspondence
for 3-fold theories of local curves to a triangle of
equivalences in all higher genera. 
The proof uses the analytic continuation of the
fundamental solution of the QDE of the Hilbert scheme of
points determined in \cite{op29}.
The GW/DT edge of the triangle in higher
genus concerns new CohFTs 
defined by varying the 3-fold local curve in the
moduli space of stable curves. 

The equivariant orbifold
Gromov-Witten theory of the symmetric product $\mathsf{Sym}^n(\mathbb{C}^2)$
is also shown to be equivalent to the theories of the
triangle in all genera. The result
establishes
a complete case of the crepant resolution conjecture \cite{bg,cit,cr}.
\end{abstract}

\setcounter{tocdepth}{1}
\tableofcontents

\setcounter{section}{-1}
\section{Introduction}
\subsection{Quantum cohomology}
The {\em Hilbert scheme} $\Hilb$
of $n$ points in the plane $\CC^2$ 
parameterizes ideals $\cI\subset \CC[x,y]$ of colength $n$,
$$
\dim_\CC {\CC[x,y]}/{\cI} = n \,. 
$$
The Hilbert scheme $\Hilb$ is a nonsingular, irreducible,
quasi-projective variety 
of  
dimension $2n$,
see
\cite{goe,Nak} for an introduction. 
An open dense set of $\Hilb$ parameterizes 
ideals associated to configurations of
$n$ distinct points.

The symmetries of $\CC^2$ lift to the Hilbert scheme. 
The algebraic torus 
$$\T=(\CC^*)^2$$ 
acts diagonally on $\CC^2$ by scaling coordinates,
$$
(z_1,z_2) \cdot (x,y) = (z_1 x, z_2 y)\, .
$$
The induced $\T$-action on $\Hilb$ will play basic
role here.

The Hilbert scheme  carries a 
tautological rank $n$ vector bundle, 
\begin{equation}\label{xx88}
\cO/\cI\rightarrow \Hilb\, ,
\end{equation}
with fiber
$\CC[x,y]/\cI$ over $[\cI]\in \Hilb$, see \cite{Lehn}.  
The $\T$-action on $\Hilb$ lifts canonically to the tautological bundle \eqref{xx88}.
Let $$
D = c_1 (\cO/\cI) \in H^2_T(\Hilb, {\mathbb Q})$$
be the $\T$-equivariant first Chern class.

The $\T$-equivariant quantum cohomology of $\Hilb$  has been determined  in
\cite{op}.
The {\em matrix elements} of the $\T$-equivariant quantum product 
 count{\footnote{The count is virtual.}} rational 
curves meeting three given 
subvarieties of $\Hilb$. 
The (non-negative) {\em degree} of an effective{\footnote{The $\beta=0$
is considered here effective.}} curve class 
$$\beta\in H_2(\Hilb,{\mathbb Z})$$ is defined by pairing with $D$,
 $$d=\int_\beta D.$$
Curves of degree $d$ are counted with 
weight $q^d$, where $q$ is the quantum parameter.
The ordinary multiplication in 
$\T$-equivariant cohomology is recovered by
setting $q=0$.

Let $\mathsf{M}_D$ denote the
operator of $\T$-equivariant quantum  multiplication by the 
divisor $D$. 
A central result of \cite{op} is an explicit formula for
 $\mathsf{M}_D$ an as operator on Fock space.

\subsection{Fock space formalism}\label{fsf}
We review the Fock space description of the $\T$-equivariant 
cohomology of the Hilbert scheme of points of $\CC^2$ following
the notation of \cite[Section 2.1]{op}, see also \cite{Gron, Nak}. 

By definition, the {\em Fock space} $\cF$ 
is freely generated over $\Q$ by commuting 
creation operators $\alpha_{-k}$, $k\in\ZZ_{>0}$,
acting on the vacuum vector $\vac$. The annihilation 
operators $\alpha_{k}$, $k\in\ZZ_{>0}$, kill the vacuum 
$$
\alpha_k \cdot \vac =0,\quad k>0 \,,
$$
and satisfy the commutation relations
$$
\left[\alpha_k,\alpha_l\right] = k \, \delta_{k+l}\,. 
$$

A natural basis of $\cF$ is given by 
the vectors  
\begin{equation}
  \label{basis}
  \lv \mu \rang = \frac{1}{\zz(\mu)} \, \prod_i \alpha_{-\mu_i} \, \vac \,
\end{equation}
indexed by partitions 
$\mu$. Here, $$\zz(\mu)=|\Aut(\mu)| \, \prod_i \mu_i$$ is the usual 
normalization factor. 
Let the length $\ell(\mu)$ denote the number of 
parts of the partition $\mu$.

The {\em Nakajima basis} defines a canonical isomorphism,
\begin{equation}
\cF \otimes _{\mathbb Q} {\mathbb Q}[t_1,t_2]\stackrel{\sim}{=} 
\bigoplus_{n\geq 0} H_{\T}^*(\Hilb,{\mathbb Q}).
\label{FockHilb}
\end{equation}
The Nakajima basis element corresponding to  
$\lv \mu \rang$  is
$$\frac{1}{\Pi_i \mu_i} [V_\mu]$$
where $[V_\mu]$ is (the cohomological dual of)
the class of the subvariety of $\mathsf{Hilb}^{|\mu|}(\CC^2)$
with generic element given by a union of 
schemes of lengths $$\mu_1, \ldots, \mu_{\ell(\mu)}$$ supported
at $\ell(\mu)$ distinct points{\footnote{The points and
parts of $\mu$ are considered
here to be unordered.}} of $\CC^2$. 
The vacuum vector $\vac$ corresponds to the unit in 
$$\mathsf{1}\in H_\T^*(\mathsf{Hilb}^0(\CC^2), {\mathbb Q})\, .$$
The variables $t_1$ and $t_2$ are the equivariant parameters corresponding
to the weights of the $\T$-action on the tangent
space 
$\text{Tan}_0(\CC^2)$ 
at the origin of $\CC^2$.

The subspace of $\cF\otimes_{\mathbb Q} {\mathbb Q}[t_1,t_2]$ corresponding to $H^*_\T(\Hilb,{\mathbb Q})$
is spanned by the vectors \eqref{basis} with $|\mu|=n$. The subspace can also 
be described as the $n$-eigenspace of the {\em  energy operator}: 
$$
|\cdot| = \sum_{k>0} \alpha_{-k} \, \alpha_k \,.
$$
The vector $\lv 1^n \rang$ corresponds to the unit 
$$\mathsf{1}\in H^*_\T(\Hilb,{\mathbb Q})\, .$$ 
A straightforward calculation shows 
\begin{equation}\label{DDDD}
D = - \lv 2,1^{n-2} \rang \,.
\end{equation}

The standard inner product on the $\T$-equivariant cohomology of
$\Hilb$ 
induces the following 
{\em nonstandard} inner product on Fock space after an extension of scalars:
\begin{equation}
  \label{inner_prod}
  \lang \mu | \nu \rang = 
\frac{(-1)^{|\mu|-\ell(\mu)}}{(t_1 t_2)^{\ell(\mu)}} 
\frac{\delta_{\mu\nu}}{\zz(\mu)} \,. 
\end{equation}
With respect to the inner product, 
\begin{equation}
  \label{adjoint}
  \left(\alpha_{k}\right)^* = (-1)^{k-1} (t_1 t_2)^{\sgn(k)} \, 
\alpha_{-k}\,.
\end{equation}

\subsection{Quantum multiplication by $D$}
\label{t1}

The formula of \cite{op} for the operator $\MM_D$ of  
quantum multiplication by $D$ is:
\begin{multline*}
\MM_D(q,t_1,t_2) = (t_1+t_2) \sum_{k>0} \frac{k}{2} \frac{(-q)^k+1}{(-q)^k-1} \,
 \alpha_{-k} \, \alpha_k\ \, -\, \frac{t_1+t_2}{2}\frac{(-q)+1}{(-q)-1}|\cdot |
 \\
 +  
\frac12 \sum_{k,l>0} 
\Big[t_1 t_2 \, \alpha_{k+l} \, \alpha_{-k} \, \alpha_{-l} -
 \alpha_{-k-l}\,  \alpha_{k} \, \alpha_{l} \Big] 
 \, .
\end{multline*}
The $q$-dependence of $\MM_D$ occurs only in the first two terms
(which acts diagonally in the basis \eqref{basis}). The two parts
of
the last sum are known respectively as the splitting and joining terms. 

Let $\mu^1$ and $\mu^2$ be partitions of $n$. The  
$\T$-equivariant Gromov-Witten invariants of  $\Hilb$ in genus $0$ 
with 3 cohomology insertions 
given (in the Nakajima basis) by $\mu^1$, $D$, and $\mu^2$ are determined by $\MM_D$:
$$\sum_{d=0}^\infty \blang \mu^1, D, \mu^2 \brang_{0,d}^{\Hilb}\, q^d\ =\ 
\big\langle \mu^1\, \big| \, \MM_D  \, \big|\,  \mu^2 \big\rangle\, .$$
Equivalently, denoting the 2-cycle $(2,1^n)$ by $(2)$, we have
\begin{equation} \label{v233}
\sum_{d=0}^\infty \blang \mu^1, (2), \mu^2 \brang_{0,d}^{\Hilb}\, q^d\ =\ 
\big\langle \mu^1\, \big| \, -\MM_D  \, \big|\,  \mu^2 \big\rangle\, .
\end{equation}


Let $\mu^1,\ldots, \mu^r\in \text{Part}(n)$. 
The $\T$-equivariant Gromov-Witten series in genus $g$,
$$\blang \mu^1, \mu^2, \ldots, \mu^r \brang_{g}^{\Hilb}
\, =\,  \sum_{d=0}^\infty 
\blang \mu^1, \mu^2, \ldots, \mu^r \brang_{g,d}^{\Hilb} q^d
\, \in \mathbb{Q}[[q]]\, ,$$
is a sum over the degree $d$ with variable $q$.
The $\T$-equivariant Gromov-Witten series in genus 0,
$$\blang \mu^1, \mu^2, \ldots, \mu^r \brang_{0}^{\Hilb}
= \sum_{d=0}^\infty 
\blang \mu^1, \mu^2, \ldots, \mu^r \brang_{0,d}^{\Hilb} q^d\, ,
$$
can be calculated from the special 3-point
invariants \eqref{v233}, see  \cite[Section 4.2]{op}.

\subsection{Higher genus}
Our first result here is a determination of the $\T$-equivariant Gromov-Witten
theory of $\mathsf{Hilb}(\mathbb{C}^2,d)$ in {\em all} higher genera $g$.
We will
use the Givental-Teleman classification of semisimple Cohomological
Field Theories (CohFTs).  The Frobenius structure determined
by the $\T$-equivariant genus $0$ theory of $\mathsf{Hilb}(\mathbb{C}^2,d)$ 
is semisimple, but {\em not}
conformal. Therefore, the $\mathsf{R}$-matrix is {\em not} determined by the
$\T$-equivariant genus $0$ theory alone. 
Fortunately,  together with the divisor equation, an evaluation of the
$\T$-equivariant higher genus theory in degree $0$ is enough
to uniquely determine the $\mathsf{R}$-matrix. 

Let
$\text{Part}(n)$ be the set of of partitions
of $n$ corresponding to the $\T$-fixed
points of $\Hilb$. For each $\eta\in \text{Part}(n)$,  let
 $\text{Tan}_\eta(\Hilb)$ be the $\T$-representation on
 the tangent space at the $\T$-fixed point
 corresponding to $\eta$.
Let
$$\mathbb{E}_g\rightarrow \oM_g$$
be the Hodge bundle of differential forms over the moduli space of
stable curves of genus $g$.

\begin{thm}\label{ff11} The $\mathsf{R}$-matrix 
of the $\T$-equivariant Gromov-Witten theory of $\Hilb$ is uniquely determined from the $\T$-equivariant genus 0 theory by the divisor equation
and the degree 0 invariants
\begin{eqnarray*}
\big\langle \mu \big\rangle_{1,0}^{\Hilb} &=& \sum_{\eta\in \text{\em Part}(n)} \mu|_\eta \int_{\oM_{1,1}}
\frac{e\left(\mathbb{E}_1^* \otimes \text{\em Tan}_\eta(\Hilb)\right)}
{e\left(\text{\em Tan}_\eta(\Hilb)\right)
}\, , 
\\
\big\langle \, \big\rangle_{g\geq 2,0}^{\Hilb} &=& \sum_{\eta\in \text{\em Part}(n)} \int_{\oM_g}
\frac{e\left(\mathbb{E}_g^* \otimes \text{\em Tan}_\eta(\Hilb)\right)}
{e\left(\text{\em Tan}_\eta(\Hilb)\right)
}\,
.
\end{eqnarray*}
\end{thm}

The insertion $\mu$ in the genus 1 invariant
is in the Nakajima basis, and 
$\mu|_\eta$ denotes the restriction to
the $\T$-fixed point corresponding
to $\eta$ --- which can be calculated
from the 
Jack polynomial $\mathsf{J}^\eta$.
The integral of the Euler class $e$ in the formula
 may be explicitly expressed in terms of
Hodge integrals and the tangent weights of the $\T$-representation
$\text{Tan}_\eta(\Hilb)$.

Apart from $\big\langle \mu \big\rangle_{1,0}^{\Hilb}$ and $\big\langle \, \big\rangle_{g\geq 2,0}^{\Hilb}\,$ , all other degree 0 invariants 
of $\Hilb$ in positive genus vanish. Theorem \ref{ff11} can be equivalently stated in the following
form: {\em the $\mathsf{R}$-matrix 
of the $\T$-equivariant Gromov-Witten theory of $\Hilb$ is uniquely determined from the $\T$-equivariant genus 0 theory by the divisor equation
and the degree 0 invariants in positive genus}.

\subsection{Lifting the triangle of correspondences}\label{lifttr}
The calculation of the $\T$-equivariant quantum cohomology of $\Hilb$
is a basic step in the proof  \cite{bp,op,op2} of
the following triangle of equivalences:
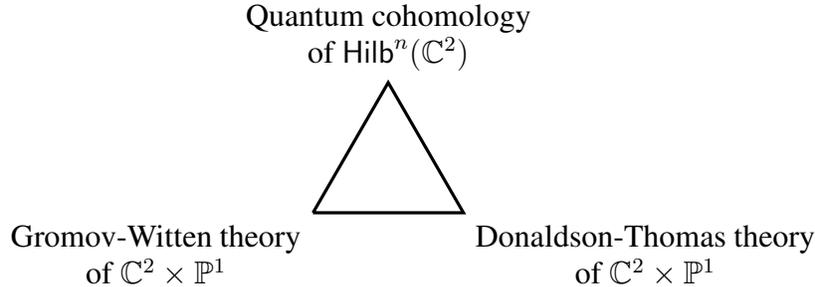
\begin{center}
\begin{tikzpicture}
 \draw[very thick]
   (0,0) node[below left, align = center]{Gromov-Witten theory\\of $\mathbb{C}^2 \times 
\mathbb{P}^1$}
-- (1,1.732) node[above, align = center]{Quantum cohomology\\of 
$\Hilb$}
-- (2,0) node[below right, align = center]{Donaldson-Thomas theory\\of $\mathbb{C}^2 \times 
\mathbb{P}^1$}
-- (0,0);
\end{tikzpicture}
\end{center}

Our second result is a lifting of the above triangle to all higher genera.
The top vertex is replaced by the {\em 
$\T$-equivariant Gromov-Witten theory of $\Hilb$
in genus $g$ with $r$ insertions}.
The bottom vertices of the triangle are {\em new} theories which are constructed here.

Let $\overline{\mathcal{M}}_{g,r}$ be the moduli space of Deligne-Mumford stable
curves of genus $g$ with $r$ markings.{\footnote{We will always assume 
$g$ and $r$ satisfy the {\em stability} condition $2g-2+r>0$.}}
Let
$$ \mathcal{C} \rightarrow \overline{\mathcal{M}}_{g,r}\, $$
be the universal curve with sections
$$\mathsf{p}_1, \ldots, \mathsf{p}_r : \overline{\mathcal{M}}_{g,r}\rightarrow
\mathcal{C}$$
associated to the markings.
Let 
$$\pi: \mathbb{C}^2\times \mathcal{C} \rightarrow \overline{\mathcal{M}}_{g,r}$$
be the universal {\em local curve} over $\overline{\mathcal{M}}_{g,r}$. The torus $\T$ acts on
the $\com^2$ factor.
The Gromov-Witten and Donaldson-Thomas theories  
of the morphism $\pi$ are defined by the $\pi$-relative $\T$-equivariant virtual class
of the universal $\pi$-relative
 moduli spaces of stable maps and stable pairs.{\footnote{While the triangle
 of equivalences originally included the Donaldson-Thomas theory of
 ideal sheaves, the theory of stable pairs \cite{PT} is much better behaved, see
 \cite[Section 5]{mpt} for a discussion valid for $\com^2 \times \proj^1$.
 We will use the theory of stable pairs here.}}

\begin{thm} \label{tt22}
For all genera $g\geq 0$, there is a triangle of 
equivalences of $\T$-equivariant theories:
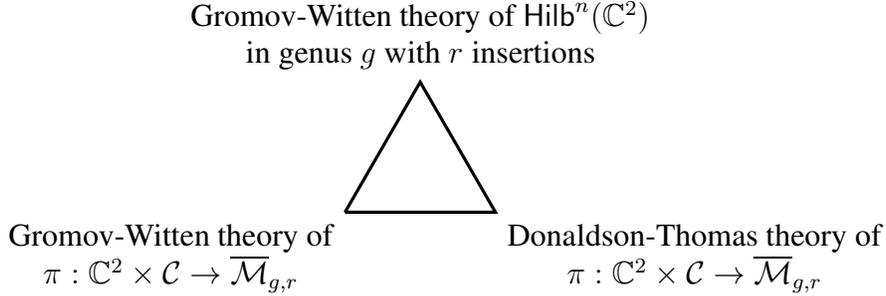
\begin{center}
\begin{tikzpicture}
 \draw[very thick]
   (0,0) node[below left, align = center]{\em Gromov-Witten theory of\\ $\pi:\mathbb{C}^2 \times 
\mathcal{C} \rightarrow
\overline{\mathcal{M}}_{g,r}$}
-- (1,1.732) node[above, align = center]{\em Gromov-Witten theory of $\Hilb$\\
\em in genus $g$ with $r$ insertions}
-- (2,0) node[below right, align = center]{\em Donaldson-Thomas theory of\\ $\pi: \mathbb{C}^2 \times 
\mathcal{C} \rightarrow\overline{\mathcal{M}}_{g,r}$ }
-- (0,0);
\end{tikzpicture}
\end{center}
\end{thm}

The triangle of Theorem \ref{tt22} may be viewed from different
perspectives.
First, all three vertices define CohFTs. Theorem \ref{tt22} may be
stated as simply an isomorphism of the three CohFTs.
A second point of view of the
bottom side of the triangle of Theorem \ref{tt22} is
as a GW/DT correspondence in {\em families of 3-folds} as
the complex structure of the local curve varies.
While a general GW/DT correspondence for families of 3-folds can be
 naturally formulated, there are very few
interesting{\footnote{The equivariant GW/DT correspondence
is a special case of the GW/DT correspondence for
families (and is well studied).}} cases studied.

For a pointed curve of fixed complex structure
$(C,p_1,\ldots,p_r)$, 
the
triangle of Theorem \ref{tt22} is a basic result of
the papers \cite{bp,op,op2} since
$C$ can be degenerated to a curve with
all irreducible components of genus 0.

\subsection{Crepant resolution}
The Hilbert scheme of points of $\mathbb{C}^2$  is well-known to be a crepant resolution
of the symmetric product,
$$\epsilon:\Hilb \ \rightarrow\ {\Sym}(\mathbb{C}^2)= (\mathbb{C}^2)^n/S_n, .$$
Viewed as an {\em orbifold}, the symmetric product $\Sym(\mathbb{C}^2)$ has
a $\T$-equivariant
Gromov-Witten theory with insertions indexed by partitions of $n$. 
The $\T$-equivariant Gromov-Witten generating series{\footnote{A full
discussion of the definition
appears in Section \ref{diveqq}.}},
$$\blang \mu^1, \mu^2, \ldots, \mu^r \brang_{g}^{{\Sym}(\mathbb{C}^2)}
\, =\,  \sum_{b=0}^\infty 
\blang \mu^1, \mu^2, \ldots, \mu^r \brang_{g,b}^{{\Sym}(\mathbb{C}^2)} u^b
\, \in \mathbb{Q}[[u]]\, ,$$
is a sum over the number of
free ramification points $b$ with variable $u$.

The Frobenius structure determined
by the $\T$-equivariant genus $0$ theory of ${{\Sym}(\mathbb{C}^2)}$ 
is semisimple, but {\em not}
conformal. Again, the $\mathsf{R}$-matrix is {\em not} determined by the
$\T$-equivariant genus $0$ theory alone. 
The determination of the $\RRR$-matrix of 
$\Sym(\mathbb{C}^2)$
 is given by the following result parallel to Theorem \ref{ff11}.

\begin{thm}\label{4444} The $\mathsf{R}$-matrix 
of the $\T$-equivariant Gromov-Witten theory of $\Sym(\mathbb{C}^2)$ is uniquely determined from the $\T$-equivariant genus 0 theory by the divisor equation
and all the unramified invariants, 
$$\big\langle \mu^1, \ldots, \mu^r \big\rangle^{\Sym(\mathbb{C}^2)}_{g,0}\, ,$$
in positive genus.
\end{thm}

For each $\eta \in \text{Part}(n)$,
Let $\mathcal{H}^\eta_g$ be the moduli
space \'etale covers with $\ell(\eta)$
connected components{\footnote{For
the definition of $\overline{\mathcal{H}}_g^\eta$,
we consider the parts of $\eta$ to
be {\em ordered} and in 
bijection with the components
of the cover.}} of
degrees
$$\eta_1, \ldots, \eta_{\ell(\eta)}$$
of a nonsingular genus $g\geq 2$ curve.
The degree of the map
$$\mathcal{H}^\eta_g \rightarrow \mathcal{M}_g$$
is an unramified Hurwitz number.
Let 
$$\overline{\mathcal{H}}^\eta_g \rightarrow \oM_g$$ 
be the compactification by 
admissible covers. 

Since
the genus of  the component  
of the cover corresponding to the
part $\eta_i$ is
$\eta_i(g-1)+1$, there is a
Hodge bundle
$$\mathbb{E}^*_{\eta_i(g-1)+1} \rightarrow
\overline{\mathcal{H}}_g^\eta$$
obtained via the corresponding component.

The simplest unramified invariants required in Theorem \ref{4444}
are
$$\big\langle \, \big\rangle_{g\geq 2,0} ^
{\Sym(\mathbb{C}^2)}= \sum_{\eta\in \text{Part}(n)} 
\frac{1}{|\text{Aut}(\eta)|}\,
\int_{\overline{\mathcal{H}}_g^{\eta}}\,
\prod_{i=1}^{\ell(\eta)}\,
\frac{e\big( \mathbb{E}^*_{\eta_i(g-1)+1} \otimes
\text{Tan}_0(\mathbb{C}^2)\big)}
{e\left(\text{Tan}_0(\mathbb{C}^2)\right)}\, .
$$
However, there are many further unramified invariants in positive genus
obtained by including insertions. Unlike the case of 
$\Hilb$, the unramified invariants with insertions for
$\Sym(\mathbb{C}^2)$ do not all vanish.

In genus $0$, the equivalence of the $\T$-equivariant
Gromov-Witten theories of $\Hilb$ and the orbifold ${\Sym}(\mathbb{C}^2)$
was proven{\footnote{The
prefactor
$(-i)^{\sum_{i=1}^r \ell(\mu^i)-|\mu^i|}$
was treated incorrectly in
\cite{bg} because of an
arthimetical error. The
prefactor here is correct.}}
in \cite{bg}. Our fourth result is a proof of
the equivalence for all genera.

\begin{thm}\label{crc} For all genera $g\geq 0$ and 
$\mu^1,\mu^2,\ldots,\mu^r\in \text{\em Part}(n)$, we have  
$$\blang \mu^1, \mu^2, \ldots, \mu^r \brang_{g}^{\Hilb} =
(-i)^{\sum_{i=1}^r \ell(\mu^i)-|\mu^i|}\blang \mu^1, \mu^2, \ldots, \mu^r \brang_{g}^{{\Sym}(\mathbb{C}^2)}$$
 after the variable change $-q=e^{iu}$.
\end{thm}

The variable change of Theorem \ref{crc} is well-defined by the following 
result (which was previously proven in genus 0 
in \cite{op}).

\begin{thm}\label{crcr} For all genera $g\geq 0$ and 
$\mu^1,\mu^2,\ldots,\mu^r\in \text{\em Part}(n)$, 
the series{\footnote{As always,
$g$ and $r$ 
are required
to be in the stable range $2g-2+r>0$.}}
$$\blang \mu^1, \mu^2, \ldots, \mu^r \brang_{g}^{\Hilb}\in \mathbb{Q}(t_1,t_2)[[q]]$$
is the Taylor expansion in $q$ of a rational
function in $\mathbb{Q}(t_1,t_1,q)$.
\end{thm}

Calculations in closed form
in higher genus are
not easily obtained. The first
nontrivial example occurs  in genus 1
for the Hilbert scheme of 2 points:
\begin{equation}\label{dxx12}
\big\langle \, (2)\,  \big\rangle_{1}^{\mathsf{Hilb}^2(\mathbb{C}^2)} = -\frac{1}{24}\frac{(t_1+t_2)^2}{t_1t_2} \cdot \frac{1+q}{1-q}\, .
\end{equation}
While the formula is simple, our virtual
localization \cite{grpan} calculation of the integral is
rather long. Since
\eqref{dxx12} captures all degrees, the full graph
sum must be controlled --
we use the
valuation at $(t_1+t_2)$ 
as
in \cite{op}. 

The above calculation \eqref{dxx12} for the
Hilbert scheme of 2 points 
yields new Hodge integral
calculations for the bielliptic
locus $\overline{\mathcal{H}}_1\left((2)^{2n}\right)$, the moduli space of double covers of elliptic curves with
$2n$ ordered branch points,
$$\overline{\mathcal{H}}_1
\left({(2)^{2n}}\right) \rightarrow 
\overline{\mathcal{M}}_{1,2n}\, .$$ 
Since the domain curves parameterized by 
$\overline{\mathcal{H}}_1
\left({(2)^{2n}}\right)$
have genus $n+1$, there is a Hodge bundle
$$\mathbb{E}^*_{n+1} \rightarrow
\overline{\mathcal{H}}_1
\left({(2)^{2n}}\right)\, .$$
By a direct application{\footnote{We follow
the standard convention $\lambda_i= c_i(\mathbb{E}_{n+1})$ .}} of Theorem \ref{crc},
\begin{eqnarray*}
\sum_{n=1}^\infty \frac{u^{2n-1}} {(2n-1)!}  \int_{\overline{\mathcal{H}}_1
\left({(2)^{2n}}\right)} \lambda_{n+1}\lambda_{n-1}
&=&
\frac{i}{24} \cdot \frac{1- e^{iu}}{1+e^{iu}} \\
& = & 
\frac{1}{48} u + \frac{1}{576} u^3 + \frac{1}{5760} u^5  + \ldots \, .
 \end{eqnarray*}
The $u$ and $u^3$ coefficients
can be checked geometrically using the
bielliptic calculations of \cite{pagfab}. 
The $u^5$ coefficient has been checked in \cite[Remark 5.14]{joh}.

The corresponding series
for higher $n$, 
\begin{equation*}
\big\langle \, (2,1^{n-2})\,  \big\rangle_{1}^{\mathsf{Hilb}^n(\mathbb{C}^2)}
\in \mathbb{Q}(t_1,t_2,q)\, ,
\end{equation*}
very likely has a simple closed formula.
We will return to these questions in a
future paper. 

\subsection{Plan of proof}
Theorems \ref{tt22} and \ref{crc} are
proven together by studying the $\mathsf{R}$-matrices of all four theories.
The $\mathsf{R}$-matrix of 
the CohFT associated to the local Donaldson-Thomas
theories of curves is easily proven
to coincide with the $\mathsf{R}$-matrix
of the $\T$-equivariant Gromov-Witten theory
of $\Hilb$ determined in Theorem \ref{ff11}.
Similarly, the $\mathsf{R}$-matrices of
the CohFTs associated to $\Sym(\mathbb{C}^2)$ and
the local Gromov-Witten theories of curves
are straightforward to match (with determination by Theorem \ref{4444}). The above results
require a detailed study of
the divisor equation in the
four cases.

The main step of the joint proof of
Theorems \ref{tt22} and \ref{crc} is to
 match the $\mathsf{R}$-matrices
of Theorems \ref{ff11} and \ref{4444}.
The matter is non-trivial since the former
is a function of $q$ and the latter is 
a function of $u$. The matching after
$$-q=e^{iu}$$
requires an analytic continuation.
Our method here is to express the $\mathsf{R}$-matrices in terms of the
solution of the QDE associated to
$\Hilb$. Fortunately, the analytic continuation of
the solution of the QDE proven in 
\cite{op29} is exactly what is needed,
after a careful study of asymptotic expansions,
to match the two $\mathsf{R}$-matrices.

The addition of $\Sym(\com^2)$ via 
Theorem \ref{crc} to the triangle
of Theorem \ref{tt22}
yields a tetrahedron of
equivalences of $\T$-equivariant
theories (as first formulated in genus 0 in 
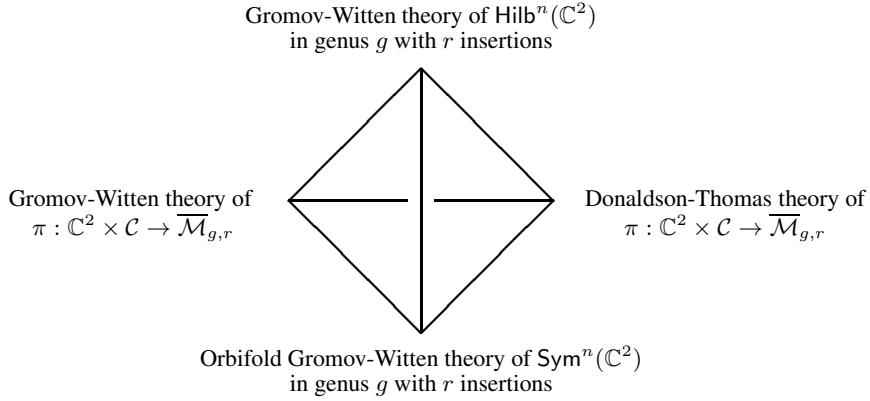
\cite{bg}).

\begin{center}
\scriptsize
\begin{picture}(200,175)(-30,-50)
\thicklines
\put(25,25){\line(1,1){50}}
\put(25,25){\line(1,-1){50}}
\put(125,25){\line(-1,1){50}}
\put(125,25){\line(-1,-1){50}}
\put(75,-25){\line(0,1){100}}
\put(25,25){\line(1,0){45}}
\put(80,25){\line(1,0){45}}
\put(75,95){\makebox(0,0){Gromov-Witten theory of $\Hilb$}}
\put(75,85){\makebox(0,0){in genus $g$
with $r$ insertions}}
\put(75,-35){\makebox(0,0){Orbifold
Gromov-Witten theory of $\Sym(\com^2)$}}
\put(75,-45){\makebox(0,0){in genus $g$ 
with $r$ insertions}}
\put(190,25){\makebox(0,0){Donaldson-Thomas theory of}}
\put(190,15){\makebox(0,0){
$\pi:\mathbb{C}^2 \times 
\mathcal{C} \rightarrow
\overline{\mathcal{M}}_{g,r}$
}}
\put(-35,25){\makebox(0,0){Gromov-Witten theory of}}
\put(-35,15){\makebox(0,0){
$\pi:\mathbb{C}^2 \times 
\mathcal{C} \rightarrow
\overline{\mathcal{M}}_{g,r}
$}}
\end{picture}
\end{center}

\normalsize

\vspace{20pt}

Before the development of the orbifold Gromov-Witten
theory of $\Sym(\com^2)$, a theory of Hurwitz-Hodge integrals
was proposed
by Cavalieri \cite{Cav}.  While the orbifold Gromov-Witten
theory is formulated
in terms of principal $S_n$-bundles over curves,
Cavalieri's theory is formulated in terms of
the associated Hurwitz covers of curves.
In fact, the virtual class of the orbifold theory of
$\Sym(\com^2)$
exactly coincides with the Hodge integrand proposed
by Cavalieri, so the two theories are {\em equal}. 
The orbifold vertex of the above
tetrahedron may therefore also be viewed via Cavalieri's definition.{\footnote{Our discussion concerns
Cavalieri's level $(0,0)$ theory extended naturally
over the moduli space of genus $g$ curves. In genus $0$, Cavalieri \cite{Cav} noted
the equivalence of his theory to the Gromov-Witten vertex.
He also defined theories of other levels $(a,b)$ which do not
exactly agree with the corresponding $(a,b)$-theories 
of the Gromov-Witten vertex (even in genus $0$). We do not
explore here the interesting geometry of
the $(a,b)$-level structure over the moduli
space of genus $g$ curves.
}}

\subsection{Acknowledgments} 
We are grateful to J.~Bryan, A.~Buryak, R. Cavalieri, T.~Graber, J. Gu\'er\'e, F.~Janda, 
Y.-P.~Lee, H. Lho,
D.~Maulik, A.~Oblomkov, A.~Okounkov, D. Oprea,
A.~Pixton,
Y. Ruan,
R.~Thomas, and D.~Zvonkine for many
conversations over the years related to the Gromov-Witten theory of
$\Hilb$, $3$-fold theories of local curves, and CohFTs.
We thank A. Givental and C. Teleman for discussions about semisimple CohFTs
and O. Costin and S. Gautam for discussions about asymptotic
expansions.

R.~P. was partially supported by 
 SNF-200020-162928, 
 SNF-200020-182181,
 ERC-2012-AdG-320368-MCSK, 
 ERC-2017-AdG-786580-MACI,
 SwissMAP, and
the Einstein Stiftung. 
H.-H. ~T. was partially supported by a Simons foundation collaboration grant and NSF grant DMS-1506551.
The research presented here was
furthered during visits of H.-H. T. to ETH Z\"urich in March 2016 and March 2017. 
The project has received funding from the European Research Council (ERC)
under the European Union Horizon 2020 Research and Innovation Program (grant No. 786580).

\section{Cohomological field theory}\label{ccfftt}

\subsection{Definitions}
The notion of a {cohomological field theory} (CohFT) was introduced in 
\cite{km}, see also \cite{m}. We follow 
closely here the treatment of \cite[Section 0.5]{ppz}, and we refer the
reader also to the survey \cite{Picm}.

Let $k$ be an algebraically closed field of characteristic $0$. Let $\RR$ be a commutative $k$-algebra. Let 
$\V$ be a free $\RR$-module of finite rank, let $$\eta: \V\otimes \V\to \RR$$ 
be an even, symmetric, non-degenerate{\footnote{The pairing
$\eta$ induces an isomorphism between $\V$ and the dual module
$\V^*= \text{Hom}(\V,\RR)$.}} pairing, and let
$\mathbf{1} \in V$ be a distinguished element.
The data $(\V,\eta, \b1)$ is the starting point for defining 
a cohomological field theory.
Given a 
basis $\{e_i\}$ of $\V$, we write the 
symmetric form as a matrix
$$\eta_{jk}=\eta(e_j,e_k) \ .$$ The inverse matrix is denoted by $\eta^{jk}$ as usual.

A {cohomological field theory} consists of 
a system $\Omega = (\Omega_{g,r})_{2g-2+r > 0}$ of elements 
$$
\Omega_{g,r} \in H^*(\oM_{g,r},\RR) \otimes (\V^*)^{\otimes r}.
$$
We view $\Omega_{g,r}$ as associating a cohomology class on $\oM_{g,r}$ 
to elements of $\V$ assigned to the $r$ markings.
The CohFT axioms imposed on $\Omega$ are:

\begin{enumerate}
\item[(i)] Each $\Omega_{g,r}$ is $S_r$-invariant, where the action of the symmetric group $S_r$ permutes both the marked points of $\oM_{g,r}$ and the
copies of $\V^*$.
\item[(ii)] Denote the basic gluing maps by
$$
q : \oM_{g-1, r+2} \to \oM_{g,r}\ ,
$$
$$
\widetilde{q}: \oM_{g_1, r_1+1} \times \oM_{g_2, r_2+1} \to \oM_{g,r}\ .
$$
The pull-backs $q^*(\Omega_{g,r})$ and 
$\widetilde{q}^*(\Omega_{g,r})$ are
 equal to the contractions of $$\Omega_{g-1,r+2}\ \ \text{and}\ \  
\Omega_{g_1, r_1+1} \otimes \Omega_{g_2, r_2+1}$$ by the bi-vector 
$$\sum_{j,k} \eta^{jk} e_j \otimes e_k$$ inserted at the two identified points. 
\item[(iii)] Let $v_1, \dots, v_r \in \V$  and let $p: \oM_{g,r+1} \to \oM_{g,r}$ be the forgetful map. We require 
$$
\Omega_{g,r+1}(v_1 \otimes \cdots \otimes v_r \otimes \b1) = p^*\Omega_{g,r} (v_1 \otimes \cdots \otimes v_r)\ ,
$$
$$\Omega_{0,3}(v_1\otimes v_2 \otimes \b1) = \eta(v_1,v_2)\ .$$
\end{enumerate}

\begin{definition}\label{defcohft}
A system $\Omega= (\Omega_{g,r})_{2g-2+r>0}$ of elements 
$$
\Omega_{g,r} \in H^*(\oM_{g,r},\RR) \otimes (V^*)^{\otimes r}
$$
satisfying properties~(i) and (ii) is a {\em cohomological field theory} ({\em CohFT}). If (iii) is also satisfied, $\Omega$ is 
a {\em CohFT with unit}.
\end{definition}

A CohFT $\Omega$ yields a {\em quantum product} $\star$ on $\V$ via 
\begin{equation}\label{gvv44}
\eta(v_1 \star v_2, v_3) = \Omega_{0,3}(v_1 \otimes v_2 \otimes v_3)\ .
\end{equation}
Associativity of $\star$ follows from (ii). The element
$\b1\in \V$ is the identity for $\star$ by (iii).

A CohFT $\omega$ composed only of degree~0 classes,
$$\omega_{g,r} \in H^0(\oM_{g,r},\RR) \otimes (\V^*)^{\otimes r}\ ,$$
 is called a {\em topological field theory}. 
Via property (ii), $\omega_{g,r}(v_1, \dots, v_r)$ 
is determined by considering stable curves with a maximal number of nodes. 
Every irreducible  
component of such a curve is 
of genus 0 with 3 special points.
The value of $\omega_{g,r}(v_1 \otimes \cdots \otimes v_r)$ is 
thus uniquely specified by the values of $\omega_{0,3}$ and by the pairing $\eta$. In other words, given $\V$ and $\eta$, a topological field theory is uniquely determined by the associated quantum product.

\subsection{Gromov-Witten theory}

Let $X$ be a nonsingular projective variety over $\mathbb{C}$. 
The stack  of stable maps{\footnote{See
\cite{FulP} for an introduction to stable maps.}} 
$\Mbar_{g,r}(X, \beta)$ of genus $g$
curves to $X$ representing the class $\beta\in H_2(X,\mathbb{Z})$
admits  an evaluation 
$$\text{ev}: \Mbar_{g,r}(X, \beta)\to X^r$$ and a forgetful map 
$$\rho: \Mbar_{g,r}(X, \beta)\to \Mbar_{g,r}\, ,$$
when $2g-2+r>0$. 
The {\em Gromov-Witten} CohFT is constructed from
the virtual classes of the moduli of stable maps of $X$, 
$$\Omega_{g,r}^{\mathsf{GW}}(v_1\otimes...\otimes v_r)
=\sum_{\beta\in H_2(X, \mathbb{Z})}q^\beta \rho_*\Big(\text{ev}^*(v_1\otimes...\otimes v_r)\cap [\Mbar_{g,r}(X, \beta)]^{vir}\Big). $$
The ground ring $\RR$ is the {\em Novikov ring}, a suitable completion of the group ring{\footnote{We take the Novikov ring with $\mathbb{Q}$-coefficients.}} associated to the semi-group of effective curve classes in $X$,
$$\V= H^*(X,\mathbb{Q}) \otimes_{\mathbb{Q}} \RR =
H^*(X,\RR)\, ,$$ 
 the pairing $\eta$ is the extension of 
the Poincar\'e pairing, and
$\mathbf{1}\in \V$ is the unit in cohomology.

\subsection{Semisimplicity}
\subsubsection{Classification}
Let $\Omega$ be a CohFT with respect to $(\V,\eta,\mathbf{1})$.
We are concerned here with theories for which the algebra $\V$ with respect
to the quantum product $\star$ defined by \eqref{gvv44} is semisimple. 
Such theories are classified in \cite{t}. 
Specifically, $\Omega$ is obtained from the algebra $\V$ via the action of an 
$\RRR$-matrix $$\RRR \in {\bf 1} + z \cdot \End(\V) [[z]],$$ satisfying the symplectic 
condition $$\RRR(z)\RRR^{\star}(-z)=\bf 1.$$ Here, $\RRR^{\star}$ denotes the 
adjoint with respect to $\eta$ and $\bf 1$ is the identity matrix. 
The explicit reconstruction of the semisimple CohFT from the $\RRR$-matrix action will 
be explained below, following \cite {ppz}.

\subsubsection {Actions on CohFTs}

Let $\Omega=(\Omega_{g, r})$ be a CohFT with respect to  $(\V, \eta,\mathbf{1})$. 
Fix a symplectic matrix 
$$\RRR\in {\bf 1}+z\cdot \text{End }(\V)[[z]]$$ as above. 
A new CohFT with respect to $(\V, \eta,\mathbf{1})$ is obtained via 
the cohomology elements $$\RRR \Omega=(\RRR\Omega)_{g, r},$$  defined as sums over stable graphs $\Gamma$ of genus $g$ with $r$ legs, with contributions coming from vertices, legs, and
edges. Specifically, \begin{equation}\label{romega}(\RRR\Omega)_{g, r}=\sum_{\Gamma} \frac{1}{|\text{Aut }(\Gamma)|} (\iota_{\Gamma})_{\star} \left(\prod_{v}\mathsf {Cont}(v)
\prod_{l}\mathsf {Cont}(l)
\prod_{e}\mathsf {Cont }(e) \right)\end{equation} where:
\begin {itemize}
\item [(i)] the vertex contribution is $$\mathsf {Cont}(v)=\Omega_{g(v), r(v)},$$ with $g(v)$ and $r(v)$ denoting the genus and number of half-edges and legs of the vertex,
\item [(ii)] the leg contribution is $$\mathsf{Cont}(l)=\RRR(\psi_l)$$ where $\psi_l$ is the cotangent class at the marking corresponding to the leg, 
\item [(iii)] the edge contribution is $${\mathsf {Cont}}(e)=\frac{\eta^{-1}-\RRR(\psi'_e)\eta^{-1} \RRR(\psi''_e)^{\top}}{\psi'_e+\psi''_e}.$$ Here $\psi'_e$ and $\psi''_e$ are the cotangent classes at the node which represents the edge $e$. The symplectic condition guarantees that the edge contribution is well-defined. 
\end {itemize}


A second action on CohFTs is given by translations. As before, let $\Omega$
be a CohFT with respect to $(\V, 1, \eta)$  
 and consider a power series $\T\in \V[[z]]$ with no terms of degrees $0$ and $1$: $$\T(z)=\T_2z^2+\T_3z^3+\ldots,\,\,\,\T_k\in \V.$$ A new CohFT wit respect to
  $(\V, 1, \eta)$, denoted $\T\Omega$, is defined by setting \begin{equation}\label{translation} (\T\Omega)_{g, r} (v_1\otimes \ldots \otimes v_r)=\sum_{m=0}^{\infty} \frac{1}{m!} (p_m)_{\star} \Omega_{g, r+m}(v_1\otimes \ldots\otimes v_r \otimes \T(\psi_{r+1})\otimes \ldots \otimes \T(\psi_{r+m}))\end{equation} where $$p_{m}:\oM_{g, r+m}\to \oM_{g, r}$$ is the forgetful morphism.

 \subsubsection{Reconstruction} With the above terminology understood, we can state the Givental-Teleman classification theorem \cite {t}. Fix $\Omega$ a semisimple CohFT with respect to $(\V, 1, \eta)$, and write $\omega$ for the degree $0$ topological part of the theory. 
Given any symplectic matrix $$\RRR\in {\mathbf 1}+z\cdot \End (\V)[[z]]$$ as above, we form a power series $\T$ by plugging $1 \in \V$ into $\RRR$, removing the free term, and multiplying by~$z$:
$$\T(z)=z(1-\RRR(1))\in \V[[z]].$$ 

\vspace{8pt}
\noindent {\bf Givental-Teleman classification.} 
{\em There exists a unique symplectic matrix $\RRR$ for which}  
 $$\Omega=\RRR \T\omega\, .$$

\noindent A proof of uniqueness can be found, for example, in \cite[Lemma 2.2]{MOPPZ}.

\subsection{Targets $\Hilb$ and $\Sym(\mathbb{C}^2)$} \label{tarhs}
The CohFTs determined by the
$\T$-equivariant Gromov-Witten theories of
$\Hilb$ and $\Sym(\mathbb{C}^2)$ are based on the 
algebras
\begin{equation}\label{eqn:ground_rings}
    \mathsf{A}= \mathbb{Q}(t_1,t_2)[[q]]\, , \ \ \ 
\widetilde{\mathsf{A}}= \mathbb{Q}(t_1,t_2)[[u]]\, ,
\end{equation}
and the corresponding free modules
\begin{equation}\label{eqn:vec_spaces}
\V = \mathcal{F}^n \otimes_ {\mathbb{Q}} \mathsf{A}\, , \ \ \ 
 \widetilde{\V} = \mathcal{F}^n \otimes_ {\mathbb{Q}} \widetilde{\mathsf{A}}\, ,
\end{equation}
 where $\mathcal{F}^n$ is the Fock space with basis
indexed by $\text{Part}(n)$.
While the inner product for the CohFT determined by $\Hilb$
is the inner product defined in \eqref{inner_prod},
$$\eta(\mu,\nu)= \langle \mu | \nu \rangle 
=\frac{(-1)^{|\mu|-\ell(\mu)}}{(t_1 t_2)^{\ell(\mu)}} 
\frac{\delta_{\mu\nu}}{\zz(\mu)} 
\,, $$
the inner product for the CohFT determined by $\Sym(\mathbb{C}^2)$
differs by a sign,
$$\widetilde{\eta}(\mu,\nu)= (-1)^{|\mu|-\ell(\mu)} \langle \mu | \nu \rangle =
\frac{1}{(t_1 t_2)^{\ell(\mu)}} 
\frac{\delta_{\mu\nu}}{\zz(\mu)} \, .
$$
Since the CohFTs are both
semisimple, we obtain unique $\mathsf{R}$-matrices from
the Givental-Teleman classification,
$$\mathsf{R}^{\mathsf{Hilb}} \, \ \ \text{and} \ \ \mathsf{R}^{\mathsf{Sym}}\, .$$

After rescaling the insertion $\mu$ in
the $\T$-equivariant Gromov-Witten theory of 
$\Sym(\mathbb{C}^2)$
by the factor $(-i)^{|\mu|-\ell(\mu)}$, the inner products match. Let
\begin{equation}\label{eqn:mu_tilde}
|\widetilde{\mu}\, \rangle= (-i)^{\ell(\mu)-|\mu|} |\mu\rangle \, \in
\widetilde{\mathsf{V}}.
\end{equation}
The linear transformation
$\mathsf{V} \rightarrow \widetilde{\mathsf{V}}$
defined by $$\mu \mapsto \widetilde{\mu}$$
respects the inner products $\eta$ and $\widetilde{\eta}$.
Moreover, the correspondence claimed by Theorem \ref{crc} 
then simplifies to   
$$\blang \mu^1, \mu^2, \ldots, \mu^r \brang_{g}^{\Hilb} =
\blang \widetilde{\mu}^1, \widetilde{\mu}^2, \ldots, \widetilde{\mu}^r \brang_{g}^{{\Sym}(\mathbb{C}^2)}$$
 after the variable change $-q=e^{iu}$.

To prove Theorems \ref{tt22} and \ref{crc},
we will construct an operator series 
$\mathsf{R}$ defined over an open set
of values of $q\in \mathbb{C}$
containing $q=0$ and $q=-1$ with
the two following properties
\begin{enumerate}
\item[$\bullet$]
$\mathsf{R}=\mathsf{R}^{\text{Hilb}}$
near $q=0$, 
\item[$\bullet$]
$\mathsf{R}=\mathsf{R}^{\text{Sym}}$
near $q=-1$ after the variable change $-q=e^{iu}$.
\end{enumerate}
The operator $\mathsf{R}$ is obtained from the asymptotic solution to the 
QDE of $\Hilb$. The existence of
$\mathsf{R}$ is guaranteed by results of \cite{d,g1}. The asymptotic solution is shown to be the asymptotic expansion of an actual solution to the QDE. The behavior of $\mathsf{R}$ near $q=-1$ is then studied by using the solution to the connection problem in \cite{op29}.

\section{The $\mathsf{R}$-matrix for $\Hilb$}

\subsection{The  formal Frobenius manifold}
We follow the notation of Section \ref{tarhs} for the
CohFT determined by the $\T$-equivariant Gromov-Witten
theory of $\Hilb$. The genus 0 Gromov-Witten potential,
$$\mathsf{F}^{\Hilb}_0(\gamma)=
\sum_{d=0}^\infty q^d \sum_{r=0}^\infty \frac{1}{r!}
\big\langle \underbrace{\gamma, \ldots, \gamma}_{r}  \big\rangle_{0,d}^{\Hilb}\,, \ \ \ \gamma \in \mathsf{V}$$
is a formal series in the ring
$\mathsf{A}[[\mathsf{V}^*]]$
where
$$ \mathsf{A}= \mathbb{Q}(t_1,t_2)[[q]]\, .$$
The $\T$-equivariant genus 0 potential $\mathsf{F}^{\Hilb}_0$
defines a formal Frobenius manifold 
$$(\mathsf{V}, \star, \eta)$$
at the
origin of $\mathsf{V}$.

A basic result of
\cite{op} is the {\em rationality} of the dependence
on $q$. Let
$$\mathsf{Q}=\mathbb{Q}(t_1,t_2,q)\ $$
be the field of rational functions.
By \cite[Section 4.2]{op}, 
$$\mathsf{F}^{\Hilb}_0(\gamma)\in \mathsf{Q}[[\mathsf{V}^*]]
\,.$$
We will often view the formal Frobenius manifold 
$(\mathsf{V}, \star, \eta)$
as defined
over the field
$\mathsf{Q}$ instead of the ring  $\mathsf{A}$.

\subsection{Semisimplicity} \label{semisemi}
The formal Frobenius manifold $(\mathsf{V}, \star, \eta)$ is semisimple at the 
origin after
extending $\mathsf{Q}$ to the algebraic
closure $\overline{\mathsf{Q}}$.
The semisimple
algebra $$(\text{Tan}_0\mathsf{V}, \star_0, \eta)$$ 
is the small quantum cohomology
of $\Hilb$.
The matrices of quantum multiplication of
the algebra $\text{Tan}_0(\mathsf{V})$
have coefficients in $\mathsf{Q}$.
The restriction to $q=0$ is well-defined
and semisimple with idempotents proportional
to the classes of the $\T$-fixed points
of $\Hilb$. The idempotents of the algebra
 $\text{Tan}_0(\mathsf{V})$
may be written in the Nakajima basis
after extending scalars to $\overline{\mathsf{Q}}$. The algebraic
closure{\footnote{More precisely, coefficients
lie in the ring $\overline{\mathsf{Q}} \cap 
\mathsf{A}$ since the eigenvalues lie in
$\mathsf{A}$.}}
is required since the eigenvalues
of the matrices of quantum multiplication
lie in finite extensions of $\mathsf{Q}$.

Since the formal Frobenius manifold $(\mathsf{V}, \star, \eta)$ is semisimple at 
the 
origin, the full algebra
\begin{equation}\label{fulla}
(\text{Tan}_0(\mathsf{V}) \otimes_{\mathsf{Q}}\mathsf{Q}[[\mathsf{V}^*]], \star, \eta)
\end{equation}
is also semisimple (after extending scalars
to $\overline{\mathsf{Q}}$).
The idempotents of the algebra $\text{Tan}_0(\mathsf{V})$ can be
lifted to all orders to obtain
idempotents $\epsilon_\mu$ of the
full algebra \eqref{fulla} parameterized
by partitions $\mu$ corresponding to the
$\T$-fixed points of $\Hilb$.

For the formal Frobenius manifold
$(\mathsf{V},\star,\eta)$, there are two
bases of vector fields which play a fundamental
role in the theory:
\begin{enumerate}
\item[$\bullet$] the {\em flat} vector fields
$\partial_\mu$
corresponding the basis elements 
$|\mu\rangle \in \mathsf{V}$,\\
\item[$\bullet$] the {\em normalized
idempotents}
$\widetilde{\epsilon}_\mu$
of the quantum product $\star$.
\end{enumerate}
The normalized idempotents are 
proportional to the idempotents   and satisfy
$$\eta(\widetilde{\epsilon}_\mu,
\widetilde{\epsilon}_\nu)= \delta_{\mu\nu}\, .$$
Let $\Psi$ be change of basis matrix,
$$\Psi^\nu_\mu = \eta(\widetilde{\epsilon}_\nu,\partial_\mu)\, \in \, \overline{\mathsf{Q}}[[\mathsf{V}^*]]\, $$
between the two frames.
Unique {\em canonical coordinates} at the origin,
$$\big\{\, u_\mu \in \overline{\mathsf{Q}}[[\mathsf{V}^*]]\, \big\}_{\mu\in \text{Part}(n)}\, $$
can be found satisfying
$$u_\mu(0)=0\,, \ \  \ \ \ \frac{\partial}{\partial u_\mu} = \epsilon_\mu \, .$$

By \cite[Prop. 1.1]{g1}, 
 the quantum differential equation{\footnote{We follow the notation of
the exposition in \cite{lp}. See
also \cite{d}.}} 
\begin{equation}\label{qqde}
\nabla_z \widetilde{\mathsf{S}}= 0
\end{equation}
has a formal fundamental solution{\footnote{Often the solution is written in the basis of
flat vector fields as
$$\mathsf{S}= \Psi^{-1} \mathsf{R}(z) e^{\mathsf{u}/z}\, .$$}}
in the basis of normalized
idempotents of the form 
\begin{equation}\label{aasymp_S}
\widetilde{\mathsf{S}}= 
\mathsf{R}(z) e^{\mathsf{u}/z}
\end{equation} 
Here, $\nabla_z$ is the Dubrovin
connection,
\begin{equation}
\label{rtrtrt}
\mathsf{R}(z)=\Id+\mathsf{R}_1z
+\mathsf{R}_2z^2+\mathsf{R}_3 z^3+ \ldots
\end{equation} is 
a series of $|\text{Part}(n)| \times
|\text{Part}(n)|$ matrices starting
with the identity matrix $\bf{1}$, and
$\mathsf{u}$ is a diagonal matrix with
the diagonal entries given by the
canonical coordinates $u_\mu$.

We view $\mathsf{R}(z)$ as
an
$\text{End}(\mathsf{V})$-valued formal power series 
in $z$
written in the basis of normalized
idempotents. The 
series $\mathsf{R}(z)$ satisfies the symplectic condition 
\begin{equation}\label{unitary}
\mathsf{R}^{\dagger}(-z)\mathsf{R}(z)=\Id,
\end{equation}
 where $\mathsf{R}^\dagger(z)$ is 
 the adjoint of 
 $\mathsf{R}(z)$ with respect to the inner
 product $\eta$. 
A detailed treatment of the construction and  properties
 of $\mathsf{R}(z)$ can found in \cite[Section 4.6 of Chapter 1]{lp}.

An {\em $\mathsf{R}$-matrix} associated
to the formal Frobenius manifold
$(\mathsf{V},\star, \eta)$ 
is a matrix series of the form \eqref{rtrtrt} which
determines a solution \eqref{aasymp_S}
of the quantum differential equation \eqref{qqde}
{\em and} satisfies the symplectic condition
\eqref{unitary}.
The two basic properties of  $\mathsf{R}$-matrices which we will use are:
\begin{enumerate}
    \item[(i)] There exists a unique
    $\mathsf{R}$-matrix $\mathsf{R}^{\mathsf{Hilb}}$
    associated to $(\mathsf{V},\star,\eta)$
      with coefficients
    in $\mathsf{A}[[\mathsf{V}^*]]$
    which generates the $\T$-equivariant
    Gromov-Witten theory of $\Hilb$ in
    all higher genus.
    \item[(ii)] Two $\mathsf{R}$-matrices    
    associated to $(\mathsf{V},\star,\eta)$
    with coefficients in $\mathsf{A}[[\mathsf{V}^*]]$
    must differ by right multiplication by 
    $$\exp\left( \sum_{j=1}^\infty \mathsf{a}_{2j-1} z^{2j-1}\right)$$
    where each $\mathsf{a}_{2j-1}$ 
    is a {\em diagonal} matrix with
    coefficients in $\mathsf{A}$.
    \end{enumerate}
Property (i) is a statement of the Givental-Teleman
classification of semisimple CohFTs applied
to the $\T$-equivariant Gromov-Witten theory
of $\Hilb$.
For properties (i) and (ii), we have 
left the field $\overline{\mathsf{Q}}$
and returned to the ring $\mathsf{A}$ because
the CohFT associated to $\Hilb$ is defined
over $\mathsf{A}$. In Section 
\ref{ppcrcr}, 
the unique $\mathsf{R}$-matrix of property (i) 
will be shown to actually
have coefficients in $\overline{\mathsf{Q}}[[\mathsf{V}^*]]$.

\subsection{Divisor equation} \label{ddtt2}
The formal Frobenius manifold $(\mathsf{V},\star,\eta)$ is actually well-defined away from the origin along the
line with coordinate $t$ determined by the vector 
$$| 2,1^{n-2} \rangle \in \mathsf{V}\,. $$
At the point $-t|2, 1^{n-2}\rangle \in \mathsf{V}$, the potential of the
Frobenius manifold is
\begin{eqnarray*}
\mathsf{F}^{\Hilb}_0\Big(
-t|2, 1^{n-2}\rangle +\gamma
\Big)&=&
\sum_{d=0}^\infty q^d \sum_{r=0}^\infty\sum_{m=0}^\infty 
\frac{t^m}{r!m!}
\big\langle \underbrace{\gamma, \ldots, \gamma}_{r}, \underbrace{D, \ldots, D}_{m}  
\big\rangle_{0,d}^{\Hilb} \\
& = & 
\sum_{d=0}^\infty q^d \sum_{r=0}^\infty\sum_{m=0}^\infty 
\frac{(dt)^m}{r!m!}
\big\langle \underbrace{\gamma, \ldots, \gamma}_{r}  
\big\rangle_{0,d}^{\Hilb} \, ,\\
& = &
\sum_{d=0}^\infty q^d e^{dt} \sum_{r=0}^\infty\frac{1}{r!}
\big\langle \underbrace{\gamma, \ldots, \gamma}_{r}\rangle_{0,d}^{\Hilb}
\end{eqnarray*}
for $\gamma \in \mathsf{V}$ and
$
D = - \lv2,1^{n-2}\rang$
as in \eqref{DDDD}.
We have used
the {\em divisor equation} of
Gromov-Witten theory in the second equality.
The potential  near the point 
$-t|2, 1^{n-2}\rangle\in \mathsf{V}$ is obtained
from the potential at $0\in \mathsf{V}$
by the substitution
$$\mathsf{F}_0^{\Hilb}
\Big(
-t|2, 1^{n-2}\rangle +\gamma
\Big)
= \mathsf{F}_0^{\Hilb}(\gamma)\Big|_{q\mapsto
qe^t}\, .$$
The Frobenius manifold is semisimple
at all the points $-t|2, 1^{n-2}\rangle\in \mathsf{V}$ .

Let $\Omega$ be the CohFT associated
to the $\T$-equivariant Gromov-Witten 
theory of $\Hilb$. The genus 0 data
of $\Omega$ is exactly given 
by the formal Frobenius manifold $(\mathsf{V},\star, \eta)$ at 
the origin. Define the $-t|2,1^{n-2}\rangle$-shifted
CohFT by
\begin{eqnarray*}
 \Omega_{g,r}^{-t|2,1^{n-2}\rangle}(\gamma \otimes \cdots \otimes \gamma)&=&\sum_{m\geq 0}\frac{t^m}{m!}\, \rho^{r+m}_{r*}\left(\Omega_{g, r+m}(\gamma \otimes\cdots \otimes \gamma\otimes D^{\otimes m})\right) \\
 &=&
 \Omega_{g,r}(\gamma \otimes \cdots \otimes \gamma)\Big|_{q\mapsto qe^t}\, .
\end{eqnarray*}
Here, $\rho^{r+m}_r$ is the forgetful map
which drops the last $m$ markings,
\begin{equation}\label{eqn:fgt_map}
\rho^{r+m}_r:\Mbar_{g,r+m}\to \Mbar_{g,r}\, .
\end{equation}
We have again used the divisor equation
of Gromov-Witten theory in the
second equality.

Let $\mathsf{R}^{\mathsf{Hilb}}$ be the unique
$\mathsf{R}$-matrix associated to the $\T$-equivariant
Gromov-Witten theory of $\Hilb$.
The shifted CohFT $\Omega_{g,r}^{-t|2,1^{n-2}\rangle}$
is obtained from the semisimple 
genus 0 data
$$\mathsf{F}_0^{\Hilb}
\Big(
-t|2, 1^{n-2}\rangle +\gamma
\Big)$$
by the unique $\mathsf{R}$-matrix
$$\mathsf{R}^{\mathsf{Hilb}}
\Big(
-t|2, 1^{n-2}\rangle +\gamma
\Big)\, .$$
On the other hand, the $\mathsf{R}$-matrix
$$\mathsf{R}^{\mathsf{Hilb}}\Big|_{q\mapsto qe^t}$$
also generates 
$\Omega_{g,r}^{-t|2,1^{n-2}\rangle}$
from the same semisimple genus 0 data.
By the uniqueness of the $\mathsf{R}$-matrix
in the Givental-Teleman classification,
$$\mathsf{R}^{\mathsf{Hilb}}
\Big(
-t|2, 1^{n-2}\rangle +\gamma
\Big) =
\mathsf{R}^{\mathsf{Hilb}}\Big|_{q\mapsto qe^t}\, .$$
We find the following differential equation:
\begin{equation}\label{j22j}
-\frac{\partial}{\partial t} \mathsf{R}^{\mathsf{Hilb}} = q\frac{\partial}{\partial q} \mathsf{R}^{\mathsf{Hilb}}\, 
.
\end{equation}

The differential equation is an {\em extra} condition
satisfied by $\mathsf{R}^{\Hilb}$
in {\em addition} to
determining a solution \eqref{aasymp_S}
of the quantum differential equation \eqref{qqde}
{and} satisfying the symplectic condition
\eqref{unitary}.

\begin{prop}
\label{kkqq22}
Two $\mathsf{R}$-matrices    
    associated to $(\mathsf{V},\star,\eta)$
    with coefficients in $\mathsf{A}[[\mathsf{V}^*]]$
    which {\em both} satisfy the
    differential equation
$$-\frac{\partial}{\partial t} \mathsf{R} = q\frac{\partial}{\partial q} \mathsf{R}
$$ 
must differ by right multiplication by 
    $$\exp\left( \sum_{i=1}^\infty \mathsf{a}_{2j-1} z^{2j-1}\right)$$
    where each $\mathsf{a}_{2j-1}$ 
    is a {\em diagonal} matrix with
    coefficients in $\mathbb{Q}(t_1,t_2)$.
\end{prop}

\begin{proof} 
Let $\widetilde{\mathsf{R}}$ and $\widehat{\mathsf{R}}$
be two $\mathsf{R}$-matrices associated to
$(\mathsf{V},\star,\eta)$ which both
satisfy the additional differential equation.
By property (ii) of
$\mathsf{R}$-matrices associated to $(\mathsf{V},\star,\eta)$,
$$
\widetilde{\mathsf{R}}= \widehat{\mathsf{R}}\cdot
\exp\left( \sum_{i=1}^\infty \mathsf{a}_{2j-1} z^{2j-1}\right)$$
    where each $\mathsf{a}_{2j-1}$ 
    is a {diagonal} matrix with
    coefficients in $$\mathsf{A}=\mathbb{Q}(t_1,t_2)[[q]]\, .$$
Differentiation yields
\begin{eqnarray*}
-\frac{\partial}{\partial t} \widetilde{\mathsf{R}} &  = &
-\frac{\partial}{\partial t} \widehat{\mathsf{R}} \cdot
\exp\left( \sum_{i=1}^\infty \mathsf{a}_{2j-1} z^{2j-1}\right) -
\widehat{\mathsf{R}}
\cdot \frac{\partial}{\partial t} 
\exp\left( \sum_{i=1}^\infty \mathsf{a}_{2j-1} z^{2j-1}\right) \\
& = & 
-\frac{\partial}{\partial t} \widehat{\mathsf{R}} \cdot
\exp\left( \sum_{i=1}^\infty \mathsf{a}_{2j-1} z^{2j-1}\right) \\
& = &
\ q\frac{\partial}{\partial q} \widehat{\mathsf{R}}
\cdot \exp\left( \sum_{i=1}^\infty \mathsf{a}_{2j-1} z^{2j-1}\right)\, .
\end{eqnarray*}
The second equality uses the independence 
of $\mathsf{a}_{2j-1}$ with respect to
the coordinate $t$. The third equality
uses the differential equation for
$\widehat{\mathsf{R}}$. Application of the operator
$q\frac{\partial}{\partial q}$ to $\widetilde{\mathsf{R}}$ yields
$$
q\frac{\partial}{\partial q} \widetilde{\mathsf{R}}  = 
q\frac{\partial}{\partial q} \widehat{\mathsf{R}} \cdot
\exp\left( \sum_{i=1}^\infty \mathsf{a}_{2j-1} z^{2j-1}\right) +
\widehat{\mathsf{R}}
\cdot q\frac{\partial}{\partial q} 
\exp\left( \sum_{i=1}^\infty \mathsf{a}_{2j-1} z^{2j-1}\right)\, . $$
Finally, using the differential equation
for $\widetilde{\mathsf{R}}$, we find
$$\widehat{\mathsf{R}}
\cdot q\frac{\partial}{\partial q} 
\exp\left( \sum_{i=1}^\infty \mathsf{a}_{2j-1} z^{2j-1}\right)=0\, . $$
By the invertibility of $\widehat{\mathsf{R}}$,
$$q\frac{\partial}{\partial q} 
\exp\left( \sum_{i=1}^\infty \mathsf{a}_{2j-1} z^{2j-1}\right)=0\, .$$ 
Hence, the matrices $\mathsf{a}_{2j-1}$
are independent of $q$.
\end{proof}

\subsection{Proof of Theorem \ref{ff11}} \label{ddtt}
Since the CohFT $\Omega_{g,r}$ associated to the 
$\T$-equivariant Gromov-Witten theory of $\Hilb$ is defined
over the ring 
$$\mathsf{A}= \mathbb{Q}(t_1,t_2)[[q]]\, ,$$
we can consider the CohFT $\Omega_{g,r}\big|_{q=0}$
defined over the ring $\mathbb{Q}(t_1,t_2)$.
The CohFT $\Omega_{g,r}\big|_{q=0}$ is semisimple with
$\mathsf{R}$-matrix given by the $q=0$ restriction,
$$\mathsf{R}^{\mathsf{Hilb}}\big|_{q=0}\, ,$$
of the  $\mathsf{R}$-matrix
 of $\Omega_{g,r}$.
As a corollary of Proposition \ref{kkqq22}, we obtain
the following characterization of 
$\mathsf{R}^{\mathsf{Hilb}}$.

\begin{prop}\label{vv44}
$\mathsf{R}^{\text{\em Hilb}}$ is the unique $\mathsf{R}$-matrix
associated to $(\mathsf{V},\star, \eta)$ which satisfies
\begin{equation*}
-\frac{\partial}{\partial t} \mathsf{R}^{\mathsf{Hilb}} = q\frac{\partial}{\partial q} \mathsf{R}^{\mathsf{Hilb}}\, 
\end{equation*}
{\em and} has $q=0$ restriction which
equals the $\mathsf{R}$-matrix of 
the restricted {\em CohFT} $\Omega_{g,r}\big|_{q=0}$.
\end{prop}

Since $\Omega_{g,r}\big|_{q=0}$ concerns constant maps
of genus $g$ curves to $\Hilb$, the CohFT can be written
explicitly in terms of Hodge integrals.
The moduli space of maps in degree 0 is
$$\oM_{g,r}(\Hilb,0) = \oM_{g,r}\times \Hilb$$
with virtual class
$$\mathbb{E}^*_g \otimes \text{Tan}(\Hilb)\, .$$
Since the Hodge bundle is pulled back from $\oM_{1,1}$ in genus 1 and
$\oM_{g}$ in higher genera,
all invariants in positive genus vanish other than
\begin{eqnarray}
\big\langle \mu \big\rangle_{1,0}^{\Hilb} &=& \sum_{\eta\in \text{Part}(n)} \mu|_\eta \int_{\oM_{1,1}}
\frac{e\left(\mathbb{E}_1^* \otimes \text{Tan}_\eta(\Hilb)\right)}
{e\left(\text{Tan}_\eta(\Hilb)\right)
}\, ,  \label{gg93}
\\ \nonumber
\big\langle \, \big\rangle_{g\geq 2,0}^{\Hilb} &=& \sum_{\eta\in \text{Part}(n)} \int_{\oM_g}
\frac{e\left(\mathbb{E}_g^* \otimes \text{Tan}_\eta(\Hilb)\right)}
{e\left(\text{Tan}_\eta(\Hilb)\right)
}\,
.
\end{eqnarray}
Theorem \ref{ff11} follows from Proposition \ref{vv44}  together
with the fact that \eqref{gg93} is the complete list of degree 0 invariants. \qed

\vspace{8pt}
The unique $\mathsf{R}$-matrix  $\mathsf{R}^{\mathsf{Hilb}}\big|_{q=0}$
of the CohFT $\Omega_{g,r}\big|_{q=0}$ 
can be explicitly written
after the
coordinates on $\mathsf{V}$
are set to 0. The formula 
is presented  in Section \ref{subsec:hilb_deg0}.

\subsection{Proof of Theorem \ref{crcr}} \label{ppcrcr}
As we have seen in Section \ref{ddtt2} using the divisor equation,
the dependence of the potential of the formal
Frobenius manifold $(\mathsf{V}, \star, \eta)$ 
at the origin,
$$\mathsf{F}_0^{\Hilb}\in \mathsf{Q}[[\mathsf{V}^*]],$$
along  $-t|2,1^{n-2}\rangle$
can be expressed as
$$ \mathsf{F}_0^{\Hilb} =
\Big(\mathsf{F}_0^{\Hilb}\big|_{t=0}\Big)_{q\mapsto qe^t}\, .$$
The same dependence on $t$ then also holds for
the matrices of quantum multiplication for $(\mathsf{V}, \star, \eta)$
and their common eigenvalues.

In the procedure{\footnote{See \cite[Section 4.6 of Chapter 1]{lp}}} for constructing an $\mathsf{R}$-matrix associated
to $(\mathsf{V},\star,\eta)$, we can take all the undetermined diagonal constants for $\mathsf{R}_{2j-1}$
equal to 0 for all $j$. The resulting associated $\mathsf{R}$-matrix $\mathsf{R}^\#$ will satisfy
\begin{equation*}
-\frac{\partial}{\partial t} \mathsf{R}^\# = q\frac{\partial}{\partial q} \mathsf{R}^\#\, 
\end{equation*}
since the same $q\mapsto qe^t$  dependence on $t$ holds for all terms in
the procedure.
By Proposition \ref{kkqq22},
\begin{equation}\label{k33k}
\mathsf{R}^{\mathsf{Hilb}} = \mathsf{R}^{\#}\cdot
\exp\left( \sum_{i=1}^\infty \mathsf{a}_{2j-1} z^{2j-1}\right)\, ,
\end{equation}
    where each $\mathsf{a}_{2j-1}$ 
    is a {\em diagonal} matrix with
    coefficients in $\mathbb{Q}(t_1,t_2)$.

By \cite[Section 4.2]{op}, the third 
derivatives of the
potential of $(\mathsf{V},\star,\eta)$
are defined over $\mathsf{Q}$,
$$\frac{\partial^3}{\partial t_{\mu^1}\partial t_{\mu^2} \partial t_{\mu^3}}\mathsf{F}_0^{\Hilb}\in \mathsf{Q}[[\mathsf{V}^*]]\, .$$
Hence, the matrices of quantum multiplication also have coefficients
in $\mathsf{Q}[[\mathsf{V}^*]]$. As we have remarked in Section \ref{semisemi},
the common eigenvalues require finite extensions of $\mathsf{Q}$.
Using the procedure for the construction of $\mathsf{R}^{\mathsf{Hilb}}$
with undetermined diagonal constants in $\mathbb{Q}(t_1,t_2)$
fixed by \eqref{k33k},
we see that the coefficients of $\mathsf{R}^{\mathsf{Hilb}}$
lie in the ring $\overline{\mathsf{Q}}[[\mathsf{V}^*]]$.

The definition of the $\mathsf{R}$-matrix action then yields
the rationality of Theorem \ref{crcr} after applying Galois
invariance:
for all genera $g\geq 0$ and 
$\mu^1,\mu^2,\ldots,\mu^r\in \text{Part}(n)$, 
the series
$$\blang \mu^1, \mu^2, \ldots, \mu^r \brang_{g}^{\Hilb}\in \mathbb{Q}(t_1,t_2)[[q]]$$
is the Taylor expansion in $q$ of a rational
function in $\mathbb{Q}(t_1,t_1,q)$. \qed

\section{The $\mathsf{R}$-matrix for $\mathsf{Sym}^n(\mathbb{C}^2)$}
\subsection{The formal Frobenius manifold}
The  $\T$-equivariant Gromov-Witten potential in genus 0, 
$$ 
\mathsf{F}_0^{\mathsf{Sym}^n(\mathbb{C}^2)}(\gamma)=\sum_{b=0}^\infty u^b\sum_{n=0}^\infty \frac{1}{r!}\big\langle\underbrace{\gamma, \ldots, \gamma}_{r} \big\rangle_{0,b}^{\mathsf{Sym}^n(\mathbb{C}^2)},\quad \gamma\in \widetilde{\V},
$$
is a formal series in the ring $\widetilde{\mathsf{A}}[[\widetilde{\V}^*]]$, where $\widetilde{\mathsf{A}}$ is given in (\ref{eqn:ground_rings}). The potential $\mathsf{F}_0^{\mathsf{Sym}^n(\mathbb{C}^2)}$ defines a formal Frobenius manifold $$(\widetilde{\V}, \widetilde{\star}, \widetilde{\eta})$$ at the origin of $\widetilde{\V}$. The Frobenius algebra $$(\text{Tan}_0 \widetilde{\V}, \widetilde{\star}_0, \widetilde{\eta})$$
is the small\footnote{The term {\em small} here refers to deformations by {\em twisted divisors}, introduced in \cite[Section 2.2]{bg}.} quantum cohomology of $\mathsf{F}_0^{\mathsf{Sym}^n(\mathbb{C}^2)}$, which is calculated in \cite[Section 3.3]{bg}. The formal Frobenius manifold $(\widetilde{\V}, \widetilde{\star}, \widetilde{\eta})$ is semisimple at the origin. The idempotents of $\text{Tan}_0\widetilde{\V}$ can be written in terms of the standard basis of the Chen-Ruan cohomology of $\mathsf{Sym}^n(\mathbb{C}^2)$ after extension of scalars. Again, the
idempotents  of $\text{Tan}_0\widetilde{\V}$ are indexed by partitions $\mu$.

We write\footnote{$\widetilde{\mathsf{A}}^{\text{cl}}$ denotes the algebraic closure of the field of fractions of  $\widetilde{\mathsf{A}}$.} $\{\widetilde{u}_\mu\in \widetilde{\mathsf{A}}^{\text{cl}}[[\widetilde{\V}^*]]\}_{\mu\in \text{Part}(n)}$ for the unique  canonical coordinates of $(\widetilde{\V}, \widetilde{\star}, \widetilde{\eta})$ satisfying that $\widetilde{u}_\mu(0)=0$ and $\partial/\partial \widetilde{u}_\mu$ is an idempotent. By \cite[Prop. 1.1]{g1}, the quantum differential equation associated to the formal Frobenius manifold $(\widetilde{\V}, \widetilde{\star}, \widetilde{\eta})$, 
$$\widetilde{\nabla}_z\widetilde{\mathsf{S}}=0,$$
admits a formal fundamental solution of the form $$\widetilde{\mathsf{S}}=\widetilde{\mathsf{R}}(z)e^{\widetilde{\mathsf{u}}/z},$$
written in the basis of normalized idempotents. Here $\widetilde{\nabla}_z$ is the Dubrovin connection associated to $(\widetilde{\V}, \widetilde{\star}, \widetilde{\eta})$, $\widetilde{\mathsf{u}}$ is the diagonal matrix with diagonal entries given by $\widetilde{u}_\mu$, and $$\widetilde{\mathsf{R}}(z)=\Id+\widetilde{\mathsf{R}}_1z+\widetilde{\mathsf{R}}_2z^2+...$$ is an $\text{End}(\widetilde{\V})$-valued formal power series in $z$ written in the basis of normalized idempotents. The symplectic condition $$\widetilde{\mathsf{R}}^\dagger(-z)\widetilde{\mathsf{R}}(z)=\Id,$$
taken with respect to the inner product $\widetilde{\eta}$, is required.

By \cite[Prop. 1.1]{g1}, two $\mathsf{R}$-matrices satisfying the quantum differential equation associated to $(\widetilde{\V}, \widetilde{\star}, \widetilde{\eta})$ and the symplectic condition must differ by right multiplication by $$\exp\left(\sum_{j=1}^\infty \widetilde{\mathsf{a}}_{2j-1}z^{2j-1} \right),$$
where each $\widetilde{\mathsf{a}}_{2j-1}$ is a diagonal matrix with coefficients in $\widetilde{\mathsf{A}}$.

\subsection{Divisor equation}\label{diveqq}

Let $\oM_{g,r}(\text{Sym}^n(\mathbb{C}^2))$ be the moduli space\footnote{The moduli stack, which parameterizes stable maps {\em with sections to all marked gerbes}, is also
used in \cite{Tseng}.}
of $n$-pointed genus $g$ stable maps to $\text{Sym}^n(\mathbb{C}^2)$. 
Let $$\text{ev}_i: \oM_{g,r}(\text{Sym}^n(\mathbb{C}^2))\to I\text{Sym}^n(\mathbb{C}^2)$$ be the $\T$-equivariant 
evaluation map at the $i$-th marked point with values in the inertia stack $$I\text{Sym}^n(\mathbb{C}^2)$$ of $\text{Sym}^n(\mathbb{C}^2)$. The inertia stack $I\text{Sym}^n(\mathbb{C}^2)$ is a disjoint union indexed by conjugacy classes of the symmetric group $S_n$. For $\mu\in \text{Part}(n)$, the component $I_\mu\subset I\text{Sym}^n(\mathbb{C}^2)$ indexed by the conjugacy class of cycle type $\mu$ is isomorphic to the stack quotient $$\left[\mathbb{C}^{2n}_\sigma/C(\sigma)\right]\, ,$$ where $\sigma\in S_n$ has cycle type $\mu$, $\mathbb{C}^{2n}_\sigma$ is the invariant part of $\mathbb{C}^{2n}$ under the action of $\sigma$, and $C(\sigma)$ is the centralizer of $\sigma\in S_n$. Let $$[I_\mu]\in H^0_\T(I_\mu)\subset  H_\T^*(I\text{Sym}^n(\mathbb{C}^2))$$ be the fundamental
class. There is an additive isomorphism $$H_\T^*(I\text{Sym}^n(\mathbb{C}^2))\simeq
\widetilde{\mathsf{V}}$$
given by sending $[I_\mu]$ to 
$|\mu\big\rangle$. 

The {\em unramified $\T$-equivariant  Gromov-Witten invariants} are defined by  $$\big\langle \mu^1,\ldots, \mu^r\big\rangle_{g,0}^{\text{Sym}^n(\mathbb{C}^2)}=\int_{[\oM_{g,r}(\text{Sym}^n(\mathbb{C}^2))]^{vir}}\prod_{i=1}^r \text{ev}_i^*([I_{\mu^i}]).$$
Consider $$\text{ev}_{r+1}^{-1}(I_{(2, 1^{n-2})})\cap\ldots \cap \text{ev}_{r+b}^{-1}(I_{(2, 1^{n-2})})\subset \oM_{g,r+b}(\text{Sym}^n(\mathbb{C}^2)),$$ and let 
$$\oM_{g,r, b}(\text{Sym}^n(\mathbb{C}^2))=[\left(\text{ev}_{r+1}^{-1}(I_{(2, 1^{n-2})})\cap\ldots \cap \text{ev}_{r+b}^{-1}(I_{(2, 1^{n-2})}) \right)/S_b]$$ where the symmetric group $S_b$ acts by permuting the last $b$ marked points. The Gromov-Witten invariants with $b$ free ramification points are defined by 
\begin{equation*}
\big\langle \mu^1, \ldots, \mu^r\big\rangle_{g,b}^{\text{Sym}^n(\mathbb{C}^2)}=\int_{[\oM_{g,r, b}(\text{Sym}^n(\mathbb{C}^2))]^{vir}}\prod_{i=1}^r \text{ev}_i^*([I_{\mu^i}])\, .
\end{equation*} The
above definition 
can be rewritten{\footnote{See  \cite[Section 2.2]{bg} for a
parallel discussion.}}
as
\begin{equation}
\label{eqn:div_eqn_sym}
\big\langle \mu^1, \ldots, \mu^r \big\rangle_{g,b}^{\mathsf{Sym}^n(\mathbb{C}^2)}
=\frac{1}{b!}\big\langle \mu^1, \ldots, \mu^r, \underbrace{(2, 1^{n-2}), \ldots, (2, 1^{n-2})}_{b}  \big\rangle_{g,0}^{\mathsf{Sym}^n(\mathbb{C}^2)},
\end{equation}

Property \eqref{eqn:div_eqn_sym}
will be important for our
study of the Frobenius manifold $(\widetilde{\V}, \widetilde{\star}, \widetilde{\eta})$.
In particular, the {\em divisor
equation} holds:
\begin{equation*}
\big\langle \mu^1, \ldots, \mu^{r}, (2,1^{n-2}) \big\rangle_{g,b}^{\mathsf{Sym}^n(\mathbb{C}^2)}
= (b+1) \big\langle \mu^1, \ldots, \mu^r
\big\rangle_{g,b+1}^{\mathsf{Sym}^n(\mathbb{C}^2)}\, ,
\end{equation*}
or equivalently,
\begin{equation*}
\big\langle \mu^1, \ldots, \mu^{r}, (2,1^{n-2}) \big\rangle_{g}^{\mathsf{Sym}^n(\mathbb{C}^2)}
= \frac{\partial}{\partial u} \big\langle \mu^1, \ldots, \mu^r
\big\rangle_{g}^{\mathsf{Sym}^n(\mathbb{C}^2)}\, ,
\end{equation*}
for the generating series
\begin{eqnarray*}
\blang \mu^1,  \ldots, \mu^r, (2,1^{n-2}) \brang_{g}^{{\Sym}(\mathbb{C}^2)}
& = &  \sum_{b=0}^\infty 
\blang \mu^1, \ldots, \mu^r ,
(2,1^{n-2})
\brang_{g,b}^{{\Sym}(\mathbb{C}^2)} u^b
\, ,\\
\blang \mu^1, \ldots, \mu^r \brang_{g}^{{\Sym}(\mathbb{C}^2)}
& = &  \sum_{b=0}^\infty 
\blang \mu^1,  \ldots, \mu^r \brang_{g,b}^{{\Sym}(\mathbb{C}^2)} u^b
\, .
\end{eqnarray*}



\subsection{Shifted CohFT}
Let $\widetilde{t}$ be the coordinate of the vector $[I_{(2, 1^{n-2})}]\in\widetilde{\V}$. The formal Frobenius manifold $(\widetilde{\V}, \widetilde{\star}, \widetilde{\eta})$ is well-defined at $\widetilde{t}\,[I_{(2, 1^{n-2})}]\in \widetilde{\V}$, at which the potential of the Frobenius manifold is 
\begin{equation*}
    \begin{split}
    \mathsf{F}_0^{\mathsf{Sym}^n(\mathbb{C}^2)}\left(\widetilde{t}\,[I_{(2, 1^{n-2})}]+\gamma\right)
    &=\sum_{b=0}^\infty u^b \sum_{r=0}^\infty\sum_{m=0}^\infty\frac{\widetilde{t}^m}{r!m!}\big\langle \underbrace{\gamma, \ldots, \gamma}_{r}, \underbrace{(2, 1^{n-2}), \ldots, (2, 1^{n-2})}_{m}  \big\rangle_{0,b}^{\mathsf{Sym}^n(\mathbb{C}^2)}\\
    &=\sum_{b=0}^\infty  \sum_{r=0}^\infty\sum_{m=0}^\infty\frac{u^b\,\widetilde{t}^m}{r!m!b!}\big\langle \underbrace{\gamma, \ldots, \gamma}_{n}, \underbrace{(2, 1^{n-2}), \ldots, (2, 1^{n-2})}_{m+b}  \big\rangle_{0,0}^{\mathsf{Sym}^n(\mathbb{C}^2)}\\
    &=\sum_{b=0}^\infty  \sum_{r=0}^\infty\sum_{m=0}^\infty\frac{u^b\, \widetilde{t}^m(m+b)!}{r!m!b!(m+b)!}\big\langle \underbrace{\gamma, \ldots, \gamma}_{r}, \underbrace{(2, 1^{n-2}), \ldots, (2, 1^{n-2})}_{m+b}  \big\rangle_{0,0}^{\mathsf{Sym}^n(\mathbb{C}^2)}\\
    &=\sum_{b=0}^\infty  \sum_{r=0}^\infty\sum_{m=0}^\infty\frac{u^b\,\widetilde{t}^m(m+b)!}{r!m!b!}\big\langle \underbrace{\gamma, \ldots, \gamma}_{r}  \big\rangle_{0,m+b}^{\mathsf{Sym}^n(\mathbb{C}^2)}\\
    &=\sum_{b=0}^\infty  \sum_{r=0}^\infty\sum_{m=0}^\infty\frac{u^b\,\widetilde{t}^m}{r!}\binom{m+b}{m}\big\langle \underbrace{\gamma, \ldots, \gamma}_{r}  \big\rangle_{0,m+b}^{\mathsf{Sym}^n(\mathbb{C}^2)}\\
    &=\sum_{d=0}^\infty (u+\widetilde{t}\, )^d\sum_{r=0}^\infty \frac{1}{r!}\big\langle\underbrace{\gamma, \ldots, \gamma}_{r} \big\rangle_{0,b}^{\mathsf{Sym}^n(\mathbb{C}^2)}.
    \end{split}
\end{equation*}
In the above calculation above, we have used (\ref{eqn:div_eqn_sym}) in the second and fourth equalities. 
We conclude
$$    \mathsf{F}_0^{\mathsf{Sym}^n(\mathbb{C}^2)}\left(\widetilde{t}\, [I_{(2, 1^{n-2})}]+\gamma\right)
=\mathsf{F}_0^{\mathsf{Sym}^n(\mathbb{C}^2)}\left(\gamma\right)\Big|_{u\mapsto u+\widetilde{t}}\ .
$$
Certainly the Frobenius manifold is semisimple at $\widetilde{t}\, [I_{(2, 1^{n-2})}]\in \widetilde{\V}$.

Let $\widetilde{\Omega}_{g,r}$ be the CohFT associated to the $\T$-equivariant Gromov-Witten theory of $\mathsf{Sym}^n(\mathbb{C}^2)$:
\begin{equation*}
    \begin{split}
    \widetilde{\Omega}_{g,r}(\gamma\otimes...\otimes\gamma)=\sum_{b=0}^\infty \frac{u^b}{b!} \rho_*\Big(\text{ev}^*(\gamma^{\otimes r}\otimes[I_{(2, 1^{n-2})}]^{\otimes  b})\cap [\oM_{g,r+b}(\text{Sym}^n(\mathbb{C}^2))]^{vir} \Big)\,  .
    \end{split}
\end{equation*}
 The genus 0 data of $\widetilde{\Omega}_{g,r}$ is exactly given by the formal Frobenius manifold $(\widetilde{\mathsf{V}},\widetilde{\star}, \widetilde{\eta})$ at 
the origin. Define the $\widetilde{t}\,[I_{(2, 1^{n-2})}]$-shifted CohFT by
\begin{eqnarray*}
 \widetilde{\Omega}_{g,r}^{\widetilde{t}\,[I_{(2, 1^{n-2})}]}(\gamma \otimes \cdots \otimes \gamma)&=&\sum_{m\geq 0}\frac{\widetilde{t}^m}{m!}\, \rho^{r+m}_{r*}\left(\widetilde{\Omega}_{g, r+m}(\gamma \otimes\cdots \otimes \gamma\otimes [I_{(2, 1^{n-2})}]^{\otimes m})\right) \\
 &=&
 \widetilde{\Omega}_{g,n}(\gamma \otimes \cdots \otimes \gamma)\Big|_{u\mapsto u+\widetilde{t}}\, ,
\end{eqnarray*}
where $\rho_r^{r+m}$ is given by \eqref{eqn:fgt_map}. 

Let $\mathsf{R}^\mathsf{Sym}$ be the unique $\mathsf{R}$-matrix associated to the $\T$-equivariant Gromov-Witten theory of $\text{Sym}^n(\mathbb{C}^2)$. The shifted CohFT $\widetilde{\Omega}_{g,r}^{\widetilde{t}\,[I_{(2, 1^{n-2})}]}$ is obtained from the semisimple genus $0$ data $$\mathsf{F}_0^{\mathsf{Sym}^n(\mathbb{C}^2)}\left(\widetilde{t}\,[I_{(2, 1^{n-2})}]+\gamma\right)$$ by the unique $\mathsf{R}$-matrix $$\mathsf{R}^\text{Sym}\left(\widetilde{t}\, [I_{(2, 1^{n-2})}]+\gamma \right).$$
On the other hand, the $\mathsf{R}$-matrix $$\mathsf{R}^\mathsf{Sym}\Big|_{u\mapsto u+\widetilde{t}}$$ also generates $\widetilde{\Omega}_{g,r}^{\widetilde{t}\,[I_{(2, 1^{n-2})}]}$ from the same semisimple genus $0$ data. By the uniqueness of the $\mathsf{R}$-matrix in the Givental-Teleman classification, $$\mathsf{R}^\mathsf{Sym}\left(\widetilde{t}\,[I_{(2, 1^{n-2})}]+\gamma \right)=\mathsf{R}^\mathsf{Sym}\Big|_{u\mapsto u+\widetilde{t}}\ .$$
Hence, we obtain the following
differential equation:
\begin{equation}
\label{eqn:dif_eqn_sym}
\frac{\partial}{\partial\widetilde{t}}\, \mathsf{R}^\mathsf{Sym} =\frac{\partial}{\partial u}\mathsf{R}^\mathsf{Sym}\, .
\end{equation}

\begin{prop}\label{prop:R_mat_sym}
Two $\mathsf{R}$-matrices associated to $(\widetilde{\mathsf{V}},\widetilde{\star}, \widetilde{\eta})$ with coefficients in $\widetilde{\mathsf{A}}[[\widetilde{\V}^*]]$ which both satisfy the differential equation \eqref{eqn:dif_eqn_sym} must differ by right multiplication by $$\exp\left(\sum_{j=1}^\infty \widetilde{\mathsf{a}}_{2j-1}z^{2j-1} \right),$$
where each $\widetilde{\mathsf{a}}_{2j-1}$ is a diagonal matrix with coefficients in $\mathbb{Q}(t_1, t_2)$.
\end{prop}
\begin{proof}
Let $\widetilde{\mathsf{R}}$ and $\widehat{\mathsf{R}}$ be two $\mathsf{R}$-matrices associated to $(\widetilde{\mathsf{V}},\widetilde{\star}, \widetilde{\eta})$ which both satisfy (\ref{eqn:dif_eqn_sym}). Then 
$$\widetilde{\mathsf{R}}=\widehat{\mathsf{R}}\cdot \exp\left(\sum_{j=1}^\infty \widetilde{\mathsf{a}}_{2j-1}z^{2j-1} \right),$$ where each $\widetilde{\mathsf{a}}_{2j-1}$ is a diagonal matrix with coefficients in $\widetilde{\mathsf{A}}=\mathbb{Q}(t_1, t_2)[[u]]$. We will show
\begin{equation}\label{eqn:indep_u}
\frac{\partial}{\partial u} \widetilde{\mathsf{a}}_{2j-1}=0,\quad j\geq 1.
\end{equation}

By the product rule, we have 
\begin{equation*}
    \begin{split}
    \frac{\partial}{\partial \widetilde{t}}\, \widetilde{\mathsf{R}}
    =\frac{\partial}{\partial \widetilde{t}}\,\widehat{\mathsf{R}}\cdot \exp\left(\sum_{j=1}^\infty \widetilde{\mathsf{a}}_{2j-1}z^{2j-1} \right)+ \widehat{\mathsf{R}}\cdot \frac{\partial}{\partial \widetilde{t}}\, \exp\left(\sum_{j=1}^\infty \widetilde{\mathsf{a}}_{2j-1}z^{2j-1} \right)\, .
    \end{split}
\end{equation*}
Since $\widetilde{\mathsf{a}}_{2j-1}$ is independent of $\widetilde{t}$, the right side is 
\begin{equation*}
    \frac{\partial}{\partial \widetilde{t}}\, \widehat{\mathsf{R}}\cdot \exp\left(\sum_{j=1}^\infty \widetilde{\mathsf{a}}_{2j-1}z^{2j-1} \right).
\end{equation*}
Since $\widehat{\mathsf{R}}$ satisfies (\ref{eqn:dif_eqn_sym}), we obtain  $$\frac{\partial}{\partial u} \widetilde{\mathsf{R}}=\frac{\partial}{\partial u}\widehat{\mathsf{R}}\cdot \exp\left(\sum_{j=1}^\infty \widetilde{\mathsf{a}}_{2j-1}z^{2j-1} \right).$$

On the other hand, product rule also yields
\begin{equation*}
    \frac{\partial}{\partial u}\widetilde{\mathsf{R}}
    =\frac{\partial}{\partial u}\widehat{\mathsf{R}}\cdot \exp\left(\sum_{j=1}^\infty \widetilde{\mathsf{a}}_{2j-1}z^{2j-1} \right)+ \widehat{\mathsf{R}}\cdot \frac{\partial}{\partial u}\exp\left(\sum_{j=1}^\infty \widetilde{\mathsf{a}}_{2j-1}z^{2j-1} \right).
\end{equation*}
Comparing the two equations above, we find $$\widehat{\mathsf{R}}\cdot \frac{\partial}{\partial u}\exp\left(\sum_{j=1}^\infty \widetilde{\mathsf{a}}_{2j-1}z^{2j-1} \right)=0.$$
Since $\widetilde{\mathsf{R}}$ is invertible, we conclude $$\frac{\partial}{\partial u}\exp\left(\sum_{j=1}^\infty \widetilde{\mathsf{a}}_{2j-1}z^{2j-1} \right)=0\,$$
which implies
\eqref{eqn:indep_u}.
\end{proof}

As a corollary of Proposition 
\ref{prop:R_mat_sym}, we obtain
the following characterization of $\mathsf{R}^{\mathsf{Sym}}$
parallel to Proposition \ref{vv44}.

\begin{prop}\label{vv55}
$\mathsf{R}^{\mathsf{Sym}}$ is the unique $\mathsf{R}$-matrix
associated to $(\widetilde{\mathsf{V}},
\widetilde{\star}, \widetilde{\eta})$ which satisfies
\begin{equation*}
\frac{\partial}{\partial \widetilde{t}}\, \mathsf{R}^{\mathsf{Sym}} = \frac{\partial}{\partial u} \mathsf{R}^{\mathsf{Sym}}\, 
\end{equation*}
{\em and} has $u=0$ restriction which
equals the $\mathsf{R}$-matrix of 
the restricted {\em CohFT} $\widetilde{\Omega}_{g,r}\big|_{u=0}$.
\end{prop}

\subsection{Proof of Theorem \ref{4444}}
The restricted
CohFT $\widetilde{\Omega}_{g,r}\Big|_{u=0}$ is defined over $\mathbb{Q}(t_1, t_2)$ and is semisimple with the $\mathsf{R}$-matrix given by $$\mathsf{R}^\text{Sym}\Big|_{u=0}.$$ 
By
Proposition \ref{vv55},  the unique $\mathsf{R}$-matrix $\mathsf{R}^\text{Sym}$ that generates the CohFT $\widetilde{\Omega}_{g,r}$ is the unique $\mathsf{R}$-matrix associated to the Frobenius manifold $(\widetilde{\mathsf{V}},\widetilde{\star}, \widetilde{\eta})$ which satisfies
\begin{equation*}
\frac{\partial}{\partial \widetilde{t}}\, \mathsf{R}^{\mathsf{Sym}} = \frac{\partial}{\partial u} \mathsf{R}^{\mathsf{Sym}}\, 
\end{equation*}
and has $u=0$ restriction equal to the $\mathsf{R}$-matrix of the CohFT $\widetilde{\Omega}_{g,r}\Big|_{u=0}$. By the definition of $\widetilde{\Omega}_{g,r}$, the $u=0$ restriction $\widetilde{\Omega}_{g,r}\Big|_{u=0}$ is equivalent to the collection of the invariants $$\big\langle \mu^1,\ldots, \mu^r \big\rangle_{g,0}^{\text{Sym}^n(\mathbb{C}^2)}$$
for
$g\geq 0$
and $\mu^1, ...,\mu^r\in \text{Part}(n)$.\qed

\noindent The $\mathsf{R}$-matrix $\mathsf{R}^\mathsf{Sym}\Big|_{u=0}$ is
written
explicitly (after the
coordinates of $\widetilde{\mathsf{V}}$
are set to 0) in Section \ref{subsec:sym_0}.

\section{Local theories of curves: stable maps}

\subsection{Local theories of curves}

Let $\overline{\mathcal{M}}_{g,r}$ be the moduli space of Deligne-Mumford stable
curves of genus $g$ with $r$ markings.{\footnote{We will always assume 
$g$ and $r$ satisfy the {\em stability} condition $2g-2+r>0$.}}
Let
$$ \mathcal{C} \rightarrow \overline{\mathcal{M}}_{g,r}\, $$
be the universal curve with sections
$$\mathsf{p}_1, \ldots, \mathsf{p}_r : \overline{\mathcal{M}}_{g,r}\rightarrow
\mathcal{C}$$
associated to the markings.
Let 
\begin{equation}\label{fddff}
\pi: \mathbb{C}^2\times \mathcal{C} \rightarrow \overline{\mathcal{M}}_{g,r}
\end{equation}
be the universal {\em local curve} over $\overline{\mathcal{M}}_{g,r}$. The torus $\T$ acts on
the $\com^2$ factor.
The Gromov-Witten and Donaldson-Thomas theories  
of the morphism $\pi$ are defined by the $\pi$-relative $\T$-equivariant virtual class
of the universal $\pi$-relative
 moduli spaces of stable maps and stable pairs.

\subsection{Stable maps}
We define a CohFT $\widetilde{\Lambda}$
via the moduli space of $\pi$-relative stable maps to the universal local
curve \eqref{fddff}
 based on the 
algebra
\begin{equation*}
\widetilde{\mathsf{A}}= \mathbb{Q}(t_1,t_2)[[u]]
\end{equation*}
and the corresponding free module
\begin{equation*}
 \widetilde{\V} = \mathcal{F}^n \otimes_ {\mathbb{Q}} \widetilde{\mathsf{A}}\, ,
\end{equation*}
where $\mathcal{F}^n$ is the Fock space with basis
indexed by $\text{Part}(n)$.
The inner product for the CohFT 
is 
$$\widetilde{\eta}(\mu,\nu)
=\frac{1}{(t_1 t_2)^{\ell(\mu)}} 
\frac{\delta_{\mu\nu}}{\zz(\mu)}\,. $$
The algebra, free module, and inner product for $\widetilde{\Lambda}$
are exactly the same as for the CohFT $\widetilde{\Omega}$ obtained
from the orbifold Gromov-Witten theory of ${\Sym}(\mathbb{C}^2)$.

Let $\mu^1,\ldots,\mu^r \in \text{Part}(n)$, and let
 $\Mbar^\bullet_h(\pi,\mu^1,\ldots,\mu^r)$ be the moduli space
of stable{\footnote{The subscript $\bullet$ indicates the domain
curves is possibly disconnected (but no connected
component is contracted to a point).
See \cite{bp}.}}
relative maps to the fibers of $\pi$,
$$\epsilon: \Mbar^\bullet_h(\pi,\mu^1,\ldots, \mu^r) \rightarrow \Mbar_{g,r}\, .$$
The fiber of $\epsilon$ over 
the moduli point
$$(C,p_1,\ldots,p_r) \in \Mbar_{g,r}\, $$
is the moduli space of stable maps of genus $h$
to $\com^2 \times C$ relative to the divisors
determined by the nodes and 
the markings of $C$
with boundary{\footnote{The cohomology weights of the boundary condition
are all the identity class.}}
condition
$\mu^i$ over 
the divisor $\com^2 \times p_i$.
The moduli space $\Mbar^\bullet_h(\pi,\mu^1,\ldots,\mu^r)$ has $\pi$-relative
virtual dimension
$$-n(2g-2+r)+\sum_{i=1}^r \ell(\mu^i)\, .$$

The CohFT $\widetilde{\Lambda}$ is defined
via the $\pi$-relative 
$\T$-equivariant
virtual class by
$$\widetilde{\Lambda}_{g,r}
(\mu^1\otimes \ldots\otimes \mu^r)
=\sum_{b\geq 0}
\frac{u^{b}}{b!}\epsilon_*\left(
\Big[\Mbar^\bullet_{h[b]}\big(\pi, \mu^1,\ldots ,\mu^r
\big)\Big]^{vir_\pi}\right). $$
Here, summation index $b$ is the branch point number, so
$$2h[b]-2= b+n(2g-2+r)- \sum_{i=1}^r \ell(\mu^i)\, .$$
The moduli space of stable maps is
empty unless $h[b]$ is an integer.

\subsection{Axioms}
The defining axioms of a CohFT are listed
in Section \ref{ccfftt}. Axiom (i)
for $\widetilde{\Lambda}$ follows
from the symmetry of the construction.
Axiom (ii) is a consequence of the
degeneration formula of relative
Gromov-Witten theory.
Axiom (iii) requires a proof.

\begin{prop}\label{00qq}
The identity axiom holds: 
$$
\widetilde{\Lambda}_{g,r+1}(\mu^1 \otimes \cdots \otimes \mu^r \otimes \b1) = p^*\widetilde{\Lambda}_{g,r} (\mu^1 \otimes \cdots \otimes \mu^r)\, 
$$
where $p: \oM_{g,r+1} \to \oM_{g,r}$ is the forgetful map.
\end{prop}

\begin{proof}
Consider first the
standard identity equation in Gromov-Witten theory
\begin{equation}\label{t23kk}
 \Big[\Mbar^\bullet_{h,1}\big(\pi, \mu^1,\ldots ,\mu^r \big)\Big]^{vir_\pi} =
p_\pi^*
\Big[\Mbar^\bullet_{h}\big(\pi, \mu^1,\ldots ,\mu^r\big)\Big]^{vir_\pi}\, ,
\end{equation}
where $1$ is a new marking on the domain curve
and $p_\pi$ is the map forgetting 1,
\begin{equation*}
 p_\pi:\Mbar^\bullet_{h,1}\big(\pi, \mu^1,\ldots ,\mu^r \big) \rightarrow
\Mbar^\bullet_{h}\big(\pi, \mu^1,\ldots ,\mu^r\big)\, .
\end{equation*}
By stability, the image of the
marking $1$ together with the
$r$ relative points yields a stable
$r+1$ pointed curve,
$$\epsilon_1:\Mbar^\bullet_{h,1}\big(\pi, \mu^1,\ldots ,\mu^r \big)
\rightarrow \Mbar_{g,r+1}\, .$$
Since the stable maps are of degree $n$,
we obtain
\begin{equation}\label{r23kk}
\epsilon_{1*}\left(
 \Big[\Mbar^\bullet_{h,1}\big(\pi, \mu^1,\ldots ,\mu^r \big)\Big]^{vir_\pi}\right) =
n p^*
 \epsilon_*\left(
\Big[\Mbar^\bullet_{h}\big(\pi, \mu^1,\ldots ,\mu^r\big)\Big]^{vir_\pi}\right)\, .
\end{equation}
The Proposition then follows from 
\eqref{r23kk} and the
relation
\begin{equation*}
\epsilon_{1*}\left(
\Big[\Mbar^\bullet_{h,1}\big(\pi, \mu^1,\ldots ,\mu^r \big)\Big]^{vir_\pi}\right) =
n\epsilon_*\left(
\Big[\Mbar^\bullet_{h}\big(\pi, \mu^1,\ldots ,\mu^r,
(1^{n})\big)\Big]^{vir_\pi}\right)\,
\end{equation*}
obtained from the
degeneration formula
by universally bubbling off the
image of the marking $1$ in the target.
\end{proof}

Finally, since $\widetilde{\Lambda}_{0,3}$
is nothing more than the Gromov-Witten theory of the local 3-fold $\com^2 \otimes\proj^1$ with 3 relative divisors
appearing in the original triangle
of equivalence of Section \ref{lifttr},
the property
$$\widetilde{\Lambda}_{0,3}(\mu^1\otimes \mu^2 \otimes \b1) = \widetilde{\eta}(v_1,v_2)\ $$
holds.

\subsection{The divisor equation}
The divisor equation for the CohFT  $\widetilde{\Lambda}$ will
play an important role.

\begin{prop}\label{g332}
The divisor equation holds:
\begin{equation*}
 \widetilde{\Lambda}
 _{g,r+1}
 \left(\mu^1 \otimes \cdots \otimes \mu^r\otimes (2,1^{n-2})\right)=
 \frac{\partial}{\partial u}
 \widetilde{\Lambda}
 _{g,r}
 (\mu^1 \otimes \cdots \otimes\mu^r)\, .
 \end{equation*} 
\end{prop}

\begin{proof}
Proposition \ref{g332} is equivalent to the following statement:
\begin{equation}\label{k23k}
p_*\epsilon_*\left(
\Big[\Mbar^\bullet_{h}\big(\pi, \mu^1,\ldots ,\mu^r,
(2,1^{n-2})\big)\Big]^{vir_\pi}\right) =
b[h]\cdot \epsilon_*\left(
\Big[\Mbar^\bullet_{h}\big(\pi, \mu^1,\ldots ,\mu^r\big)\Big]^{vir_\pi}\right)\, ,
\end{equation}
where $b[h]$ is defined by the relation
$$2h-2= b[h]+n(2g-2+r)- \sum_{i=1}^r \ell(\mu^i)\, .$$
In order to prove \eqref{k23k}, consider first the
standard dilaton equation
\begin{equation}\label{k23kk}
\epsilon_*p_{\pi*}\left(
\psi_1 \Big[\Mbar^\bullet_{h,1}\big(\pi, \mu^1,\ldots ,\mu^r \big)\Big]^{vir_\pi}\right) =
\big(2h-2+  \sum_{i=1}^r \ell(\mu^i)\big) \cdot
 \epsilon_*\left(
\Big[\Mbar^\bullet_{h}\big(\pi, \mu^1,\ldots ,\mu^r\big)\Big]^{vir_\pi}\right)\, .
\end{equation}
Here, $\psi_1$ is the cotangent line on the domain curve
of the map. Equation \eqref{k23k} then follows from 
\eqref{k23kk} and the
degeneration relation
\begin{multline*}
\epsilon_{1*}\left(
\psi_1 \Big[\Mbar^\bullet_{h,1}\big(\pi, \mu^1,\ldots ,\mu^r \big)\Big]^{vir_\pi}\right) =\\
\epsilon_*\left(
\Big[\Mbar^\bullet_{h}\big(\pi, \mu^1,\ldots ,\mu^r,
(2,1^{n-2})\big)\Big]^{vir_\pi}\right) \\
+n\psi \cdot \epsilon_*\left(
\Big[\Mbar^\bullet_{h}\big(\pi, \mu^1,\ldots ,\mu^r, (1^n) \big)\Big]^{vir_\pi}\right)\, 
\end{multline*}
obtained 
by universally bubbling off the
image of the marking $1$ in the target.
In the second term on the right in
the degeneration relation,
 $\psi$ is the cotangent line
at relative marking associated
to the boundary condition $(1^n)$.
There is no higher genus distribution
to the bubble as can be seen
from calculation of the equivariant cap.{\footnote{The 
equivariant cap invariant
for $(1^n)$
has been calculated in \cite[Section 2.5]{PPP}. The crucial
point is that 
$s_3$ occurs 
{\em only}
in the leading $u^{-2}$
summand. The leading summand
contributes 
the right most
term in the
above
degeneration formula.
All
other summands
of the equivariant
cap invariant
for $(1^n)$
have vanishing contributions.}}

We now apply $p_*$ to the degeneration
relation. Since $\epsilon p_\pi= p \epsilon_1$, we have
$$\epsilon_*p_{\pi*}\left(
\psi_1 \Big[\Mbar^\bullet_{h,1}\big(\pi, \mu^1,\ldots ,\mu^r \big)\Big]^{vir_\pi}\right) =
p_{*}\epsilon_{1*}\left(
\psi_1 \Big[\Mbar^\bullet_{h,1}\big(\pi, \mu^1,\ldots ,\mu^r \big)\Big]^{vir_\pi}\right)\, .$$
Using the identity axiom, we see
$$p_*\left(n\psi \cdot \epsilon_*\left(
\Big[\Mbar^\bullet_{h}\big(\pi, \mu^1,\ldots ,\mu^r, (1^n) \big)\Big]^{vir_\pi}\right)\right)=
n(2g-2+r) \cdot \epsilon_*\left(
\Big[\Mbar^\bullet_{h}\big(\pi, \mu^1,\ldots ,\mu^r \big)\Big]^{vir_\pi}\right)\, .$$
Equation \eqref{k23k} then follows.
\end{proof}

\subsection{Equivalence}
The CohFT $\widetilde{\Omega}$ obtained from
the orbifold
Gromov-Witten theory of $\Sym(\mathbb{C}^2)$
and the CohFT $\widetilde{\Lambda}$ are
equal in genus 0 (and hence both 
semisimple) by \cite{bg,bp}.
The two CohFTs
satisfy identical divisor equations.
The equality after restriction to $u=0$,
$$\widetilde{\Omega}_{g,r}(\mu^1,\ldots,\mu^r)|_{u=0} = 
\widetilde{\Lambda}_{g,r}(\mu^1,\ldots,\mu^r)|_{u=0}\, , \ \ \ 
g\geq 0\, , \ \
\mu^1, ...,\mu^r\in \text{Part}(n)\, ,$$
follows from a simple matching of moduli
spaces and obstruction theories.
By Theorem \ref{4444}, we conclude the following
equivalence.

\begin{prop} We have \label{pp12}
$\widetilde{\Omega}=\widetilde{\Lambda}$.
\end{prop}

\section{Local theories of curves: stable pairs}

\subsection{Stable pairs}

We define a CohFT ${\Lambda}$
via the moduli space of $\pi$-relative $\T$-equivariant stable pairs on the universal local
curve, 
\begin{equation}\label{fddff2}
\pi: \mathbb{C}^2\times \mathcal{C} \rightarrow \overline{\mathcal{M}}_{g,r}\, ,
\end{equation}
 based on the 
algebra
\begin{equation*}
{\mathsf{A}}= \mathbb{Q}(t_1,t_2)[[q]]
\end{equation*}
and the corresponding free module
\begin{equation*}
 {\V} = \mathcal{F}^n \otimes_ {\mathbb{Q}} {\mathsf{A}}\, ,
\end{equation*}
where $\mathcal{F}^n$ is the Fock space with basis
indexed by $\text{Part}(n)$.
The inner product for the CohFT 
is 
$${\eta}(\mu,\nu)
=\frac{(-1)^{|\mu|-\ell(\mu)}}{(t_1 t_2)^{\ell(\mu)}} 
\frac{\delta_{\mu\nu}}{\zz(\mu)}\,. $$
The algebra, free module, and inner product for ${\Lambda}$
are exactly the same as for the CohFT ${\Omega}$ obtained
from Gromov-Witten theory of ${\Hilb}$.

Let $\mu^1,\ldots,\mu^r \in \text{Part}(n)$, and let
 $\mathsf{P}_k(\pi,\mu^1,\ldots,\mu^r)$ be the moduli space
of stable pairs
on the fibers of $\pi$,
$$\epsilon: \mathsf{P}_k(\pi,\mu^1,\ldots, \mu^r) \rightarrow \Mbar_{g,r}\, .$$
The fiber of $\epsilon$ over 
the moduli point
$$(C,p_1,\ldots,p_r) \in \Mbar_{g,r}\, $$
is the moduli space of stable pairs of Euler characteristic $k$
on $\com^2 \times C$ relative to the divisors
determined by the nodes and 
the markings of $C$
with boundary{\footnote{The cohomology weights of the boundary condition
are all the identity class.}}
condition
$\mu^i$ over 
the divisor $\com^2 \times p_i$.
The moduli space $\mathsf{P}_k(\pi,\mu^1,\ldots,\mu^r)$ has $\pi$-relative
virtual dimension
$$-n(2g-2+r)+\sum_{i=1}^r \ell(\mu^i)\, .$$

The CohFT ${\Lambda}$ is defined
via the $\pi$-relative 
$\T$-equivariant virtual class by
$${\Lambda}_{g,r}
(\mu^1\otimes \ldots\otimes \mu^r)
=\sum_{d\geq 0}
q^d\epsilon_*\left(
\Big[\mathsf{P}_{d+n(1-g)}\big(\pi, \mu^1,\ldots ,\mu^r \big)\Big]^{vir_\pi}\right). $$

\subsection{Axioms
and the divisor equation}
The defining axioms of a CohFT are listed
in Section \ref{ccfftt}. Axiom (i)
for ${\Lambda}$ follows
from the symmetry of the construction.
Axiom (ii) is a consequence of the
degeneration formula of relative
stable pairs theory.
The proof of Axiom (iii)
is parallel to the proof of Proposition
\ref{00qq}.

\begin{prop}
The identity axiom holds: 
$$
{\Lambda}_{g,r+1}(\mu^1 \otimes \cdots \otimes \mu^r \otimes \b1) = p^*{\Lambda}_{g,r} (\mu^1 \otimes \cdots \otimes \mu^r)\, 
$$
where $p: \oM_{g,r+1} \to \oM_{g,r}$ is the forgetful map.
\end{prop}

\noindent
Instead of using the 1-pointed stable
map space, we consider the universal
target 
\begin{equation}\label{ft999}
p_\pi:\mathcal{T}_k\rightarrow \mathsf{P}_k(\pi, \mu^1,\ldots,\mu^r)
\end{equation}
with the universal curve expressed
as the descendent $\text{ch}_2(\mathcal{F})$ of the universal sheaf $\mathcal{F}$ on $\mathcal{T}_k$. Otherwise, the proof is identical.

The divisor equation for $\Lambda$ is
identical to the divisor equation
for the Gromov-Witten theory of 
$\Hilb$. The proof is parallel to the proof of 
Proposition \ref{g332}.

\begin{prop}\label{x332}
The divisor equation holds: 
$$
-{\Lambda}_{g,r+1}(\mu^1 \otimes \cdots \otimes \mu^r \otimes (2,1^{n-2})) = q\frac{\partial}{\partial q} {\Lambda}_{g,r} (\mu^1 \otimes \cdots \otimes \mu^r)\, .
$$
\end{prop}

\begin{proof}
Instead of using the 1-pointed stable
map space with $\psi_1$, we consider the universal
target \eqref{ft999} 
with 
the descendent $\text{ch}_3(\mathcal{F})$ of the universal sheaf $\mathcal{F}$ on $\mathcal{T}_k$. The notation here will be parallel
to the stable maps case,
$$\epsilon_1:\mathcal{T}_k \rightarrow \Mbar_{g,r+1}\,
.$$

Proposition \ref{x332} is equivalent to the following statement:
\begin{equation*}
p_*\epsilon_*\left(
\Big[\mathsf{P}_{d+n(1-g)}\big(\pi, \mu^1,\ldots ,\mu^r,
(2,1^{n-2})\big)\Big]^{vir_\pi}\right) =
d\cdot \epsilon_*\left(
\Big[\mathsf{P}_{d+n(1-g)}\big(\pi, \mu^1,\ldots ,\mu^r\big)\Big]^{vir_\pi}\right)\, .
\end{equation*}
Consider first the
 dilaton equation{\footnote{See \cite[Section 3.2]{PRP}.}} in the 
 theory of stable pairs
\begin{equation*}
\epsilon_*p_{\pi*}\left(
\ch_3(\mathcal{F}) \Big[\mathcal{T}_{d+n(1-g)}\Big]^{vir_\pi}\right) =
\left(d+n(1-g)-\frac{n}{2}(2-2g)\right)
\cdot
 \epsilon_*\left(
\Big[\mathsf{P}_{d+n(1-g)}\big(\pi, \mu^1,\ldots ,\mu^r\big)\Big]^{vir_\pi}\right)\, .
\end{equation*}
We have also the
degeneration relation
$$
\epsilon_{1*}\left(
\text{ch}_3(\mathcal{F}) \Big[\mathcal{T}_{d+n(1-g)}\Big]^{vir_\pi}\right) =
-\epsilon_*\left(
\Big[\mathsf{P}_{d+n(1-g)}\big(\pi, \mu^1,\ldots ,\mu^r,
(2,1^{n-2})\big)\Big]^{vir_\pi}\right) $$
obtained 
by universally bubbling off the
descendent in the target.
The contribution of
the boundary
condition $(1^n)$
vanishes as can be seen
from calculation of the equivariant cap.{\footnote{The equivariant cap invariant
for $(1^n)$
has been calculated in \cite[Section 2.5]{PPP}. For stable pairs, 
$s_3$ does not occur.
All
 summands
have vanishing contributions.}}
We conclude
\begin{eqnarray*}
p_*\epsilon_*\left(
\Big[\mathsf{P}_{d+n(1-g)}\big(\pi, \mu^1,\ldots ,\mu^r,
(2,1^{n-2})\big)\Big]^{vir_\pi}\right) &=&
p_*\epsilon_{1*}\left(
\text{ch}_3(\mathcal{F}) \Big[\mathcal{T}_{d+n(1-g)}\Big]^{vir_\pi}\right) \\
& = &
\epsilon_{*}p_{\pi*} \left(
\text{ch}_3(\mathcal{F}) \Big[\mathcal{T}_{d+n(1-g)}\Big]^{vir_\pi}\right)\\& =& 
d\cdot \epsilon_*\left(
\Big[\mathsf{P}_{d+n(1-g)}\big(\pi, \mu^1,\ldots ,\mu^r\big)\Big]^{vir_\pi}\right)\, ,
\end{eqnarray*}
completing the proof of the divisor equation.
\end{proof}

\subsection{Equivalence}
The CohFT ${\Omega}$ obtained from
the 
Gromov-Witten theory of $\Hilb$
and the CohFT ${\Lambda}$ are
equal in genus 0 (and hence both 
semisimple) by \cite{op29,op2}.
The two CohFTs
satisfy identical divisor equations.
The equality after restriction to $q=0$,
$${\Omega}_{g,r}(\mu^1,\ldots,\mu^r)|_{q=0} = 
{\Lambda}_{g,r}(\mu^1,\ldots,\mu^r)|_{q=0}\, , \ \ \ 
g\geq 0\, , \ \
\mu^1, ...,\mu^r\in \text{Part}(n)\, ,$$
follows from a simple matching of moduli
spaces and obstruction theories.
By Theorem \ref{ff11}, we conclude the following
equivalence.

\begin{prop} \label{pp15} We have
${\Omega}={\Lambda}$.
\end{prop}

Since the matching of obstruction theories
here in the $q=0$ case is interesting, we discuss the
matter in more detail.
The moduli spaces
$$\prod_{i=1}^r \text{ev}_i^{-1}(\mu^i) \subset \overline{\mathcal{M}}_{g,r}(\Hilb,0) 
\ \ \ \ \ \text{and} \ \ \ \ \
\mathsf{P}_{n(1-g)}(\pi, \mu^1, \ldots, \mu^r)$$
are isomorphic. The $\mathsf{T}$-equivariant 
obstruction bundle for the Gromov-Witten theory
of $\Hilb$ in the $q=0$ case is
\begin{equation}\label{zz88}
\mathsf{Tan}(\Hilb) \otimes \mathbb{E}^*_g\, .
\end{equation}
The $\mathsf{T}$-equivariant 
obstruction bundle{\footnote{The calculation
is by Serre duality and standard techniques.}} for $\pi$-relative stable pairs theory 
in the $q=0$ case is
\begin{equation}
\label{zz888}
\mathsf{Tan}(\Hilb)^* \otimes 
\mathbb{E}^*_g \otimes [t_1+t_2]\, , 
\end{equation}
where $[t_1+t_2]$ denotes a trivial bundle
of weight $t_1+t_2$.
The matching of \eqref{zz88} and \eqref{zz888} follows
from the holomorphic symplectic condition which
takes the following $\mathsf{T}$-equivariant form:
$$\mathsf{Tan}(\Hilb)^*  =
\mathsf{Tan}(\Hilb) \otimes [-t_1-t_2]\, .$$

\section{The restricted CohFTs}

\subsection{The Hilbert scheme $\Hilb$}\label{subsec:hilb_deg0}
Following \cite[Chapter VI, Section 10]{mac}, let
$$J_\lambda\in \mathcal{F}\otimes_\com \com(t_1,t_2)$$ be 
an integral form of the Jack symmetric function depending on the parameter $\alpha=1/\theta$. Define $$\mathsf{J}^\lambda=t_2^{|\lambda|}t_1^{\ell(\cdot )}J_\lambda|_{\alpha=-t_1/t_2}\, .$$
The vector
$\mathsf{J}^\lambda$ in Fock space corresponds to the $\T$-equivariant class
of the $\T$-fixed point of $\Hilb$ associated to $\lambda$. See also
\cite[Section 2.2]{op29}.

The Bernoulli numbers $B_{m}$ are defined by 
\begin{equation}\label{bern}
\frac{t}{e^t-1}=\sum_{m=0}^\infty \frac{B_m}{m!}t^m\, .
\end{equation}
For $\lambda \in \text{Part}(n)$, define 
\begin{equation*}
N_{2m-1, \lambda}(t_1,t_2)= \sum_{s\in \mathsf{D}_\lambda} \left(\frac{1}{(a(s)t_2-(l(s)+1)t_1)^{2m-1}}+\frac{1}{(l(s)t_1-(a(s)+1)t_2)^{2m-1}}\right)\, .
\end{equation*}
Here, the sum is
over all boxes $s$
in 
the Young diagram
$\mathsf{D}_\lambda$  corresponding to $\lambda$.
The standard
arm and leg lengths
are
$a(s)$ and $l(s)$
respectively.

\begin{prop}\label{prophilb}
After the coordinates
of $\mathsf{V}$ are set to 0,
the matrix $\mathsf{R}^\mathsf{Hilb}\Big|_{q=0}$ in the basis $$\Big\{\, 
\ \mathsf{J}^\lambda\, \Big|\, \lambda\in \text{\em Part}(n)\, \Big\}$$ of\, $\mathsf{V}$ is diagonal with entries 
\begin{equation*}
\exp\left(\sum_{m>0}\frac{B_{2m}}{2m(2m-1)}z^{2m-1}N_{2m-1, \lambda}(t_1,t_2) \right)\, .
\end{equation*}
\end{prop}

\begin{proof}
The $\mathsf{R}$-matrix $\mathsf{R}^\mathsf{Hilb}\Big|_{q=0}$ is the $\mathsf{R}$-matrix classifying the CohFT $\Omega\big|_{q=0}$. The CohFT $\Omega\big|_{q=0}$ 
concerns
the $\mathsf{T}$-equivariant Gromov-Witten theory of $\Hilb$
in degree 0.
The $\mathsf{T}$-fixed locus $$\Hilb^\mathsf{T}\subset \Hilb$$ is a set of isolated points indexed by partitions of $n$ (see, for example, \cite[Section 5.2]{Nak}).  A partition $$\lambda=(\lambda_1, \lambda_2, \lambda_3, \ldots)\in \text{Part}(n)$$ defines a Young diagram $\mathsf{D}_\lambda$ whose $i$-th row has $\lambda_i$ boxes. The point in $\Hilb^\mathsf{T}$ indexed by $\lambda$ is defined by the monomial ideal $\mathcal{I}_\lambda\subset \mathbb{C}[x,y]$ generated by $$(y^{\lambda_1}, xy^{\lambda_2}, x^2y^{\lambda_3},\ldots, x^{i-1}y^{\lambda_i},\ldots)\, .$$
The Young diagram $\mathsf{D}_\lambda$
also determines the conjugate partition $$\lambda'=(\lambda_1', \lambda_2',\lambda'_3,\ldots)$$ obtained by setting $\lambda_i'$ to be the number of boxes in the $i$-th column of $\mathsf{D}_\lambda$. 

For $s=(i,j)\in \mathsf{D}_\lambda$, the box on the $i^{th}$ row and the $j^{th}$ column, define the arm
and leg lengths by $$
a(s)=\lambda_j'-i\ \ \ \ \text{and} \ \ \ \
l(s)=\lambda_i-j$$ 
respectively.
By \cite[Proposition 5.8]{Nak}, the 
weights
associated
to $s\in \mathsf{D}_\lambda$
of the $\mathsf{T}$-action on the tangent space $\text{Tan}_{\lambda}(\hilbnc)$ are
\begin{equation}\label{eqn:tangent_weights}
-a(s)t_2+(l(s)+1)t_1\,, \ \ \ -l(s)t_1+(a(s)+1)t_2\, .
\end{equation}
As $s$ varies in 
$\mathsf{D}_\lambda$,
we obtain $2n$ tangent
weights.

By virtual localization, the  $\mathsf{T}$-equivariant Gromov-Witten theory of $\Hilb$ 
in degree 0 
equals  the Gromov-Witten theory of $\Hilb^\mathsf{T}$ twisted\footnote{In the sense of \cite{cg}.} by the inverse $\mathsf{T}$-equivariant Euler class of the rank $2n$ bundle on which $\mathsf{T}$ acts with weights (\ref{eqn:tangent_weights}).

The $\mathsf{R}$-matrix $\mathsf{R}^\mathsf{Hilb}\Big|_{q=0}$ is the unique matrix which transforms the Gromov-Witten theory of 
$\Hilb^\mathsf{T}$ to the aforementioned twisted Gromov-Witten theory of $\Hilb^\mathsf{T}$. The latter theory consists of Hodge integrals over $\overline{\mathcal{M}}_{g,r}$, which have been calculated in \cite{FaP,mumford}. To identify the $\mathsf{R}$-matrix, we consider the twisted theory in the setting of \cite{cg} and apply the quantum Riemann-Roch theorem.  Then, $\mathsf{R}^\mathsf{Hilb}\Big|_{q=0}$ coincides with the symplectic transformation in the quantum Riemann-Roch theorem\footnote{Since $\mathsf{R}^\mathsf{Hilb}\Big|_{q=0}=\text{Id}+O(z)$, we may discard the scalar factors in quantum Riemann-Roch theorem.}, which is 
the claim of the
the Proposition.
\end{proof}

\subsection{The symmetric product $\Sym(\mathbb{C}^2)$}\label{subsec:sym_0}

For $\Sym(\CC^2)$, the
vector
$|\mu\rangle\in \widetilde{\mathsf{V}}$ in Fock space corresponds
to the fundamental
class $[I_\mu]$ of the
component
$$I_\mu \subset
I\Sym(\CC^2)$$
of the inertial stack.

\begin{prop}\label{propsym}
After the coordinates of $\widetilde{\mathsf{V}}$ are set to 0, the matrix $\mathsf{R}^\mathsf{Sym}\Big|_{u=0}$ in the basis 
$$\Big\{\,|\mu\rangle\, \Big|\, \mu\in \text{\em Part}(n)\, \Big\}$$ 
of \, $\widetilde{\mathsf{V}}$
is diagonal with entries 
\begin{equation*}
\exp\left(-\sum_{m>0}\frac{B_{2m}}{2m(2m-1)} \sum_{i=1}^{\ell(\mu)}\left(\frac{1}{(\mu_it_1)^{2m-1}}+\frac{1}{(\mu_it_2)^{2m-1}}\right)z^{2m-1} \right).
\end{equation*}
\end{prop}

\begin{proof}
The $\mathsf{R}$-matrix $\mathsf{R}^\mathsf{Sym}\Big|_{u=0}$ classifies the restricted CohFT $\widetilde{\Omega}_{g,r}\Big|_{u=0}$. By   definition, the CohFT $\widetilde{\Omega}_{g,r}\Big|_{u=0}$ is obtained
from  the $\mathsf{T}$-equivariant Gromov-Witten theory of $\text{Sym}^n(\mathbb{C}^2)$ -- the Gromov-Witten theory of the classifying orbifold $BS_n$ twisted\footnote{In the sense of \cite{Tseng}.} by the inverse $\mathsf{T}$-equivariant Euler class of the rank $2n$ trivial bundle on which $\mathsf{T}$ acts with weights $t_1$ and $t_2$ (each appearing $n$ times). By the orbifold quantum Riemann-Roch theorem of \cite{Tseng}, the twisted
theory is obtained from the Gromov-Witten theory of $BS_n$ by the action of a symplectic operator $\mathsf{Q}$. The operator $\mathsf{Q}$ coincides with $\mathsf{R}^\text{Sym}\Big|_{u=0}$ in the basis $\{[I_\mu]\}_{\mu\in \text{Part}(n)}$ of $$H^*_\mathsf{T}(IBS_n)=H^*_\mathsf{T}(I\Sym(\mathbb{C}^2))\, .$$

Next, we identify $\mathsf{Q}$. Consider the conjugacy class 
$\text{Conj}(\mu)$ of $S_n$ corresponding to the partition $$\mu=(\mu_1,\mu_2,\mu_3,\ldots)\, .$$ Let $\sigma\in S_n$ 
be an element of $\text{Conj}(\mu)$.
Then,
$\sigma$ can be written as a product of disjoint cycles of lengths $\mu_i$. The vector space $(\mathbb{C}^2)^{\oplus n}$ decomposes into a direct sum of $\sigma$-eigenspaces. The eigenvalues and ranks of these eigenspaces depend only on the conjugacy class, not $\sigma$. The eigenvalues are $$\exp\left(2\pi\sqrt{-1}\frac{l}{\mu_i}\right), \quad 0\leq l\leq \mu_i-1,\quad i=1,2,3, \ldots .$$ Each 
such{\footnote{If $\mu$
has equal parts,
the eigenspaces
increase by the multiplicity factor.}}
eigenvalue has a rank $2$ eigenspace with $\T$-weights $t_1, t_2$. 

By \cite{Tseng}, the operator $\mathsf{Q}$ is diagonal in the basis $\{ [I_\mu]\}$ of $H^*_\mathsf{T}(IBS_n)$ 
with entry 
at $[I_\mu]$ given by
\begin{equation*}
\exp\left(\sum_{m>1}\frac{-1}{m(m-1)}\sum_{i=1}^{\ell(\mu)} \left\{B_m(0)+B_m\left(\frac{1}{\mu_i}\right)+
\ldots+B_m\left(\frac{\mu_i-1}{\mu_i}\right)\right\}z^{m-1}\left(\frac{1}{t_1^{m-1}}+\frac{1}{t_2^{m-1}}\right) \right).
\end{equation*}
Here, $B_m(x)$ is the
Bernoulli polynomial,
$$\frac{te^{xt}}{e^t-1}=\sum_{m=0}^\infty B_m(x)\frac{t^m}{m!}\, .$$
From the identity $$\sum_{l=0}^{r-1}\frac{te^{tl/r}}{e^t-1}=\frac{t}{e^t-1}\frac{e^t-1}{e^{t/r}-1}=\frac{t/r}{e^{t/r}-1}\cdot r,$$
and the definition \eqref{bern} of the Bernoulli
numbers, we see
$$\sum_{l=0}^{r-1}B_m
\left(\frac{l}{r}\right)=\frac{B_m}{r^{m-1}}\, .$$
The above
expression for
the diagonal elements
of $\mathsf{Q}$ 
can be written as
\begin{equation*}
\exp\left(-\sum_{m>0}\frac{B_{2m}}{2m(2m-1)} \sum_{i=1}^{\ell(\mu)}\left(\frac{1}{(\mu_it_1)^{2m-1}}+\frac{1}{(\mu_it_2)^{2m-1}}\right)z^{2m-1} \right)
\end{equation*}
which completes the proof.
\end{proof}

The product structure on $H_{\mathsf{T}}^*(I\text{Sym}^n\mathbb{C}^2)$ is described in \cite[Section 3.3]{bg}
in terms of the representation
theory of the symmetric group.
The {\em normalized} idempotents of $H_{\mathsf{T}}^*(I\text{Sym}^n\mathbb{C}^2)$ are 
\begin{equation}\label{eqn:idemp}
\mathsf{I}^\lambda=\sum_{\mu\in \text{Part(n)}}\chi_{\lambda}(\mu) I_\mu (t_1t_2)^{\ell(\mu)/2},
\end{equation}
where $\chi_{\lambda}(\mu)$ is the character of $S_n$.
The action of the matrix $\mathsf{R}^\text{Sym}\Big|_{u=0}$ in the idempotent basis of $H_{\mathsf{T}}^*(I\text{Sym}^n\mathbb{C}^2)$ is given by  
\begin{multline}
\mathsf{R}^\mathsf{Sym}\Big|_{u=0}(\mathsf{I}^\lambda)=\\
\sum_{\mu}\chi_{\lambda}(\mu) \exp\left(-\sum_{m>0}\frac{B_{2m}}{2m(2m-1)} \sum_{i=1}^{\ell(\mu)}\left(\frac{1}{(\mu_it_1)^{2m-1}}+\frac{1}{(\mu_it_2)^{2m-1}}\right)z^{2m-1} \right)
 (t_1t_2)^{\ell(\mu)/2}
 \, \Big|\mu\Big\rangle
 \, 
\end{multline}
written in $\widetilde{\mathsf{V}}$. Under the identification{\footnote{We
use $\sqrt{-1}$ here instead of $i$ for clarity in the
formulas (since $i$
also occurs as an index of summation).}} $\mathsf{V}\to \widetilde{\mathsf{V}}$ given by (\ref{eqn:mu_tilde}), we have
\begin{multline}\label{eqn:R_sym0}
\mathsf{R}^\mathsf{Sym}\Big|_{u=0}(\mathsf{I}^\lambda)=\\
\sum_{\mu}\chi_{\lambda}(\mu) \exp\left(-\sum_{m>0}\frac{B_{2m}}{2m(2m-1)} \sum_{i=1}^{\ell(\mu)}\left(\frac{1}{(\mu_it_1)^{2m-1}}+\frac{1}{(\mu_it_2)^{2m-1}}\right)z^{2m-1} \right)\\
\vspace{5pt}
\cdot
\sqrt{-1}^{\ell(\mu)-|\mu|} (t_1t_2)^{\ell(\mu)/2}\,
\Big|\mu\Big\rangle
\,
\end{multline}
written in $\mathsf{V}$.

\section{Quantum differential equations}
We recall the quantum differential equation for $\hilbnc$ calculated in \cite{op} and further studied in \cite{op29}. We follow here the exposition \cite{op,op29}.

Consider the Fock space
introduced
in Section \ref{fsf}
(after extension
of scalars to $\com$),
\begin{equation}\label{ffkk}
\cF \otimes _{\mathbb C} {\mathbb C}[t_1,t_2]\stackrel{\sim}{=} 
\bigoplus_{n\geq 0} H_{\T}^*(\Hilb,{\mathbb C})\, ,
\end{equation}
 freely generated by commuting creation operators $\alpha_{-k}$
 for $k\in \mathbb{Z}_{>0}$ acting on the vacuum vector $v_{\emptyset}$. 
 The intersection pairing on the $\T$-equivariant cohomology of $\hilbnc$ induces a pairing
 on Fock space,
  $$\eta( \mu,\nu)=\frac{(-1)^{|\mu|-\ell(\mu)}}{(t_1t_2)^{\ell(\mu)}}\frac{\delta_{\mu\nu}}{\mathfrak{z}(\mu)}\, .$$
  The paper \cite{op29} also uses a Hermitian pairing $\<-,-\>_{H}$ on the Fock space \eqref{ffkk} defined by the three following
  properties
\begin{enumerate}
\item[$\bullet$]
$\langle \mu|\nu\rangle_{H}=\frac{1}{(t_1t_2)^{\ell(\mu)}}\frac{\delta_{\mu\nu}}{\mathfrak{z}(\mu)}$ ,

\vspace{7pt}
\item[$\bullet$]
$\langle a f, g\rangle_H=a\langle f,g\rangle_H, \quad a\in \mathbb{C}(t_1, t_2)$,

\vspace{7pt}
\item[$\bullet$]
$\langle f,g\rangle_H=\overline
{\langle g,f\rangle}_H, \text{ where } \overline{a(t_1,t_2)}=a(-t_1, -t_2)$ .
\end{enumerate}

The quantum differential equation (QDE) for the Hilbert schemes of points on $\mathbb{C}^2$ is
given by
\begin{equation}\label{qde}
q\frac{d}{dq}\Phi=\mathsf{M}_D\Phi\, , \quad \Phi\in \mathcal{F}\otimes_\mathbb{C} \mathbb{C}(t_1, t_2),
\end{equation}
where $\mathsf{M}_D$
is the operator of quantum multiplication by $D$,
\begin{multline}\label{operator_M}
\mathsf{M}_D=
(t_1+t_2)\sum_{k>0}\frac{k}{2}\frac{(-q)^k+1}{(-q)^k-1}\alpha_{-k}\alpha_k
\, -\, \frac{t_1+t_2}{2}\frac{(-q)+1}{(-q)-1}|\cdot |
\\
+\frac{1}{2}\sum_{k,l>0}\Big[t_1t_2\alpha_{k+l}\alpha_{-k}\alpha_{-l}-\alpha_{-k-l}\alpha_k\alpha_l\Big]\, .
\end{multline}
While the quantum differential equation
\eqref{qde} has a regular singular point at $q=0$, the point $q=-1$ is regular.

The quantum differential equation considered in Givental's theory contains a parameter $z$. In the case of the Hilbert schemes of points on $\mathbb{C}^2$, the QDE with parameter $z$ is \begin{equation}\label{qde_z}
zq\frac{d}{dq}\Phi=\mathsf{M}_D\Phi, \quad \Phi\in \mathcal{F}\otimes_\mathbb{C} \mathbb{C}(t_1, t_2)\, .
\end{equation}
For $\Phi\in \mathcal{F}\otimes_\mathbb{C}\mathbb{C}(t_1,t_2)$, define 
$$\Phi_z=\Phi\left(\frac{t_1}{z}, \frac{t_2}{z}, q\right)\, .$$ Define $\Theta\in \text{Aut}(\mathcal{F})$ by $$\Theta|\mu\rangle=z^{\ell(\mu)}|\mu\rangle\, .$$
The following Proposition allows us to use the results in \cite{op29}.

\begin{prop}\label{ll44}
If $\Phi$ is a solution of (\ref{qde}), then $\Theta\Phi_z$ is a solution of (\ref{qde_z}).
\end{prop}

\noindent Proposition
\ref{ll44} follow immediately from the following direct computation.

\begin{lem}
For $k>0$, we have $\Theta\alpha_k=\frac{1}{z}\alpha_k\Theta$ and $\Theta\alpha_{-k}=z\alpha_{-k}\Theta$.
\end{lem}

\section{Solutions of the QDE}

\subsection{Preparations}
In what follows, we fix an integer $n\geq 1$ and
consider the solutions of the QDE for $\Hilb$.
The equivariant parameters $t_1, t_2$ are treated as complex numbers varying in a Euclidean open domain
of $\com$. We work with the energy $n$ subspace
of Fock space. For notational simplicity, we often omit $n$ from the formulas.

\subsection{At the origin}
At $q=0$, the operator $\mathsf{M}_D(0)$ has distinct eigenvalues $$-c(\lambda; t_1, t_2)=-\sum_{(i,j)\in \lambda}[(j-1)t_1+(i-1)t_2].$$ So $\mathsf{M}_D(q)$ has distinct eigenvalues for small for $q\in \com$
with $|q|$ small. Furthermore, the values of $q$ for which $\mathsf{M}_D(q)$ has repeated eigenvalues are roots of polynomial equations, so there are only finitely many such $q$. Therefore, we can find a path 
$$\gamma: [0,1]\to \mathbb{C}\,, \  \ \ \ \ \gamma(0)=0\, , \ \gamma(1)=-1\, ,$$ and an open neighborhood $\mathcal{U}\subset \mathbb{C}$ containing $\gamma$ such that
\begin{enumerate}
\item[(i)]
$\mathsf{M}_D(q)$ has distinct eigenvalues for $q\in \mathcal{U}$,
\item[(ii)]
As $q$ moves along $\gamma$ to $q=-1$, the function $q^{-c(\lambda; t_1,t_2)}$ is transported to $\exp(\pi\sqrt{-1}c(\lambda;t_1,t_2))$.
\end{enumerate}
For $q\in \mathcal{U}$, the eigenvalues of $\mathsf{M}_D(q)$, denoted by $v(\lambda; q)$, are analytic in $q$. Near $q=0$, we have a power series expansion $$v(\lambda; q)=-c(\lambda;t_1, t_2)+O(q)\,.$$
Therefore, the canonical coordinates $$u(\lambda;q)=\int v(\lambda; q) d\log q\, ,$$ are analytic in $q$. We choose the integration constants so that in $q\to 0$ we have the asymptotics
\begin{equation}\label{lem:asym_u}
u(\lambda;q)=-c(\lambda;t_1, t_2)\log q +O(q).
\end{equation}
\noindent Let $\mathsf{u}$ be the diagonal matrix with entries $u(\lambda; q)$.

\subsection{Formal solutions}
By \cite[Proposition 1.1]{g1},  we have
the following results (see also \cite{d} and \cite[Theorem 1]{lp}).
\begin{prop}\label{prop:asym_sol}
There is an open neighborhood $\mathcal{U}_0\subset \mathcal{U}$ of $0$ on which the QDE (\ref{qde_z}) has a formal solution of the form 
\begin{equation}\label{eqn:asym_sol}
\mathsf{R}^{\mathsf{Hilb}} e^{\mathsf{u}/z}\, .
\end{equation}
\end{prop}

In \cite[Proposition 1.1]{g1}, the asymptotics takes the form 
\begin{equation}\label{n55}
\Psi \mathsf{R} e^{\mathsf{u}/z}\end{equation}
where $\Psi(q)$ is the matrix{\footnote{In \cite[Theorem 1]{lp}, 
a different convention is followed: the same
transition matrix is denoted
$\Psi^{-1}$.}}
whose columns are length-$1$ eigenvectors of $\mathsf{M}_D(q)$.
In other word, $\Psi$ is the transition matrix from the canonical basis to the flat basis. 
In \eqref{n55}, the matrix $\mathsf{R}$
is in canonical coordinates.
Together $\Psi \mathsf{R}$ is the $\mathsf{R}$-matrix in
flat coordinates. 

After the restriction
$q=0$, canonical and flat 
coordinates are 
the given by the respective bases 
$$\Big\{\,\mathsf{J}^\lambda\, \Big|\, \lambda\in \text{Part}(n)\, \Big\}\ \ \ 
\text{and}\ \ \
\Big\{\,|\mu\rangle\, \Big|\, \mu\in \text{Part}(n)\, \Big\}$$ 
of Fock space.



The
 asymptotics in the $z\to 0$ limit
 will play a
 crucial role.
 For 
 our study,
 we must specify how $z$ approaches $0\in \com$. Let $$\mathfrak{R}\subset\mathbb{C}$$ be a ray emanating from $0$ satisfying
 the following four conditions:  
\begin{enumerate}
\item[$\bullet$]
For $z\in \mathfrak{R}$ or $z\in -\mathfrak{R}$,
\begin{equation}\label{t_condition_z}
\Big|\arg\left(\frac{t_1}{z}\right)\Big|< \pi,\quad  \Big|\arg\left(\frac{t_2}{z}\right)\Big|< \pi.
\end{equation}
\item[$\bullet$]
For $z\in \mathfrak{R}$ or $z\in -\mathfrak{R}$, and for any partition $\lambda$ of $n$, and $s\in \mathsf{D}_\lambda$, we have
\begin{equation}\label{eqn:w_condition_z}
\Big|\arg\left(\frac{(l(s)+1)t_1-a(s)t_2}{z}\right)\Big|< \pi,\quad  \Big|\arg\left(\frac{-l(s)t_1+(a(s)+1)t_2}{z}\right)\Big|< \pi.
\end{equation}
\item[$\bullet$]
For $z\in \mathfrak{R}$ or $z\in -\mathfrak{R}$, and for any two partitions $\lambda, \lambda'$ of $n$, we have
\begin{equation}\label{l_condition_z}
\arg\left(\frac{-c(\lambda; t_1, t_2)+c(\lambda';t_1, t_2)}{z}\right)\notin \frac{\pi}{2}+\mathbb{Z}\pi.
\end{equation}
\item[$\bullet$]
For $z\in \mathfrak{R}$,
\begin{equation}\label{real_condition_z}
\text{Re}(\sqrt{-1}t_1/z)<0, \quad \text{Re}(\sqrt{-1}t_2/z)>0.
\end{equation}
\end{enumerate}
For suitable $t_1, t_2$ varying in small
enough domains,
we can find $\mathfrak{R}\subset \com$ satisfying these conditions.

\subsection{Solutions}
We recall the solution of QDE (\ref{qde}) constructed in \cite{op29}. 

As in Section \ref{subsec:hilb_deg0},
let $J_\lambda\in \mathcal{F}\otimes_\com \com(t_1,t_2)$ be the integral form of the Jack symmetric function depending on the parameter $\alpha=1/\theta$ of \cite{mac,op29}. Then $$\mathsf{J}^\lambda=t_2^{|\lambda|}t_1^{\ell(\cdot )}J_\lambda|_{\alpha=-t_1/t_2}$$ is an eigenfunction of $\mathsf{M}_D(0)$ with eigenvalue $-c(\lambda;t_1,t_2)$. The coefficient of $$|\mu\rangle\in \mathcal{F}\otimes_\com \com(t_1,t_2)$$ in the expansion of $\mathsf{J}^\lambda$ is $(t_1t_2)^{\ell(\mu)}$ times a polynomial in $t_1$ and $t_2$ of degree $|\lambda|-\ell(\mu)$.

By  direct calculation, detailed in Section \ref{subsec:scalar_prod}, we find 
\begin{equation}\label{eqn:jack_pairings}
\<\mathsf{J}^\lambda, \mathsf{J}^\mu\>_H=\eta(\mathsf{J}^\lambda, \mathsf{J}^\mu)\, .
\end{equation}
Since $\mathsf{J}^\lambda$ corresponds to the $\T$-equivariant class
of the $\T$-fixed point of $\Hilb$ associated to $\lambda$,
\begin{equation}\label{eqn:jack_lengths}
||\mathsf{J}^\lambda||^2=||\mathsf{J}^\lambda||_H^2=\prod_{\mathsf{w}: \text{ tangent weights at $\lambda$}} \mathsf{w}\,
\end{equation}
see \cite{op29}. 
The tangent weights are given by \eqref{eqn:tangent_weights}.

There are solutions\footnote{See, for example, \cite[Chapter XIX]{ince} for a discussion of how these solutions are constructed.} to (\ref{qde}) of the form $$\mathsf{Y}^\lambda(q)q^{-c(\lambda;t_1,t_2)}, \quad \mathsf{Y}^\lambda(q)\in \mathcal{F}\otimes_\mathbb{C}\mathbb{C}(t_1, t_2)[[q]],$$ which converge for $|q|<1$ and satisfy $\mathsf{Y}^\lambda(0)=\mathsf{J}^\lambda$. 

By \cite[Corollary 1]{op29}, 
\begin{equation}\label{eqn:length_Y}
\langle \mathsf{Y}^\lambda(q), \mathsf{Y}^\mu(q)\rangle_H=\delta_{\lambda\mu} ||\mathsf{J}^\lambda||_H^2=\langle \mathsf{J}^\lambda, \mathsf{J}^\mu\rangle_H.
\end{equation}

As in \cite[Section 3.1.3]{op29}, let $\mathsf{Y}$ be the matrix whose column vectors are $\mathsf{Y}^\lambda$. Let $\mathsf{J}$ be the matrix whose column vectors are $\mathsf{J}^\lambda$. Let $\mathsf{G}_{\text{DT}}(t_1, t_2)$ be the diagonal matrix with eigenvalues $$q^{-c(\lambda;t_1,t_2)}\prod_{\mathsf{w}: \text{ tangent weights at $\lambda$}}\frac{1}{\Gamma(\mathsf{w}+1)}.$$
Define the following further diagonal matrices:
\begin{center}
\begin{tabular}[c]{|l|r|}
\hline
{Matrix} & {Eigenvalues} \\
\hline
$L$ & $z^{-|\lambda|}\prod_{\mathsf{w}: \text{ tangent weights at $\lambda$}} \mathsf{w}^{1/2}$\\
\hline
$L_0$ & $q^{-c(\lambda;t_1,t_2)/z}$\\
\hline
$A$ & $\prod_{\mathsf{w}: \text{ tangent weights at $\lambda$}}(\mathsf{w}/z)^{\mathsf{w}/z}e^{-\mathsf{w}/z}$\\
\hline
\end{tabular}
\end{center}

Consider the following solution to (\ref{qde_z}),
\begin{equation}\label{eqn:sol}
S=(2\pi)^{|\cdot|}\Theta\mathsf{Y}_z\mathsf{G}_{\text{DT}z}A\, ,
\end{equation}
where $S$ is defined over $\mathcal{U}$. As before,
$$\mathsf{Y}_z=\mathsf{Y}\left(\frac{t_1}{z}, \frac{t_2}{z}, q\right)\, , \ \ \ \
\mathsf{G}_{\text{DT}z}=\mathsf{G}_{\text{DT}}
\left(\frac{t_1}{z}, \frac{t_2}{z}, q\right)\, .$$

\begin{prop}\label{prop:asymp_S}
As $z\to 0$ along $\mathfrak{R}$, the operator $Se^{-\mathsf{u}/z}|_{q=0}$ has the asymptotics $$\mathsf{R}^{\text{\em Hilb}}\Big|_{q=0}.$$
\end{prop}
\begin{proof}
We write $S$ as
\begin{equation*}
\begin{split}
(2\pi)^{|\cdot|}\Theta\mathsf{Y}_z\mathsf{G}_{\text{DT}z}A&= (2\pi)^{|\cdot|}\Theta\mathsf{Y}_zL^{-1}LL_0L_0^{-1}\mathsf{G}_{\text{DT}z}A.\\
&=(\Theta\mathsf{Y}_zL^{-1})((2\pi)^{|\cdot|}LL_0^{-1}\mathsf{G}_{\text{DT}z}A)L_0.
\end{split}
\end{equation*}
At $q=0$, the columns of $\Theta\mathsf{Y}_zL^{-1}$ are 
\begin{eqnarray*}
\Theta \mathsf{Y}^\lambda\left(0;\frac{t_1}{z}, \frac{t_2}{z}\right)z^{|\lambda|}\prod_{\mathsf{w}: \text{ tangent weights at
$\lambda$}} \mathsf{w}^{-1/2}&
=&\Theta \mathsf{J}^\lambda\left(\frac{t_1}{z}, \frac{t_2}{z}\right)z^{|\lambda|}\prod_{\mathsf{w}: \text{ tangent weights
at $\lambda$}} \mathsf{w}^{-1/2}\\
&=&\mathsf{J}^\lambda(t_1, t_2)\prod_{\mathsf{w}: \text{ tangent weights at $\lambda$}} \mathsf{w}^{-1/2}.
\end{eqnarray*}
So $\Psi|_{q=0}=\Theta\mathsf{Y}_zL^{-1}$.
By \eqref{lem:asym_u}, $L_0e^{-\mathsf{u}/z}\Big|_{q=0}=\Id$,
where $\Id$ is the identity matrix.

It remains to calculate the asymptotics of $(2\pi)^{|\cdot|}LL_0^{-1}\mathsf{G}_{\text{DT}z}A$. Recall the Stirling asymptotics for Gamma function (see, for example, \cite{ww}):
\begin{equation}\label{eqn:stirling}
\frac{1}{\Gamma(x+1)}\sim \frac{x^{-1/2}x^{-x}e^x}{\sqrt{2\pi}}\exp\left(\sum_{m>0}\frac{B_{2m}}{2m(2m-1)}\left(\frac{-1}{x}\right)^{2m-1}\right), \quad  |x|\to \infty, |\text{arg}(x)|<\pi.
\end{equation}
By condition \eqref{eqn:w_condition_z}, 
the formula \eqref{eqn:stirling} is applicable to $\mathsf{G}_{\text{DT}}$ as $z\to 0$ along $\mathfrak{R}$. We 
conclude that the asymptotics of $L_0^{-1}\mathsf{G}_{\text{DT}z}$ is a diagonal matrix with eigenvalues 
\begin{multline*}
\prod_{\mathsf{w}: \text{ tangent weights at $\lambda$}} \frac{(\mathsf{w}/z)^{-1/2}(\mathsf{w}/z)^{-(\mathsf{w}/z)}e^{\mathsf{w}/z}}{\sqrt{2\pi}}\exp\left(\sum_{m>0}\frac{B_{2m}}{2m(2m-1)}\left(\frac{-1}{(\mathsf{w}/z)}\right)^{2m-1}\right)\\
=\frac{z^{|\lambda|}}{(2\pi)^{|\lambda|}}\prod_{\mathsf{w}: \text{ tangent weights at $\lambda$}} \mathsf{w}^{-1/2}(\mathsf{w}/z)^{-\mathsf{w}/z}e^{\mathsf{w}/z}\exp\left(\sum_{m>0}\frac{B_{2m}}{2m(2m-1)}\left(\frac{-z}{\mathsf{w}}\right)^{2m-1}\right).
\end{multline*} 
Therefore, $(2\pi)^{|\cdot|}LL_0^{-1}\mathsf{G}_{\text{DT}z}A$ has asymptotics given by a diagonal matrix with eigenvalues
\begin{equation*}
\prod_{\mathsf{w}: \text{ tangent weights at $\lambda$}} \exp\left(\sum_{m>0}\frac{B_{2m}}{2m(2m-1)}\left(\frac{-z}{\mathsf{w}}\right)^{2m-1}\right),
\end{equation*}
which coincides, by Proposition \ref{prophilb},  with $\mathsf{R}^{\text{Hilb}}\Big|_{q=0}$ written as a matrix
with both domain and range in the {canonical} basis.

Taken all together, $Se^{-\mathsf{u}/z}|_{q=0}$ is
$$(\Theta\mathsf{Y}_zL^{-1})((2\pi)^{|\cdot|}LL_0^{-1}\mathsf{G}_{\text{DT}z}A)L_0 \exp(-\mathsf{u}/z)|_{q=0}$$
and has the asymptotics $\mathsf{R}^{\mathsf{Hilb}}|_{q=0}$ written as a matrix with domain
in the canonical basis and
range 
in the {flat} basis.
\end{proof}

\subsection{Asymptotics of solutions}


\begin{prop}\label{prop:sol_w_asym}
For every $p\in \gamma$, there exists an open neighborhood $$\mathcal{U}_p\subset \mathcal{U}$$ 
on which the system (\ref{qde_z}) has a fundamental solution $\WWW$ satisfying the following
property: in the $z\to 0$ limit along $\mathfrak{R}$, $\WWW$ has  asymptotics of the form $$\WWW \sim  \WW e^{\mathsf{u}/z}$$ where $\WW=\Id+\WW_1z+\WW_2z^2+\ldots$ is an operator-valued $z$-series
with coefficients analytic in $q$.
\end{prop}
\begin{proof}
For $p\neq 0$, the result follows immediately from \cite[Theorem 26.3]{wasow}. More precisely, we use the change of variables $q=pe^{-x}$ to transform the system (\ref{qde_z}) into a system of the form $$z\frac{d}{dx}\Phi= M\Phi.$$ We use an analytical simplification, which exists according to \cite[Theorem 26.1]{wasow}, to transform the latter
system to a collection of $1$-dimensional ODEs of the form $$z\frac{df}{dx}=a(x, t_1, t_2)f\, ,$$
which can be easily solved. Combining the solutions of these ODEs with the analytical simplification gives the solution $\mathcal{S}$. The $z\to 0$ limit can be taken along $\mathfrak{R}$ because the proof of \cite[Theorem 26.1]{wasow} allows the $z\to 0$ limit to be taken along any direction.

For $p=0$ the above argument is still valid. The needed analytical simplification exists because of the condition (\ref{l_condition_z}) allows us to apply \cite[Theorem 3]{rs1}. We then transform (\ref{qde_z}) to a collection of $1$-dimensional ODEs of the form $$zq\frac{df}{dq}=a(q, t_1, t_2)f\, .$$ Combining the solutions of these ODEs with the analytical simplification gives the solution $\mathcal{S}$. By condition (\ref{l_condition_z}), the $z\to 0$ limit can be taken along $\mathfrak{R}$.
\end{proof}

\begin{prop}\label{prop:sol_w_asym2}
There exists an open neighborhood $\mathcal{U}'\subset \mathcal{U}$ of $\gamma$ on which the solution $S$ has the  asymptotics 
\begin{equation}\label{eqn:S_asymp_exp}
{\text{$S\sim \mathsf{R} e^{\mathsf{u}/z}$ as  $z\to 0$ along $\mathfrak{R}$,}}
\end{equation}
where $\mathsf{R}=\Id+\mathsf{R}_1z+\mathsf{R}_2z^2+...$ is an operator-valued $z$-series with coefficients analytic in $q$ and $\mathsf{R}=\mathsf{R}^{\mathsf{Hilb}}$ in a neighborhood of $q=0$. Moreover, $\mathsf{R}$ is symplectic,
$$\mathsf{R}^\dagger(-z)\mathsf{R}(z)=\Id\, $$
with the adjoint taken with respect to the pairing $\eta$.
\end{prop}
\begin{proof}
After intersection, we may assume the two open sets $\mathcal{U}_0$ in Propositions \ref{prop:asym_sol} and \ref{prop:sol_w_asym} are the same.

Consider the open sets $\mathcal{U}_q$ of Proposition \ref{prop:sol_w_asym}. Certainly $\gamma\subset \bigcup_{q\in \gamma} \mathcal{U}_q$. Since $\gamma$ is compact, there exist finitely
many points
$$t_0=0< t_1<t_2<...< t_k=1$$
for which $\gamma\subset \bigcup_{i=0}^k\mathcal{U}_{\gamma(t_i)}$. Define $$\mathcal{U}_i=\mathcal{U}_{\gamma(t_i)}.$$

Consider the solution $\WWW$ over $\mathcal{U}_0$ constructed in Proposition \ref{prop:sol_w_asym}. The asymptotics of $\WWW$ as $z\to 0$ along $\mathfrak{R}$, $\WW e^{\mathsf{u}/z}$ is an asymptotical solution of \eqref{qde_z}. We compare  $\WW e^{\mathsf{u}/z}$ 
with the asymptotical solution (\ref{eqn:asym_sol}). By 
\cite[Remark 1 after the proof of Proposition 1.1]{g1}, $\WW$ and $\mathsf{R}^{\mathsf{Hilb}}$ differ by a $z$-series with coefficients given by diagonal 
matrices\footnote{Here, all operators are represented by matrices in the canonical basis. Since both 
$e^{\mathsf{u}/z}$ and
$\left(\mathsf{R}^{\mathsf{Hilb}}\Big|_{q=0}\right)^{-1} \left(\WW\Big|_{q=0}\right)$ 
are diagonal, these matrices
commute.} 
independent of $q$. Therefore, $$\WW=\mathsf{R}^{\mathsf{Hilb}}\left(\mathsf{R}^{\mathsf{Hilb}}\Big|_{q=0}\right)^{-1} \left(\WW\Big|_{q=0}\right)$$ and $\left(\mathsf{R}^{\mathsf{Hilb}}\Big|_{q=0}\right)^{-1} \left(\WW\Big|_{q=0}\right)$ is diagonal.

The solution $S$ in (\ref{eqn:sol}) is a fundamental solution to (\ref{qde_z}). Hence, there exists a matrix $C$, independent of $q$, satisfying
$$SC=\WWW$$ on $\mathcal{U}_0$. The asymptotics of $S\Big|_{q=0}$ as $z\to 0$ along $\mathfrak{R}$ were calculated in Proposition \ref{prop:asymp_S}. 
After comparing with the asymptotics of $\WWW\Big|_{q=0}$, we find  $$C\sim \left(\mathsf{R}^{\mathsf{Hilb}}\Big|_{q=0}\right)^{-1} \left(\WW\Big|_{q=0}\right)$$ as $z\to 0$ along $\mathfrak{R}$. Therefore, on $\mathcal{U}_0$, we find 
\begin{equation*}
S=\WWW C^{-1}\ \ \sim \ \  \WW e^{\mathsf{u}/z} \left(\WW\Big|_{q=0}\right)^{-1}\left(\mathsf{R}^{\mathsf{Hilb}}\Big|_{q=0}\right)
=  \WW \left(\WW\Big|_{q=0}\right)^{-1}
\left(\mathsf{R}^{\mathsf{Hilb}}\Big|_{q=0}\right) e^{\mathsf{u}/z}
= \mathsf{R}^\mathsf{Hilb} e^{\mathsf{u}/z}
\end{equation*}
as $z\to 0$ along $\mathfrak{R}$.

Suppose now (\ref{eqn:S_asymp_exp}) is proven over $\mathcal{U}_{<l}=\bigcup_{i=0}^{l-1} \mathcal{U}_i$. Consider the solution $\WWW$ over $\mathcal{U}_l$ constructed in Proposition \ref{prop:sol_w_asym}. The asymptotics of $\WWW$ as $z\to 0$ along $\mathfrak{R}$, $\WW e^{\mathsf{u}/z}$, is an asymptotical solution of
\eqref{qde_z} on $\mathcal{U}_l$. We compare
$\WW e^{\mathsf{u}/z}$
with the asymptotical solution \eqref{eqn:S_asymp_exp} over $\mathcal{U}_{<l}\cap\mathcal{U}_l$. As before, 
$\WW$ and $\mathsf{R}$ differ by a $z$-series with coefficients diagonal matrices independent of $q$. Let $p_l\in  \mathcal{U}_{<l}\cap\mathcal{U}_l$. Then, over $\mathcal{U}_{<l}\cap\mathcal{U}_l$,
we have 
$$\WW=\mathsf{R}
\left(\mathsf{R}\Big|_{q=p_l}\right)^{-1} \left(\WW\Big|_{q=p_l}\right)\,
.$$ 
Moreover, $\left(R\Big|_{q=p_l}\right)^{-1} \left(\WW\Big|_{q=p_l}\right)$ is a diagonal matrix. As before, there exists a matrix $C$, independent of $q$, satisfying $$SC=\WWW$$ on 
$\mathcal{U}_{<l}\cap\mathcal{U}_l$. Comparing asymptotics at $q=p_l$, we find 
$$C\sim \left(\mathsf{R}\Big|_{q=p_l}\right)^{-1} \left(\WW\Big|_{q=p_l}\right)$$ as $z\to 0$ along $\mathfrak{R}$.

On $\mathcal{U}_l$, define $\mathsf{R}=\WW\left(\WW\Big|_{q=p_l}\right)^{-1}\left(\mathsf{R}\Big|_{q=p_l}\right)$.  Over $\mathcal{U}_{<l}\cap \mathcal{U}_l$, we have 
\begin{equation*}
S=\WWW C^{-1}\ \sim \  \WW e^{\mathsf{u}/z} \left(\WW\Big|_{q=p_l}\right)^{-1}\left(\mathsf{R}\Big|_{q=p_l}
\right)
=  \WW \left(\WW\Big|_{q=p_l}^{-1}\right)\left(\mathsf{R}\Big|_{q=p_l}\right)e^{\mathsf{u}/z}
=\mathsf{R} e^{\mathsf{u}/z} 
\end{equation*}
as $z\to 0$  along  $\mathfrak{R}$.
We have proven \eqref{eqn:S_asymp_exp} over $\mathcal{U}_{<l}\cup \mathcal{U}_l$.

Finally, we prove
the constructed $\mathsf{R}$ is symplectic\footnote{An easier way to see the symplectic property of $\mathsf{R}$ is the following: being symplectic is a closed condition in $q$, and $\mathsf{R}$ is symplectic on $\mathcal{U}_0$ because $\mathsf{R}$ is $\mathsf{R}^\text{Hilb}$ on $\mathcal{U}_0$.
However, we include
a detailed calculation
to verify the symplectic
condition to show how all
the formula fit together.}. We compute  $S^\dagger(-z)S(z)$ in two ways. By \eqref{eqn:S_asymp_exp}, as $z\to 0$ along $\mathfrak{R}$, we find 
\begin{equation*}
S^\dagger(-z)S(z) \ \sim\  e^{-\mathsf{u}^\dagger/z}
\mathsf{R}^\dagger(-z) \mathsf{R}(z)e^{\mathsf{u}/z}
=e^{-\mathsf{u}/z}\mathsf{R}^\dagger(-z) \mathsf{R}(z)e^{\mathsf{u}/z}\, .
\end{equation*} 
On the other hand, using the definition of $S$ in 
\eqref{eqn:sol} and the matrix $L$, we find
\begin{equation*}
    S(z)=(2\pi)^{|\cdot|}\Theta\mathsf{Y}_z\mathsf{G}_{\text{DT}z}A=(\Theta\mathsf{Y}_zL^{-1})(2\pi)^{|\cdot|}L\mathsf{G}_{\text{DT}z}A\, .
\end{equation*}
By direct calculation, detailed in Section \ref{subsec:scalar_prod}, we have 
\begin{equation}\label{eqn:Y_pairings}
\langle \Theta\mathsf{Y}^\lambda_z z^{|\lambda|}, \Theta\mathsf{Y}^\mu_z z^{|\mu|}\rangle_H=\eta\left(\Theta\mathsf{Y}^\lambda_z z^{|\lambda|}, (\Theta\mathsf{Y}^\mu_z z^{|\mu|})\Big|_{z\mapsto -z}\right) =\delta_{\lambda\mu}\prod_{\mathsf{w}: \text{ tangent weight at }\lambda} \mathsf{w}\, .
\end{equation}
Hence, $$(\Theta\mathsf{Y}_{z}L^{-1})^\dagger\Big|_{z\mapsto -z}(\Theta\mathsf{Y}_zL^{-1})=\Id\, .$$
Then,
\begin{equation*}
\begin{split}
S^\dagger(-z)S(z)&=((2\pi)^{|\cdot|}L\mathsf{G}_{\text{DT}z}A)^\dagger\Big|_{z\mapsto -z}(\Theta\mathsf{Y}_zL^{-1})^\dagger\Big|_{z\mapsto -z}(\Theta\mathsf{Y}_zL^{-1})((2\pi)^{|\cdot|}L\mathsf{G}_{\text{DT}z}A)\\
&=((2\pi)^{|\cdot|}L\mathsf{G}_{\text{DT}z}A)^\dagger\Big|_{z\mapsto -z}((2\pi)^{|\cdot|}L\mathsf{G}_{\text{DT}z}A)\, .
\end{split}
\end{equation*}

By the analysis of asymptotics of $\mathsf{G}_{\text{DT}}$ in the proof of Proposition \ref{prop:asymp_S},  we have the following as 
$z\to 0$ along $\mathfrak{R}$:
\begin{equation*}
\begin{split}
(2\pi)^{|\cdot|}L\mathsf{G}_{\text{DT}z}A\sim \text{Diag}\left(q^{-c(\lambda; t_1, t_2)/z}\prod_{\mathsf{w}: \text{ tangent weights at } \lambda} \exp\left(\sum_{m>0}\frac{B_{2m}}{2m(2m-1)}\left(\frac{-z}{\mathsf{w}}\right)^{2m-1}\right)\right),\\
((2\pi)^{|\cdot|}L\mathsf{G}_{\text{DT}z}A)^\dagger\Big|_{z\mapsto -z}\sim \text{Diag}\left(q^{c(\lambda; t_1, t_2)/z}\prod_{\mathsf{w}: \text{ tangent weights at } \lambda} \exp\left(\sum_{m>0}\frac{B_{2m}}{2m(2m-1)}\left(\frac{z}{\mathsf{w}}\right)^{2m-1}\right)\right)\, .
\end{split}    
\end{equation*}
We conclude $S^\dagger(-z)S(z)\sim \Id$. 

Comparing the two asymptotics of $S^\dagger(-z)S(z)$, we find $$\exp(-\mathsf{u}/z)\mathsf{R}^\dagger(-z)\mathsf{R}(z)\exp(\mathsf{u}/z)=\Id\, ,$$ which implies $\mathsf{R}^\dagger(-z)\mathsf{R}(z)=\Id$.
\end{proof}

\subsection{Calculations of scalar products}\label{subsec:scalar_prod}

We first check (\ref{eqn:jack_pairings}).
Write $$\mathsf{J}^\lambda=\sum_\epsilon \mathsf{J}^\lambda_\epsilon(t_1, t_2) |\epsilon\rangle\, ,\quad \mathsf{J}^\mu=\sum_{\epsilon'}\mathsf{J}^\mu_{\epsilon'}(t_1, t_2) |\epsilon\rangle\, ,$$
where $\mathsf{J}^\lambda_\epsilon(t_1, t_2), \mathsf{J}^\mu_{\epsilon'}(t_1, t_2)\in \mathbb{C}(t_1, t_2)[[q]]$.
Then 
\begin{equation}\label{eqn0}
    \begin{split}
    \<\mathsf{J}^\lambda, \mathsf{J}^\mu \>_H
    =&\sum_{\epsilon, \epsilon'}\mathsf{J}^\lambda_\epsilon(t_1, t_2)\overline{\mathsf{J}^\mu_{\epsilon'}(t_1,t_2)}\<\epsilon|\epsilon'\>_H\\
    =&\sum_{\epsilon, \epsilon'}\mathsf{J}^\lambda_\epsilon(t_1, t_2){\mathsf{J}^\mu_{\epsilon'}(-t_1,-t_2)}\frac{1}{(t_1t_2)^{\ell(\epsilon)}}\frac{\delta_{\epsilon\epsilon'}}{\mathfrak{z}(\epsilon)}\, .
    \end{split}
\end{equation}
Since $\mathsf{J}^\mu_{\epsilon'}(t_1,t_2)$ is $(t_1t_2)^{\ell(\epsilon')}$ times a polynomial in $t_1$ and $t_2$ of degree $|\mu|-\ell(\epsilon')$, we have $$\mathsf{J}^\mu_{\epsilon'}(-t_1,-t_2)=(-1)^{2\ell(\epsilon')}(-1)^{|\mu|-\ell(\epsilon')}\mathsf{J}^\mu_{\epsilon'}(t_1,t_2)\, .$$
We can therefore
write \eqref{eqn0} as
\begin{equation*}
    \begin{split}
    &\sum_{\epsilon, \epsilon'}\mathsf{J}^\lambda_\epsilon(t_1, t_2){\mathsf{J}^\mu_{\epsilon'}(t_1,t_2)}\frac{(-1)^{|\mu|-\ell(\epsilon')}}{(t_1t_2)^{\ell(\epsilon)}}\frac{\delta_{\epsilon\epsilon'}}{\mathfrak{z}(\epsilon)}\\
    =&\sum_{\epsilon, \epsilon'}\mathsf{J}^\lambda_\epsilon(t_1, t_2){\mathsf{J}^\mu_{\epsilon'}(t_1,t_2)}\frac{(-1)^{|\epsilon|-\ell(\epsilon)}}{(t_1t_2)^{\ell(\epsilon)}}\frac{\delta_{\epsilon\epsilon'}}{\mathfrak{z}(\epsilon)}\\
    =&\eta(\mathsf{J}^\lambda, \mathsf{J}^\mu)\, ,
    \end{split}
\end{equation*}
where, in the first equality, we have used $|\mu|=|\epsilon|$.

We now check (\ref{eqn:Y_pairings}). 
Write $$\mathsf{Y}^\lambda=\sum_\epsilon \mathsf{Y}^\lambda_\epsilon(t_1, t_2) |\epsilon\rangle\, ,\quad \mathsf{Y}^\mu=\sum_{\epsilon'}\mathsf{Y}^\mu_{\epsilon'}(t_1, t_2) |\epsilon'\rangle\, ,$$
where $\mathsf{Y}^\lambda_\epsilon(t_1, t_2), \mathsf{Y}^\mu_{\epsilon'}(t_1, t_2)\in \mathbb{C}(t_1, t_2)[[q]]$.
Then, 
\begin{equation*}
\begin{split}
\langle \Theta\mathsf{Y}^\lambda_z z^{|\lambda|}, \Theta\mathsf{Y}^\mu_z z^{|\mu|}\rangle_H
=&z^{|\lambda|+|\mu|}\sum_{\epsilon, \epsilon'} \mathsf{Y}^\lambda_\epsilon\left(\frac{t_1}{z}, \frac{t_2}{z}\right)\overline{\mathsf{Y}^\mu_{\epsilon'}\left(\frac{t_1}{z}, \frac{t_2}{z}\right)}\langle \Theta\epsilon|\Theta\epsilon'\rangle_H\\
=&z^{|\lambda|+|\mu|}\sum_{\epsilon, \epsilon'} \mathsf{Y}^\lambda_\epsilon\left(\frac{t_1}{z}, \frac{t_2}{z}\right){\mathsf{Y}^\mu_{\epsilon'}\left(-\frac{t_1}{z}, -\frac{t_2}{z}\right)}\langle \epsilon|\epsilon'\rangle_H z^{\ell(\epsilon)+\ell(\epsilon')}\\
=&z^{|\lambda|+|\mu|}\sum_{\epsilon, \epsilon'} \mathsf{Y}^\lambda_\epsilon\left(\frac{t_1}{z}, \frac{t_2}{z}\right){\mathsf{Y}^\mu_{\epsilon'}\left(-\frac{t_1}{z}, -\frac{t_2}{z}\right)} \frac{z^{\ell(\epsilon)+\ell(\epsilon')}}{(t_1t_2)^{\ell(\epsilon)}}\frac{\delta_{\epsilon\epsilon'}}{\mathfrak{z}(\epsilon)}\, .
\end{split}
\end{equation*}
We have 
\begin{equation}\label{eqn99}
    \begin{split}
    &\sum_{\epsilon, \epsilon'} \mathsf{Y}^\lambda_\epsilon\left(\frac{t_1}{z}, \frac{t_2}{z}\right){\mathsf{Y}^\mu_{\epsilon'}\left(-\frac{t_1}{z}, -\frac{t_2}{z}\right)} \frac{z^{\ell(\epsilon)+\ell(\epsilon')}}{(t_1t_2)^{\ell(\epsilon)}}\frac{\delta_{\epsilon\epsilon'}}{\mathfrak{z}(\epsilon)}\\
    =&\sum_{\epsilon, \epsilon'} \mathsf{Y}^\lambda_\epsilon\left(\frac{t_1}{z}, \frac{t_2}{z}\right){\mathsf{Y}^\mu_{\epsilon'}\left(-\frac{t_1}{z}, -\frac{t_2}{z}\right)} \frac{1}{(\frac{t_1}{z}\frac{t_2}{z})^{\ell(\epsilon)}}\frac{\delta_{\epsilon\epsilon'}}{\mathfrak{z}(\epsilon)}\\
    =&\langle \mathsf{Y}^\lambda, \mathsf{Y}^\mu\rangle_H\Big|_{t_i\mapsto t_i/z}\\
    =&\delta_{\lambda\mu}\prod_{\mathsf{w}: \text{ tangent weight at }\lambda} \mathsf{w}\Big|_{t_i\mapsto t_i/z}\\
    =&\delta_{\lambda\mu}\prod_{\mathsf{w}: \text{ tangent weight at }\lambda} \mathsf{w}/z^{2|\lambda|}\, .
    \end{split}
\end{equation}
In the second equality of
\eqref{eqn99}, 
we have used the definition of $\langle-,-\rangle_H$. In the third equality, we have used
\eqref{eqn:jack_lengths} and
\eqref{eqn:length_Y}.
Since $|\lambda|=|\mu|$, we have  $$\langle \Theta\mathsf{Y}^\lambda_z z^{|\lambda|}, \Theta\mathsf{Y}^\mu_z z^{|\mu|}\rangle_H=\delta_{\lambda\mu}\prod_{\mathsf{w}: \text{ tangent weight at }\lambda} \mathsf{w}\, .$$

On the other hand, we have
\begin{equation*}
    \begin{split}
    &\eta\left(\Theta\mathsf{Y}^\lambda_z z^{|\lambda|}, (\Theta\mathsf{Y}^\mu_z z^{|\mu|})\Big|_{z\mapsto -z}\right)\\
=&z^{|\lambda|+|\mu|}
(-1)^{|\mu|}
\sum_{\epsilon, \epsilon'} \mathsf{Y}^\lambda_\epsilon\left(\frac{t_1}{z}, \frac{t_2}{z}\right){\mathsf{Y}^\mu_{\epsilon'}\left(-\frac{t_1}{z}, -\frac{t_2}{z}\right)}\eta( \epsilon|\epsilon') z^{\ell(\epsilon)+\ell(\epsilon')}(-1)^{\ell(\epsilon')}\\
=&z^{|\lambda|+|\mu|}\sum_{\epsilon, \epsilon'} \mathsf{Y}^\lambda_\epsilon\left(\frac{t_1}{z}, \frac{t_2}{z}\right){\mathsf{Y}^\mu_{\epsilon'}\left(-\frac{t_1}{z}, -\frac{t_2}{z}\right)}
\frac{(-1)^{|\epsilon|-\ell(\epsilon)}}{(t_1t_2)^{\ell(\epsilon)}}\frac{\delta_{\epsilon\epsilon'}}{\mathfrak{z}(\epsilon)}
z^{\ell(\epsilon)+\ell(\epsilon')}(-1)^{\ell(\epsilon')}(-1)^{|\mu|}.
    \end{split}
\end{equation*}
Since the factors of $(-1)$ cancel, the above is
\begin{equation*}
    z^{|\lambda|+|\mu|}\sum_{\epsilon, \epsilon'} \mathsf{Y}^\lambda_\epsilon\left(\frac{t_1}{z}, \frac{t_2}{z}\right){\mathsf{Y}^\mu_{\epsilon'}\left(-\frac{t_1}{z}, -\frac{t_2}{z}\right)}
\frac{z^{\ell(\epsilon)+\ell(\epsilon')}}{(t_1t_2)^{\ell(\epsilon)}}\frac{\delta_{\epsilon\epsilon'}}{\mathfrak{z}(\epsilon)}
=\delta_{\lambda\mu}\prod_{\mathsf{w}: \text{ tangent weight at }\lambda} \mathsf{w}\, ,
\end{equation*}
    where we have used \eqref{eqn99}. We conclude 
$$\eta\left(\Theta\mathsf{Y}^\lambda_z z^{|\lambda|}, (\Theta\mathsf{Y}^\mu_z z^{|\mu|})\Big|_{z\mapsto -z}\right)=\delta_{\lambda\mu}\prod_{\mathsf{w}: \text{ tangent weight at }\lambda} \mathsf{w}\, .$$

\section{Analytic continuation}

\subsection{Asymptotics near $q=-1$}
We study the value of $S$ at $q=-1$, using the solution to the connection problem in \cite[Section 4]{op29}. Let $$\mathsf{H}^\lambda(q,t)$$ be the integral form of the Macdonald polynomial as in \cite[Equation (33)]{op29}. More precisely,
$$\mathsf{H}^\lambda(q,t)=t^{\mathsf{n}(\lambda)}\prod_{\square\in \lambda} (1-q^{a(\square)}t^{-l(\square)-1}){\Upsilon}\mathsf{P}^\lambda(q, t^{-1})\, ,$$
where $$\Upsilon |\mu\rangle =\prod_{i=1}^{\ell(\mu)} (1-t^{-\mu_i})^{-1} |\mu\rangle\, $$
and $$\mathsf{n}(\lambda)=\sum_{i=1}^{\ell(\lambda)} (i-1)\lambda_i\, .$$

Let $\mathsf{H}$ be the matrix with columns $\mathsf{H}^\lambda$ and the following identification of parameters: 
\begin{equation}\label{eqn:parameters}
(q, t)=(T_1, T_2), \quad T_i=\exp(2\pi\sqrt{-1}t_i)\, .
\end{equation}
Define the operators $\mathsf{G}_{\text{GW}}$ and $\mathbf{\Gamma}$ by 
\begin{eqnarray*}
\mathsf{G}_{\text{GW}}(t_1, t_2)|\mu\rangle&=&\prod_i g(\mu_i, t_1)g(\mu_i, t_2)|\mu \rangle\, ,\\ \mathbf{\Gamma} |\mu \rangle&=&\frac{(2\pi\sqrt{-1})^{\ell(\mu)}}{\prod_i \mu_i}\mathsf{G}_{\text{GW}}(t_1, t_2)|\mu\rangle\, ,
\end{eqnarray*}
where $g(x, t)=x^{tx}/\Gamma(tx)$.

By \cite[Theorem 4]{op29}, $\mathsf{Y}\mathsf{G}_{\text{DT}}|_{q=-1}=\frac{1}{(2\pi\sqrt{-1})^{|\cdot|}}\mathbf{\Gamma}\mathsf{H}$. Therefore, the solution $S$ in (\ref{eqn:sol}) satisfies 
\begin{equation}\label{eqn:sol_-1}
S|_{q=-1}=\frac{1}{\sqrt{-1}^{|\cdot|}}\Theta\mathbf{\Gamma}_z\mathsf{H}_zA\, .
\end{equation}
By Proposition \ref{prop:sol_w_asym2}, 
\begin{equation}\label{eqn:s_1_asymp}
S|_{q=-1}\sim \mathsf{R}\Big|_{q=-1}e^{\mathsf{u}/z}\Big|_{q=-1}
\end{equation}
as $z\to 0$ along $\mathfrak{R}$. We will determine $\mathsf{R}\Big|_{q=-1}$ by studying the asymptotics of the right side of (\ref{eqn:sol_-1}). 

As we will see, as $z\to 0$ along $\mathfrak{R}$, the right side of (\ref{eqn:sol_-1}) admits an asymptotical expansion of the following form 
\begin{equation}\label{eqn:s_1_asymp2}
\frac{1}{\sqrt{-1}^{|\cdot|}}\Theta\mathbf{\Gamma}_z\mathsf{H}_zA\sim \mathsf{Z}^+\mathsf{Z}^-,
\end{equation}
where $\mathsf{Z}^+=\Id+O(z)$ is a $z$-series and $\mathsf{Z}^-=\Id+O(1/z)$ is a $1/z$-series. 

We write $A=A_0A_1$, where
\begin{center}
\begin{tabular}[c]{|l|r|}
\hline
{Matrix} & {Eigenvalues} \\
\hline
$A_0$ & $\prod_{\mathsf{w}: \text{ tangent weights at $\lambda$}}z^{-\mathsf{w}/z}e^{-\mathsf{w}/z}$\\
\hline
$A_1$ & $\prod_{\mathsf{w}: \text{ tangent weights at $\lambda$}}\mathsf{w}^{\mathsf{w}/z}$\\
\hline
\end{tabular}
\end{center}
The operator $A_0$ is a scalar multiple of the identity matrix 
since
$$\sum_{\mathsf{w}: \text{ tangent weights}}\mathsf{w}=|\lambda|(t_1+t_2)\, .$$ 
As a result,
\begin{equation}\label{eqn:sol_-12}
S|_{q=-1}=\frac{1}{\sqrt{-1}^{|\cdot|}}\Theta\mathbf{\Gamma}_zA_0\mathsf{H}_zA_1\, .
\end{equation}
Since we have
$$\prod_{\mathsf{w}: \text{ tangent weights at $\lambda$}}\mathsf{w}^{\mathsf{w}/z}=\exp\left(\frac{1}{z}\sum_{\mathsf{w}: \text{ tangent weights at $\lambda$}}\mathsf{w}\log \mathsf{w}\right)\,, $$ $A_1$ contributes to $\mathsf{Z}^-$.

By \cite[Chapter VI, equation (8.19)]{mac}, matrix coefficients of $\mathsf{H}$ are polynomials in $T_1, T_2$. As mentioned in \cite{op29}, our $\mathsf{H}^\lambda$  is the same as $\widetilde{H}_\lambda$ in \cite[Definition 3.5.2]{haiman_cdm}. To see the equivalence, the first
step is 
$$\widetilde{H}_\lambda(z;q,t)=t^{\mathsf{n}(\lambda)}J_\lambda[Z/(1-t^{-1}); q, t^{-1}]\, ,$$
by the equation just below Theorem/Definition 6.1 of \cite{haiman_sym2001}. Here, $Z/(1-t^{-1})$ stands for the plethystic substitution defined in \cite[Section 3.3]{haiman_cdm}.  The function $J_\lambda(z;q, t)$, defined by equation (54) of \cite{haiman_cdm}, is the same\footnote{Note there is a typo in equation (54) of \cite{haiman_cdm}: the plethystic substitution should be $(1-t)Z$ instead of $(1-t^{-1})Z$.} as that defined in \cite[Chapter VI, Section 8, (8.3)]{mac}, as remarked just above Section 3.5.2 of \cite{haiman_cdm}. By \cite[Chapter VI, Section 8, (8.3)]{mac}, $$J_\lambda(z;q,t)=\prod_{\square\in \lambda}(1-q^{a(\square)}t^{l(\square)+1})\mathsf{P}^\lambda(z;q,t)\, .$$
Working through the definition given in  \cite[Section 3.3]{haiman_cdm}, we find that the plethystic substitution $$Z\mapsto Z/(1-t^{-1})$$ is equivalent to the map $|\mu\rangle \mapsto \frac{1}{\prod_i (1-t^{-\mu_i})}|\mu\rangle$. Thus $$\mathsf{H}^\lambda(q,t)=\widetilde{H}_\lambda(z;q,t).$$

By the identification of parameters (\ref{eqn:parameters}) and condition (\ref{real_condition_z}), we see that as $z\to 0$ along $\mathfrak{R}$, we have $q\to 0, t^{-1}\to 0$.
By \cite[Definition 3.5.3]{haiman_cdm}, we have 
$$\mathsf{H}^\lambda(q,t)=\sum_\mu \widetilde{K}_{\mu\lambda}(q,t)s_\mu,$$
where 
$\widetilde{K}_{\mu\lambda}(q,t)$ are the  {\em Kostka-Macdonald polynomials} (or $q,t$-Kostka coefficients) and
$$s_\mu=\sum_\nu\chi_\mu(\nu)|\nu\rangle$$
is the Schur function. 

By the discussion below Definition 3.5.3 of \cite{haiman_cdm}, we can define $K_{\mu\lambda}(q,t)$ by $$J_\lambda(z;q,t)=\sum_\mu K_{\mu\lambda}(q,t)s_\mu[Z/(1-t)].$$
As noted in the discussion below Definition 3.5.3 of \cite{haiman_cdm}, 
\begin{equation}\label{eqn:kostka}
\widetilde{K}_{\mu\lambda}(q,t)=t^{\mathsf{n}(\lambda)}K_{\mu\lambda}(q, t^{-1}).
\end{equation}
Therefore, we can write
\begin{equation*}
    \begin{split}
     \sum_\mu \widetilde{K}_{\mu\lambda}(q,t)s_\mu
   =&\sum_\mu \widetilde{K}_{\mu\lambda}(q,t)\sum_\nu \chi_\mu(\nu)|\nu\rangle\\
   =&\sum_\nu \left(\sum_\mu \widetilde{K}_{\mu\lambda}(q,t)\chi_\mu(\nu) \right)|\nu\rangle,
    \end{split}
\end{equation*}
and 
\begin{equation*}
\begin{split}
    \sum_\mu \widetilde{K}_{\mu\lambda}(q,t)\chi_\mu(\nu)
    =\sum_\mu t^{\mathsf{n}(\lambda)}K_{\mu\lambda}(q, t^{-1})\chi_\mu(\nu).
\end{split}
\end{equation*}

As $z\to 0$ along $\mathfrak{R}$, we have $q, t^{-1}\to 0$. By equation (59) of \cite{haiman_cdm}, $$\widetilde{K}_{\mu\lambda}(0, t)=\widetilde{K}_{\mu\lambda}(t),$$
where $\widetilde{K}_{\mu\lambda}(t)$ is the {\em cocharge Kostka-Foulkes polynomial} given in \cite[Definition 3.4.13]{haiman_cdm} in terms of the Kostka-Foulkes polynomials $K_{\mu\lambda}(t)$,
$$\widetilde{K}_{\mu\lambda}(t)=t^{\mathsf{n}(\lambda)}K_{\mu\lambda}(t^{-1}).$$
Comparing this with (\ref{eqn:kostka}), we see that $$K_{\mu\lambda}(0,t)=K_{\mu\lambda}(t).$$
By \cite[Corollary 3.4.12 (vi)]{haiman_cdm}, $$K_{\mu\lambda}(0)=\delta_{\mu\lambda}.$$
Therefore $$\lim_{z\to 0 \text{ along }\mathfrak{R}} K_{\mu\lambda}(q, t^{-1})=\delta_{\mu\lambda}.$$
Thus as $z\to 0$ along $\mathfrak{R}$, we have $\mathsf{X}^{-1}\mathsf{H}_z\exp(-2\pi\sqrt{-1}t_2\mathsf{n}/z)$ tends to $\Id$, where $\mathsf{X}$ is the matrix with entries $\chi_\lambda(\mu)$ and $\mathsf{n}$ is the diagonal matrix with diagonal entries $\mathsf{n}(\lambda)$. Hence,
\begin{equation}\label{eqn:asym_H}
\mathsf{H}_z\sim \mathsf{X}\exp(2\pi\sqrt{-1}t_2\mathsf{n}/z),
\end{equation}
as $z\to 0$ along $\mathfrak{R}$. So $\mathsf{H}_z$ contributes $\mathsf{X}$ to $\mathsf{Z}^+$ and $\exp(2\pi\sqrt{-1}t_2\mathsf{n}/z)$ to $\mathsf{Z}^-$.

It remains to study the asymptotics of $\frac{1}{\sqrt{-1}^{|\cdot|}}\Theta\mathbf{\Gamma}_zA_0$. By condition (\ref{t_condition_z}), the Stirling asymptotics (\ref{eqn:stirling}) is applicable to $\mathbf{\Gamma}_z$ as $z\to 0$ along $\mathfrak{R}$. We find $\mathbf{\Gamma}_z$ has the asymptotics a diagonal matrix with entries:
\begin{multline*}
\sqrt{-1}^{\ell(\mu)}({t_1t_2})^{\ell(\mu)/2}{z}^{-\ell(\mu)}e^{\frac{|\mu|(t_1+t_2)}{z}}\exp\left(-|\mu|\frac{t_1\log t_1+t_2\log t_2}{z} + |\mu|(t_1+t_2)\frac{\log z}{z}\right)\\
\times \prod_{i=1}^{\ell(\mu)} \exp\left(\sum_{m>0}\frac{B_{2m}}{2m(2m-1)}\left(\left(\frac{-z}{\mu_it_1}\right)^{2m-1}+\left(\frac{-z}{\mu_it_2}\right)^{2m-1}\right) \right)\, .
\end{multline*}
There are now several cancellations:
\begin{enumerate}
\item[$\bullet$]
 The term $e^{\frac{|\mu|(t_1+t_2)}{z}}$ cancels with $\prod_{\mathsf{w}: \text{ tangent weights at $\mu$}}e^{-\mathsf{w}/z}$ in $A_0$. 
 \item[$\bullet$]
 The term $\exp(|\mu|(t_1+t_2)\frac{\log z}{z})$ cancels with $\prod_{\mathsf{w}: \text{ tangent weights t $\mu$}}z^{-\mathsf{w}/z}$ in $A_0$. 
 \item[$\bullet$]
 The factor $z^{-\ell(\mu)}$ cancels with $\Theta$. 
 \end{enumerate}
 Also, the term $\exp(-|\mu|\frac{t_1\log t_1+t_2\log t_2}{z})$  contributes to $\mathsf{Z}^-$.
 
 Therefore, the part of $\mathsf{Z}^+$ coming from  $\frac{1}{\sqrt{-1}^{|\cdot|}}\Theta\mathbf{\Gamma}_zA_0$ is the diagonal matrix with entries 
 \begin{equation}\label{eqn:asym_coef2}
 \sqrt{-1}^{\ell(\mu)-|\mu|}(t_1t_2)^{\ell(\mu)/2}
 \prod_{i=1}^{\ell(\mu)}
 \exp\left(\sum_{m>0}\frac{B_{2m}}{2m(2m-1)}\left(\left(\frac{-z}{\mu_it_1}\right)^{2m-1}+\left(\frac{-z}{\mu_it_2}\right)^{2m-1}\right) \right).
 \end{equation}
 
 Comparing\footnote{What we need here is that if a function of the form $\exp(A/z)$, with $A$ a diagonal matrix independent of $z$, has an asymptotical expansion into a $z$-series starting with $\Id$, then $A=0$. The result follows by computing the asymptotical coefficients using their definitions.} the two asymptotical expansions of $S\Big|_{q=-1}$ in (\ref{eqn:s_1_asymp}) and (\ref{eqn:s_1_asymp2}), we find  $$\mathsf{R}\Big|_{q=-1}=\mathsf{Z}^+, \quad e^{\mathsf{u}/z}\Big|_{q=-1}=\mathsf{Z}^-.$$
Hence, $\mathsf{u}\Big|_{q=-1}$ is the diagonal matrix with diagonal entries $$-|\lambda|(t_1\log t_1+t_2\log t_2)+\sum_{\mathsf{w}: \text{ tangent weights at }\lambda} \mathsf{w}\log \mathsf{w}+2\pi\sqrt{-1}t_2\mathsf{n}(\lambda).$$ 

By construction, the columns of $\mathsf{R}\Big|_{q=-1, z=0}$ are normalized idempotents of the ring $$H_\mathsf{T}^*(I\text{Sym}^n(\mathbb{C}^2))$$ written in the basis $\{|\mu\rangle\, \mu\in \text{Part}(n)\}$ of $\mathsf{V}$. Since $\mathsf{R}\Big|_{q=-1}=\mathsf{Z}^+$, by looking at $\mathsf{Z}^+\Big|_{z=0}$, we find the idempotent appearing on the column indexed by $\lambda$ is $\mathsf{I}^\lambda$
written in the basis $\{|\mu\rangle\, \mu\in \text{Part}(n)\}$ of $\mathsf{V}$. Moreover, the action of $\mathsf{R}\Big|_{q=-1}$ on $\mathsf{I}^\lambda$ is given by the column of $\mathsf{Z}^+$ indexed by $\lambda$, which is computed by combining (\ref{eqn:asym_H}) and (\ref{eqn:asym_coef2}). More precisely,
  \begin{multline}\label{eqn:column1}
\mathsf{R}\Big|_{q=-1}(\mathsf{I}^\lambda)=\sum_{\mu}\chi_\lambda(\mu) \exp\left(-\sum_{m>0}\frac{B_{2m}}{2m(2m-1)} \sum_i\left(\frac{1}{(\mu_it_1)^{2m-1}}+\frac{1}{(\mu_it_2)^{2m-1}}\right)z^{2m-1} \right)\\
\cdot \sqrt{-1}^{\ell(\mu)-|\mu|} (t_1t_2)^{\ell(\mu)/2}\,
\Big|\mu\Big\rangle
\, .
\end{multline}

Comparing (\ref{eqn:R_sym0}) with (\ref{eqn:column1}), we find $$\mathsf{R}^\text{Sym}\Big|_{u=0}= \mathsf{R}\Big|_{q=-1}\, ,$$ after $-q=e^{iu}$.
As
a consequence, in a neighborhood of $q=-1$, $\mathsf{R}^{\text{Sym}}=\mathsf{R}$.

We have proven the following result
parallel to Proposition
\ref{prop:asymp_S}.

\begin{prop}\label{prop:asymp_S2}
As $z\to 0$ along $\mathfrak{R}$, the operator $Se^{\mathsf{-u}/z}|_{q=-1}$ has the asymptotics $$\mathsf{R}^{\text{\em Sym}}\Big|_{u=0}.$$
\end{prop}

\subsection{Proof of Theorem \ref{crc}}
\label{thcrc}
We already have a match of the
complete genus 0 theory for
$\Hilb$ and $\Sym(\com^2)$
under the identification
\begin{equation}\label{p345}
\mathsf{V} \rightarrow \widetilde{\mathsf{V}}\,, \ \ \ \ \ \ \
|\mu\rangle \mapsto |\widetilde{\mu}\rangle\, .
\end{equation}
To prove Theorem \ref{crc},
we need only match the
$\mathsf{R}$-matrices $\mathsf{R}^{\mathsf{Hilb}}$
and $\mathsf{R}^{\mathsf{Sym}}$ via
\eqref{p345}  and the
variable change
\begin{equation} \label{ww99}
-q=e^{iu}\, .
\end{equation}
The coordinate $\widetilde{t}$
along $|2,1^{n-2}\rangle\in \widetilde{\mathsf{V}}$
is related to the
coordinate
$t$ along $|2,1^{n-2}\rangle
\in \mathsf{V}$ via \eqref{p345} by
$$\tilde{t}= (-i)^{-1}t= it\,.$$
By the chain rule,
$$-\frac{\partial}{\partial t}
= -i\frac{\partial}{\partial \widetilde{t}} \ \ \ \
\text{and} \ \ \ \
q\frac{\partial}{\partial q}
=-i \frac{\partial}{\partial u}\, .
$$
The differential equations
\begin{equation*}
-\frac{\partial}{\partial t} \mathsf{R}^{\mathsf{Hilb}} = q\frac{\partial}{\partial q} \mathsf{R}^{\mathsf{Hilb}} \ \ \  \
\text{and} \ \ \ \
\frac{\partial}{\partial \widetilde{t}}\, \mathsf{R}^{\mathsf{Sym}} = \frac{\partial}{\partial u} \mathsf{R}^{\mathsf{Sym}}\, 
\end{equation*}
therefore exactly match via \eqref{p345}
and the variable change \eqref{ww99}.
Hence, by Proposition \ref{vv55}, we need only match
$$
\mathsf{R}^{\mathsf{Hilb}}|_{q=-1}
=\left[\mathsf{R}^{\mathsf{Hilb}}|_{-q=e^{iu}}\right]_{u=0}$$
with $\mathsf{R}^{\mathsf{Sym}}|_{u=0}$.

The matching $\mathsf{R}^{\mathsf{Hilb}}|_{q=-1}
=\mathsf{R}^{\mathsf{Sym}}|_{u=0}$
is a non-trivial assertion.
The difficulty can be summarized
as follows. While we have closed forms for
$$\mathsf{R}^{\mathsf{Hilb}}|_{q=0}
\ \ \ \ \text{and} \ \ \ \
\mathsf{R}^{\mathsf{Sym}}|_{u=0}\, ,$$
by Propositions \ref{prophilb} and 
\ref{propsym} respectively, we must 
control the $q=-1$ evaluation
of $\mathsf{R}^{\mathsf{Hilb}}$
which is far away from $q=0$.

The issue is
resolved  by the analytic
continuation of the solution
to the QDE of $\Hilb$
computed in \cite{op29}.
The study of the QDE of $\Hilb$
in \cite{op29} concerns only
the {\em small} quantum cohomology
(all the coordinates of
$\mathsf{V}$ are set to 0).
The results of Proposition
\ref{prop:asymp_S} and Proposition \ref{prop:asymp_S2}
show the $\mathsf{R}$-matrix associated
to  the solution $S$ of the
QDE has asymptotics
$$\mathsf{R}^\mathsf{Hilb}\Big|_{q=0}= \mathsf{R}\Big|_{q=0}\ \ \ \ \text{and}\ \ \ \
\mathsf{R}^\mathsf{Sym}\Big|_{u=0}= \mathsf{R}\Big|_{q=-1}.$$
Since $\mathsf{R}$ matches
$\mathsf{R}^{\mathsf{Hilb}}$
for small $q$ 
by Proposition \ref{prop:sol_w_asym2}
and 
is analytic along
the path $\gamma$ connecting
$0$ to $-1$, we conclude
\begin{equation}\label{ggeedd}
\mathsf{R}^{\mathsf{Hilb}}|_{q=-1}
=\mathsf{R}^{\mathsf{Sym}}|_{u=0}
\end{equation}
at least when all the
coordinates of $\mathsf{V}$ are
set to 0. By Proposition 
\ref{prop:R_mat_sym}, any
difference between two operators
\eqref{ggeedd} persists after
setting the coordinates of
$\mathsf{V}$ to 0. Hence, the
equality \eqref{ggeedd} 
is valid when the dependence
on $\mathsf{V}$ is included. \qed

\vspace{8pt}
The proof of Theorem \ref{crc} not only yields
the series result  
$$\blang \mu^1, \mu^2, \ldots, \mu^r \brang_{g}^{\Hilb} =
(-i)^{\sum_{i=1}^r \ell(\mu^i)-|\mu^i|}\blang \mu^1, \mu^2, \ldots, \mu^r \brang_{g}^{{\Sym}(\mathbb{C}^2)}\ \ \ \ \ \text{after} \ \ \ -q=e^{iu}\, $$
but matches the full CohFTs 
$$\Omega_{g,r} = \widetilde{\Omega}_{g,r}
\ \ \ \ \ \text{after} \ \ \ \ \ \mathsf{V} \rightarrow 
\widetilde{\mathsf{V}} \ \ \ \text{and} \ \ \ -q=e^{iu}\,. $$

\subsection{Proof of Theorem \ref{tt22}}
We obtain a tetrahedron of 
equivalences of CohFTs:
\vspace{-15pt}
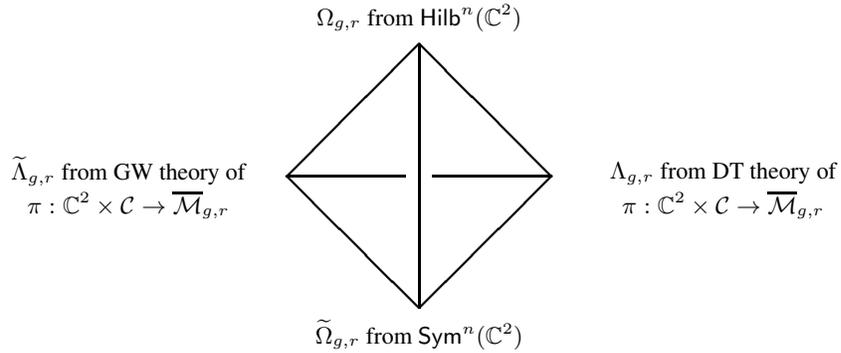
\begin{center}
\scriptsize
\begin{picture}(200,175)(-30,-50)
\thicklines
\put(25,25){\line(1,1){50}}
\put(25,25){\line(1,-1){50}}
\put(125,25){\line(-1,1){50}}
\put(125,25){\line(-1,-1){50}}
\put(75,-25){\line(0,1){100}}
\put(25,25){\line(1,0){45}}
\put(80,25){\line(1,0){45}}
\put(75,95){\makebox(0,0){ }}
\put(75,85){\makebox(0,0){$\Omega_{g,r}$ from $\Hilb$}}
\put(75,-35){\makebox(0,0){$\widetilde{\Omega}_{g,r}$
from $\Sym(\com^2)$}}
\put(75,-45){\makebox(0,0){}}
\put(190,26){\makebox(0,0){${\Lambda}_{g,r}$ from DT theory of}}
\put(190,14){\makebox(0,0){$\pi:\mathbb{C}^2 \times 
\mathcal{C} \rightarrow
\overline{\mathcal{M}}_{g,r}
$}}
\put(-35,27){\makebox(0,0){$\widetilde{\Lambda}_{g,r}$ from 
GW theory of}}
\put(-35,14){\makebox(0,0){$\pi:\mathbb{C}^2 \times 
\mathcal{C} \rightarrow
\overline{\mathcal{M}}_{g,r}
$}}
\end{picture}
\end{center}
\normalsize
\vspace{5pt}

\noindent Both the CohFTs $\Omega$ and $\Lambda$ are
based on 
$$\mathsf{A}=\mathbb{Q}(t_1,t_2)[[q]]$$
and
$(\mathsf{V},\eta)$. The exact matching
$$\Omega=\Lambda$$
is Proposition \ref{pp15}. Similarly,
$\widetilde{\Omega}$ and $\widetilde{\Lambda}$ are
based on 
$$\widetilde{\mathsf{A}}=\mathbb{Q}(t_1,t_2)[[u]]$$
and
$(\widetilde{\mathsf{V}},\widetilde{\eta})$. The exact matching
$$\widetilde{\Omega}=\widetilde{\Lambda}$$
is Proposition \ref{pp12}. Finally, we have the 
crepant resolution matching
$$\Omega_{g,r} = \widetilde{\Omega}_{g,r}
\ \ \ \ \ \text{after} \ \ \ \ \ \mathsf{V} \rightarrow 
\widetilde{\mathsf{V}} \ \ \ \text{and} \ \ \ -q=e^{iu}\,$$
from Section \ref{thcrc}. As a result, we also 
obtain the GW/DT matching
$$\widetilde{\Lambda}_{g,r} = {\Lambda}_{g,r}
\ \ \ \ \ \text{after} \ \ \ \ \ \widetilde{\mathsf{V}} \leftarrow 
{\mathsf{V}} \ \ \ \text{and} \ \ \ -q=e^{iu}\,. $$
The proof of Theorem \ref{tt22} is complete. \qed

\end{document}